\documentclass[11pt,fullpage]{article}
\pdfoutput = 1

\usepackage[english]{babel} 
\usepackage{amsmath,amsfonts,amssymb,amsthm,graphics,epsfig,psfrag,epstopdf}
\usepackage[active]{srcltx}
\usepackage{paralist,subfigure}
\usepackage{color,soul}
\setstcolor{red}

\newtheorem{theorem}{Theorem}[section]
\newtheorem{corollary}[theorem]{Corollary}
\newtheorem{lemma}[theorem]{Lemma}
\newtheorem{proposition}[theorem]{Proposition}
\newtheorem{definition}[theorem]{Definition}
\newtheorem{remark}[theorem]{Remark}

\def\N{\mathbb{N}}
\def\Z{\mathbb{Z}}

\def\R{\mathbb{R}}
\def\epsilon{\varepsilon}
\let\e=\varepsilon
\let\vp=\varphi

\let\t=\tilde
\let\ol=\overline
\let\ul=\underline
\let\.=\cdot
\let\0=\emptyset
\let\mc=\mathcal

\def\eq#1{{\rm(\ref{eq:#1})}}
\def\thm#1{Theorem~\ref{thm:#1}}
\def\seq#1{(#1_n)_{n\in\N}}
\def\limn{\lim_{n\to+\infty}}

\def\MP{maximum principle}

\def\tilde{\widetilde}
\def\1{\text{1}}

\newenvironment{formula}[1]{\begin{equation}\label{eq:#1}}{\end{equation}\noindent}
\def\Fi#1{\begin{formula}{#1}}
\def\Ff{\end{formula}\noindent}

\newcommand{\SE}{\setcounter{equation}{0} \section}
\newcommand{\be}{\begin{equation}}
\newcommand{\ee}{\end{equation}}
\newcommand{\baa}{\begin{array}}
\newcommand{\eaa}{\end{array}}
\newcommand{\ba}{\begin{eqnarray}}
\newcommand{\ea}{\end{eqnarray}}

\begin{document}
\date{}
\title{\bf{Admissible speeds of transition fronts for non-autonomous monostable
equations}}
\author{Fran\c cois Hamel$^{\hbox{\small{ a}}}$ and Luca
Rossi$^{\hbox{\small{ b }}}$\thanks{This work has been carried out in the framework of the
Labex Archim\`ede (ANR-11-LABX-0033) and of the A*MIDEX project
(ANR-11-IDEX-0001-02), funded by the ``Investissements d'Avenir" French Government
program managed by the French National Research Agency (ANR). The research leading to
these results has also received funding from the European Research Council under the
European Union's Seventh Framework Programme (FP/2007-2013) / ERC Grant Agreement
n.321186~- ReaDi~- Reaction-Diffusion Equations, Propagation and Modelling, and from Italian
GNAMPA-INdAM. Part of this work was carried out during visits by F.~Hamel to the
Universit\`a di Padova, whose hospitality is thankfully acknowledged.}\\
\\
\footnotesize{$^{\hbox{a }}$Aix Marseille Universit\'e, CNRS, Centrale Marseille}\\
\footnotesize{Institut de Math\'ematiques de Marseille, UMR 7373, 13453 Marseille, France}\\
\footnotesize{$^{\hbox{b }}$Dipartimento di Matematica, Universit\`a di Padova,
via Trieste 63, 35121 Padova, Italy}}
\maketitle

\begin{abstract}
We consider a reaction-diffusion equation with a nonlinear term of the Fisher-KPP type,
depending on time $t$ and admitting two limits as $t\to\pm\infty$. We derive the set of
admissible asymptotic past and future speeds of transition fronts for such equation. We further
show that any transition front which is non-critical as $t\to-\infty$ always admits two
asymptotic past and future speeds. We finally describe the asymptotic profiles of the
non-critical fronts as $t\to\pm\infty$.
\end{abstract}


\section{Introduction and main results}\label{intro}

This paper is concerned with asymptotic dynamical properties of front-like solutions
for time-dependent reaction-diffusion equations of the type
\Fi{P=f}
u_t=u_{xx}+f(t,u),\quad t\in\R,\ x\in\R.
\Ff
We focus here on the case where the reaction term $f$ admits some limits as
$t\to\pm\infty$. These limits are in general different and the medium is thus in
general not uniquely ergodic. Actually, even if the limits of $f(t,u)$ as
$t\to\pm\infty$ are equal, the medium is in general truly time-dependent and is
not periodic, almost-periodic or even recurrent. We prove the existence of
solutions which move with some~--~in general different~--~speeds as
$t\to\pm\infty$ and we also characterize the set of all admissible asymptotic
speeds as~$t\to\pm\infty$ among all time-global front-like solutions.

Throughout the paper, the reaction term $f:\R\times[0,1]\to\R$ is assumed to be
uniformly H\"older
continuous, of class $C^1$ and such that $\partial_uf:=\frac{\partial f}{\partial u}$ is
bounded in $\R\times[0,1]$. We will further require that
\be\label{eq:f/u}\left\{\baa{l}
f(t,0)=f(t,1)=0\ \hbox{ for all }t\in\R,\vspace{3pt}\\
f(t,u)\ge 0\ \hbox{ for all }(t,u)\in\R\times[0,1],\vspace{3pt}\\
\displaystyle\frac{f(t,u)}{u}\hbox{ is nonincreasing with respect to }u\in(0,1],
\ \hbox{ for all }t\in\R,\eaa\right.
\ee
and that there are two $C^1$ functions $f_{\pm}:[0,1]\to\R$ such that
\be\label{fpm}\left\{\baa{l}
f_{\pm}(0)=f_{\pm}(1)=0,\ \ f_{\pm}(u)>0\ \hbox{ for all }u\in(0,1),\vspace{3pt}\\
\displaystyle\frac{f(t,u)}{f_{\pm}(u)}\to1\hbox{ as }t\to\pm\infty,\
\hbox{ uniformly for }u\in(0,1).\eaa\right.
\ee
Notice that the hypotheses imply in particular that the functions $u\mapsto f_{\pm}(u)/u$
are nonincreasing with respect to $u\in(0,1]$ and that $f'_{\pm}(0)>0$. We denote
\be\label{defmupm}
\mu_{\pm}:=f'_{\pm}(0)>0\ \hbox{ and }\ \ul\mu:=\min(\mu_-,\mu_+)>0.
\ee
In order to derive the existence result, we will also need the following regularity property:
\be\label{C1omega}
f(t,u)\geq \partial_u f(t,0)u-Cu^{1+\omega}\ \text{ for all }(t,u)\in\R\times(0,\delta),
\ee
for some $C>0$ and $\delta,\omega\in(0,1]$.

Such nonlinearities $f(t,u)$ are said to be of the Fisher or KPP type, for Kolmogorov,
Petrovski and Piskunov, by analogy with the time-independent case $f(u)$ which was first
considered in~\cite{fi,kpp}. It follows from assumption~(\ref{fpm}) that $f(t,\cdot)>0$ in $(0,1)$
for all $|t|$ large enough. However, no strict sign assumption is made on $f$ for intermediate
times $t$. In particular, the functions~$f(t,\cdot)$ may well be identically equal to $0$ on
$[0,1]$ for some times $t$ belonging to a non-empty bounded set.

A typical example of a function $f(t,u)$ satisfying all above assumptions is
$f(t,u)=\tilde{\mu}(t)\,\tilde{f}(u)$
where~$\tilde{f}:[0,1]\to\R$ is of class $C^{1,\omega}$, $\tilde{f}(0)=\tilde{f}(1)=0$,
$\tilde{f}>0$ on $(0,1)$, $u\mapsto\tilde{f}(u)/u$ is nonincreasing with respect to $u\in(0,1]$,
and the function $\tilde{\mu}:\R\to[0,+\infty)$ is of class $C^1$ and admits some limits
$\tilde{\mu}_{\pm}=\lim_{t\to\pm\infty}\tilde{\mu}(t)$ in $(0,+\infty)$. In this case,
$f_{\pm}(u)=\tilde{\mu}_{\pm}\tilde{f}(u)$ for all~$u\in[0,1]$ and
$\mu_{\pm}=\tilde{\mu}_{\pm}\tilde{f}'(0)$.


\subsection{Notions of transition fronts and asymptotic mean speeds}

Equations of the type~\eq{P=f} are known to be good models to describe the propagation of
fronts connecting the steady states $0$ and $1$, see e.g.~\cite{f,mu,sk}. The solution $u$
typically stands for the density of a species invading an open space and the fronts are known
to play a fundamental role in the description of the dynamical properties of the solutions
of~\eq{P=f}. We will recall a bit later some of the main results about the existence and
dynamical properties of known front-like solutions of particular equations of the type~\eq{P=f}.

From a mathematical point of view, for problem~\eq{P=f}, using the same terminology as
in~\cite{bh1,bh2}, the front-like solutions connecting $0$ (say, on the right) and $1$ (on the
left) are called transition fronts and they are defined as follows:

\begin{definition}\label{def1} For problem~\eq{P=f}, a transition front connecting $0$ and~$1$
is a time-global classical solution~$u:\R\times\R\to[0,1]$ for which there exists a
function~$X:\R\to\R$ such that
\be\label{gtf}\left\{\baa{ll}
u(t,X(t)+x)\to 1 & \hbox{as }x\to-\infty\vspace{3pt}\\
u(t,X(t)+x)\to 0 & \hbox{as }x\to+\infty\eaa\right.\hbox{ uniformly in }t\in\R.
\ee
\end{definition}

Several comments on this definition are in order. First of all, it is actually a particular
case of a more general definition given in~\cite{bh1,bh2} in a broader framework. For
the one-dimensional equation~\eq{P=f}, the transition fronts connec\-ting~$0$ and
$1$ correspond to the ``wave-like" solutions defined in~\cite{s3,s4} (see also~\cite{m}
for a different notion involving the continuity with respect to the environment around
the front position). Roughly speaking, condition~(\ref{gtf}) means that the
diameter of the transition zone between the sets where~$u\simeq 1$ and
$u\simeq 0$ is uniformly bounded in time: the fundamental property in~(\ref{gtf}) is
the uniformity of the limits with respect to time~$t\in\R$. The limits~(\ref{gtf}) imply in
particular that, given any real numbers~$a$ and~$b$ such that~$0<a\le b<1$, there is a
cons\-tant~$C=C(u,a,b)\ge0$ such that, for every $t\in\R$, $\big\{x\in\R;\ a\le u(t,x)\le b\big\}
\subset\big[X(t)-C,X(t)+C\big]$. Standard parabolic estimates and the strong maximum
principle also easily imply that, for any transition front $u$ connecting $0$ and $1$, and for
every $C\ge0$,
\be\label{infsup}
0<\inf_{t\in\R,\,x\in[X(t)-C,X(t)+C]}u(t,x)\le\sup_{t\in\R,\,x\in[X(t)-C,X(t)+C]}u(t,x)<1.
\ee
Notice furthermore that, for a given transition front $u$ connecting $0$ and $1$, the family
$(X(t))_{t\in\R}$ is not uniquely defined since, for any bounded function $\xi:\R\to\R$,
the family $(X(t)+\xi(t))_{t\in\R}$ satisfies~(\ref{gtf}) if~$(X(t))_{t\in\R}$ does. Hence, one
can choose for instance as $X(t)$ one point $x$ such that~$u(t,x)=1/2$.
On the other hand, if $(X(t))_{t\in\R}$ and $(\tilde{X}(t))_{t\in\R}$ are
associated to a given transition front~$u$ connecting~$0$ and $1$ in the sense
of~(\ref{gtf}), then it can be immediately seen that
\be\label{tildeX}
\sup_{t\in\R}\big|X(t)-\tilde{X}(t)\big|<+\infty.
\ee
Lastly, it is shown in Proposition 4.1 of~\cite{hr}, in the case of general space-time
dependent reaction-diffusion equations, that for any transition front $u$ connecting $0$
and $1$, any function $X$ such that~\eqref{gtf} holds has uniformly bounded
local oscillations,
that is,
\be\label{Xtau}
\forall\,\tau\ge0,\ \ \sup_{(t,s)\in\R^2,\,|t-s|\le\tau}|X(t)-X(s)|<+\infty.
\ee

When the function $f=f(u)$ does not depend on the time variable, the most typical examples of
transition fronts connecting $0$ and $1$ are the standard traveling fronts $u(t,x)=\phi(x-ct)$
with~$0=\phi(+\infty)<\phi(\xi)<\phi(-\infty)=1$ for all $\xi\in\R$. Under the Fisher-KPP
hypothesis, such traveling fronts exist if and only if $c\ge2\sqrt{f'(0)}$ and, for each
$c\ge2\sqrt{f'(0)}$, the function $\phi=\phi_c$ is decreasing and unique up to
shifts~\cite{aw,kpp}. Furthermore, these fronts~$\phi_c(x-ct)$ are known to be stable with
respect to perturbations in some suitable weighted spaces and to attract the solutions of the
associated Cauchy problem for a large class of exponentially decaying initial conditions, see
e.g.~\cite{bmr,b,ev,hnrr,k,kpp,lau,sa,u}. When the function~$f=f(t,u)$ depends periodically on
time $t$, the standard traveling fronts do not exist anymore in general and the notion of
traveling fronts is replaced by that of pulsating traveling fronts $\phi(t,x-ct)$, where $\phi$ is
periodic in its first variable and converges to $1$ (resp.~$0$) as $x-ct\to-\infty$ (resp. as
$x-ct\to+\infty$). The existence, uniqueness and stability properties of such pulsating traveling
fronts have been established in~\cite{hr2,lyz,lz1,lz2,n1,nrx,w}. The notions of pulsating
traveling fronts can also be extended in time almost-periodic, almost-automorphic, recurrent or
uniquely ergodic media, we refer to~\cite{hs,s1,s2,s3,s5,s6,s7} for further existence,
qualitative and asymptotic properties in such media.

In the present paper, due to the time-dependence and assumption~(\ref{fpm}),
equation~\eq{P=f} is not assumed to be periodic, almost-periodic, recurrent or
uniquely ergodic in time and the standard traveling or pulsating traveling
fronts no longer exist in general. The notion of transition fronts
satisfying~(\ref{gtf}) provides the good framework to describe the propagation
of more general front-like solutions. This notion has already been used in
various contexts. For instance, particular transition fronts have recently been
constructed for monostable equations~\eq{P=f} in~\cite{bh2,NR1} (see the
comments after Theorem~\ref{th1} below). Recently, transition fronts for
reaction-diffusion equations with general non-periodic monostable $x$-dependent
nonlinearities~$f(x,u)$ have also been constructed in~\cite{n2,nrrz,z1,z2}.

The time-dependent monostable equation~\eq{P=f} considered here, with $f(t,u)$
having some limits as~$t\to\pm\infty$, is one of the simplest examples of
heterogeneous equations which are not periodic, recurrent or uniquely ergodic in
time. Nevertheless, equation~\eq{P=f} already captures new interesting
propagating front-like solutions. In particular, we will prove the existence of
transition fronts which had not been considered in~\cite{bh2,NR1}. We will
characterize the set of all admissible rates of propagation, as well as the asymptotic
profiles of the non-critical fronts.

As for the possible rates of propagation, an important notion associated to the
transition fronts is that of their possible speed. Namely, we say that a
transition front connecting~$0$ and $1$ for~\eq{P=f}, in the sense
of Definition~\ref{def1}, has a {\em global mean speed} $\gamma$ if
\be\label{meanspeed}
\frac{X(t)-X(s)}{t-s}\to\gamma\ \hbox{ as }t-s\to+\infty,
\ee
that is $(X(t+\tau)-X(t))/\tau\to\gamma$ as $\tau\to+\infty$ uniformly in
$t\in\R$.\footnote{We point out that this definition slightly differs from the one used
in~\cite{bh1,bh2}, but, in this one-dimensional situation and given Definition~\ref{def1}, it is
quite easy to see that the two definitions coincide.}
Applying recursively Property~\eqref{Xtau} with, say, $\tau=1$, one readily
sees that
the function $\tau\mapsto(X(t+\tau)-X(t))/\tau$ is bounded on
$[1,+\infty)$, independently
of $t$, whence the global mean speed cannot be infinite.

If a transition front connecting $0$ and $1$ has a global mean speed $\gamma$, then this
speed does not depend on the family~$(X(t))_{t\in\R}$, due to the property~(\ref{tildeX}). But
the global mean speed, if any, does depend on the transition front. This is seen already in
homogeneous media, where~$f=f(u)$. Indeed, any standard traveling front $\phi(x-ct)$ in a
homogeneous medium has a global mean speed equal to $c$, and then the set of admissible
speeds in the class of all standard traveling fronts is equal to~$[2\sqrt{f'(0)},+\infty)$. This
property remains true if one considers the whole class of transition fronts, because the global
mean speed cannot be smaller than the spreading speed for the Cauchy problem with
compactly supported initial datum, which is $2\sqrt{f'(0)}$, see~\cite{aw}. In
the non-homogeneous case considered in the present paper, the picture is more
complicated and may
be radically different. Namely, under some slightly stronger assumptions
than~\eqref{eq:f/u}-\eqref{fpm}, Corollary~\ref{cor1} below asserts that the set of global mean
speeds among all transition fronts coincides with $[2\sqrt{\mu_-},+\infty)$ if $\mu_+\leq\mu_-$,
whereas it is empty if $\mu_+>\mu_-$.

It is important to realize at this stage that, in general, a given transition front
connec\-ting~$0$ and~$1$ may not have any global mean speed, even for some
homogeneous equations~\eq{P=f} with $f=f(u)$. More precisely, it follows from~\cite{hn2}, as
shown in~\cite{hr}, that for a $C^2$ concave~$f:[0,1]\to\R$ such that
$f(0)=f(1)=0<f(s)$ on $(0,1)$ and for any real numbers $c_1$ and~$c_2$ such that
$2\sqrt{f'(0)}\le c_1<c_2$, there exist transition fronts $u$ connecting $0$ and $1$ such that
\be\label{transitionhn}\left\{\baa{ll}
u(t,x)-\phi_{c_1}(x-c_1t)\to0 & \hbox{as }t\to-\infty,\vspace{3pt}\\
u(t,x)-\phi_{c_2}(x-c_2t)\to0 & \hbox{as }t\to+\infty,\eaa\right.\hbox{ uniformly in }x\in\R,
\ee
where $\phi_{c_1}(x-c_1t)$ and $\phi_{c_2}(x-c_2t)$ are any two given standard traveling
fronts connecting~$0$ and~$1$ with speeds $c_1$ and $c_2$ respectively. This result implies
in particular that, even in a homogeneous medium, the notion of transition fronts is necessary
to describe front-like solutions that are not standard traveling fronts. Furthermore, the transition
fronts satis\-fying~(\ref{transitionhn}) for~\eq{P=f} with~$f=f(u)$ do not have a global mean
speed as soon as $c_1\neq c_2$ since, whatever~$X(t)$ may be, one has $X(t)/t\to c_1$ as
$t\to-\infty$ and $X(t)/t\to c_2$ as~$t\to+\infty$ for these transition fronts. Nevertheless, it is
natural to say that these fronts have an asymptotic speed, $c_1$, as $t\to-\infty$, and another
asymptotic speed, $c_2$, as $t\to+\infty$.

These facts lead us naturally to the definition of the notion of possible asymptotic past and
future mean speeds, as $t\to-\infty$ and as $t\to+\infty$, for the general time-dependent
equation~\eq{P=f}.

\begin{definition} We say that a transition front connecting $0$ and~$1$ for the
equation~\eq{P=f} has an {\em asymptotic past speed} $c_-\in\R$, resp. an {\em asymptotic
future speed} $c_+\in\R$, if
\be\label{cpm2}
\frac{X(t)}{t}\to c_-\hbox{ as }t\to-\infty,\ \hbox{ resp. }\frac{X(t)}{t}\to c_+\hbox{ as }t\to+\infty.
\ee
\end{definition}

Notice that, for a given transition front, these signed speeds, if any, do not depend on the
family $(X(t))_{t\in\R}$, due to~\eqref{tildeX} again. Clearly, if a transition front admits a global
mean speed~$\gamma$ in the sense of~(\ref{meanspeed}), then it has asymptotic past and
future speeds equal to $\gamma$. It is natural to ask if the reverse property holds true.
Namely, if a front has past and future speeds both equal to some $\gamma$, does it admit a
global mean speed equal to $\gamma$~? We partially answered to this question in
Remark~4.1 of~\cite{hr} in the homogeneous case: the answer is yes provided
$\gamma>2\sqrt{f'(0)}$, but it is not known if $\gamma=2\sqrt{f'(0)}$. Here, we extend this
result to~\eq{P=f}, under some stronger hypotheses on the convergences $f\to f_\pm$
as $t\to\pm\infty$, showing that the answer is yes provided $\gamma>2\sqrt{\mu_-}$.
This can be derived from the last statement of
Theorem~\ref{th2} below about the convergence of the profile to that of standard
fronts, as shown in Section \ref{sec24}. We point out that the answer to the
above question in
general is no for a nonlinearity $f(t,u)$ which does not satisfy~\eqref{fpm}, see
Remark~\ref{rk:NR1} below.

The asymptotic past and future speeds characterize the rate of expansion of
the front at large negative or positive times. These asymptotic speeds might not
exist a priori and one could wonder whether these notions of speeds as
$t\to\pm\infty$ would be sufficient to describe the large time dynamics of all
transition fronts for~\eq{P=f}. As a matter of fact, one of the main purposes of
the present paper will be to characterize completely the set of all admissible
asymptotic speeds and to show that the asymptotic speeds exist for all fronts
which are supercritical as~$t\to-\infty$, in a sense which will be made more
precise in Theorem~\ref{th2} below.


\subsection{Existence of transition fronts}

We first show the existence of some transition fronts with asymptotic speeds~$c_{\pm}$
as~$t\to\pm\infty$ ranging in some explicitly given semi-infinite intervals.

\begin{theorem}\label{th1}
Under the assumptions~\eqref{eq:f/u},~\eqref{fpm} and~\eqref{C1omega}, let $\mu_{\pm}$ be
defined as in~$(\ref{defmupm})$, and~$c_{\pm}$ be any two real numbers such that
\be\label{cpm1}
c_-\ge2\sqrt{\mu_-}\ \ \hbox{and}\ \ c_+\ge\kappa+\frac{\mu_+}{\kappa},\ \ \hbox{with}\ \
\kappa=\min\Big(\sqrt{\mu_+},\frac{c_--\sqrt{c_-^2-4\mu_-}}{2}\Big)>0.
\ee
Then equation~\eq{P=f} admits some transition fronts $u$ connecting $0$ and $1$ with
asymptotic past and future speeds~$c_{\pm}$, that is, such that~\eqref{cpm2} holds.
Furthermore, $u$ satisfies $u_x(t,x)<0$ for all~$(t,x)\in\R\times\R$.
Lastly, in all cases, except possibly when $\mu_+>\mu_-$ and $c_\pm$ satisfy
$c_-=2\sqrt{\mu_-}$ and $c_+=\sqrt{\mu_-}+\frac{\mu_+}{\sqrt{\mu_-}}$,
there exists a bounded function $\xi:\R\to\R$ such that
\Fi{phi+-}
u(t,X(t)+\xi(t)+\cdot)\to\phi_{c_{\pm}}\ \hbox{ in }C^2(\R)\ \hbox{ as
}t\to\pm\infty,
\Ff
where $\phi_{c_{\pm}}(x-c_{\pm}t)$ are standard traveling fronts connecting $0$
and~$1$ for the limiting equations with non\-linearities $f_{\pm}$.
\end{theorem}

Concerning the missing case for the last statement of
Theorem \ref{th1}, we show that \eqref{eq:phi+-} holds as
$t\to-\infty$, but we derive the convergence to $\phi_{c_+}$ only along a
particular sequence $t_n\to+\infty$, under the additional assumption that $f_+$
is $C^2$ and concave, see Proposition \ref{pro:criticprofile} below.

A first obvious observation following~(\ref{cpm1}) is that $c_+\ge2\sqrt{\mu_+}$. This relation
is not at all sur\-prising, and is actually immediately necessary, since $2\sqrt{\mu_+}$ is the
spreading speed of the solutions~$u$ of the Cauchy problem~\eq{P=f} with compactly
supported nonzero initial conditions~$0\le u_0\le 1$, in the sense that
$\max_{\R\backslash(-ct,ct)}u(t,\cdot)\to0$ as $t\to+\infty$ for every $c>2\sqrt{\mu_+}$ and
$\min_{[-ct,ct]}u(t,\cdot)\to1$ as $t\to+\infty$ for every~$0\le c<2\sqrt{\mu_+}$ (this asymptotic
result can be easily obtained from the maximum principle and the facts that $u(t,x)$ has a
Gaussian decay as~$x\to\pm\infty$ at every time $t>0$, and initial conditions with Gaussian
decay spread at the speed~$2\sqrt{f'(0)}$ in the time-independent case~\cite{aw,u}). We also
refer to Proposition~\ref{pro:spreading} below for a direct proof of the bounds
$c_\pm\ge2\sqrt{\mu_\pm}$.

A second observation is a comparison between the range of asymptotic past and future
speeds provided by Theorem~\ref{th1} and the sets of admissible speeds for the limiting
problems as $t\to-\infty$ and $t\to+\infty$, which are given  by
$[2\sqrt{\mu_-},+\infty)$ and $[2\sqrt{\mu_+},+\infty)$ respectively. On the one
hand, the range of past
speeds in Theorem~\ref{th1} coincides with $[2\sqrt{\mu_-},+\infty)$, no matter what the
relation between $\mu_-$ and $\mu_+$ is. On the other hand, the range of future speeds
coincides with $[2\sqrt{\mu_+},+\infty)$ if and only if $\kappa$ in~\eqref{cpm1} coincides with
$\sqrt{\mu_+}$ for some $c_-\ge2\sqrt{\mu_-}$, and this happens if and only if
$\mu_+\leq\mu_-$. Indeed, if $\mu_+>\mu_-$ then the minimal future speed
given by
Theorem~\ref{th1} is strictly larger than $2\sqrt{\mu_+}$, namely, it is larger than the spreading
speed for the solutions of the Cauchy problem with compactly supported initial conditions.

The set of asymptotic speeds $c_\pm$ provided by Theorem~\ref{th1}, that is,
satisfying~\eqref{cpm1}, can be equivalently expressed by
\Fi{alternative}
c_\pm=\kappa_\pm+\frac{\mu_\pm}{\kappa_\pm},\quad\kappa_-\in\big(0,\sqrt{\mu_-}\big],
\quad\kappa_+\in\big(0,\min(\kappa_-,\sqrt{\mu_+})\big].
\Ff
 \begin{figure}[ht]
 \centering
 \subfigure[$\mu_+<\mu_-$]
   {\includegraphics[width=5cm]{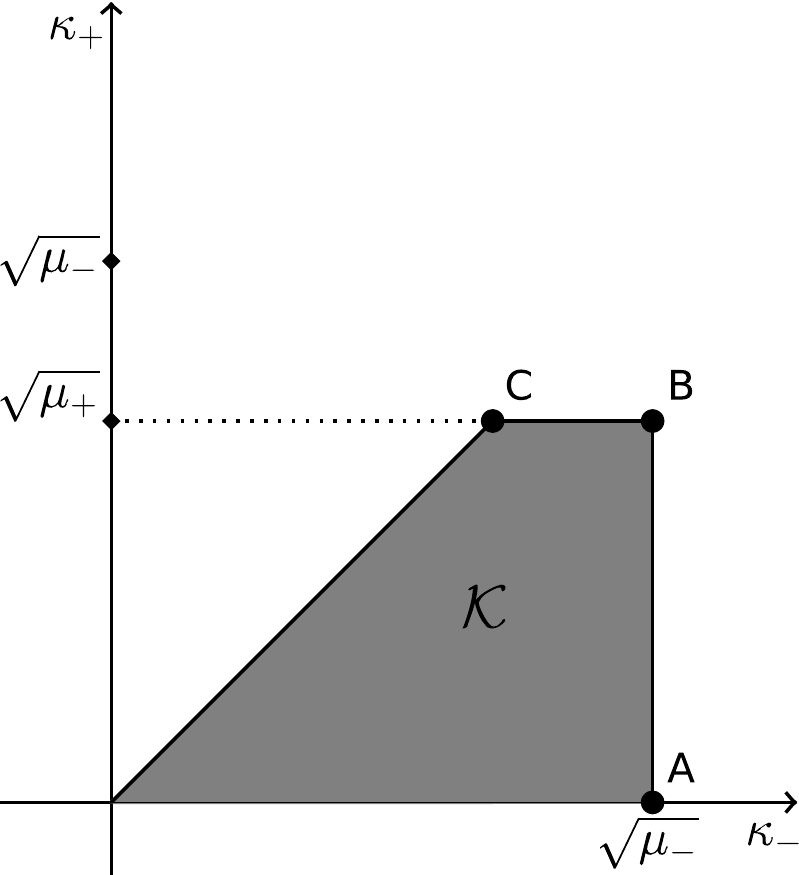}}
 \hspace{15mm}
 \subfigure[$\mu_+\geq\mu_-$]
   {\includegraphics[width=5cm]{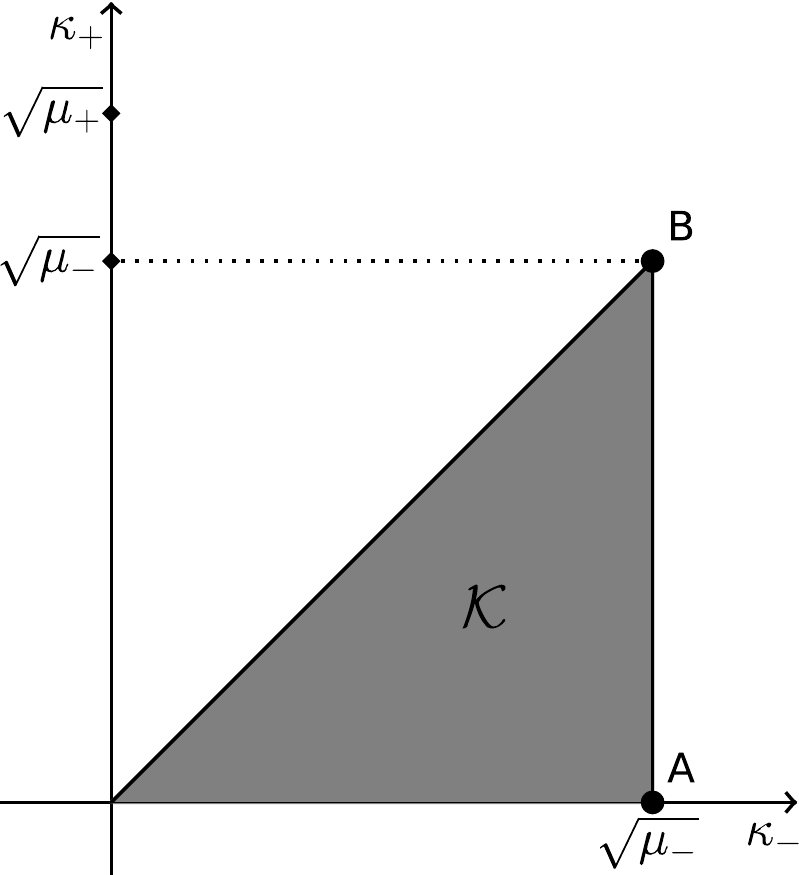}}
 \caption{The set $\mc{K}$ of admissible exponential decays.}
 \label{fig:Kappa}
 \end{figure}

\noindent The admissible pairs $(\kappa_-,\kappa_+)$ in \eqref{eq:alternative}
are represented by the shaded region in Figure \ref{fig:Kappa}.
The expression~\eqref{eq:alternative} yields an immediate
interpretation of the asymptotic speeds: $\mu_\pm$ reflect the characteristics
of the medium as $t\to\pm\infty$, while  $\kappa_\pm$ are the
exponential decaying rates (as $x\to+\infty$) of the asymptotic profiles~$\phi_{c_{\pm}}$
of the front as $t\to\pm\infty$. Thus, when $t$ passes from $-\infty$ to
$+\infty$, the rate of decay of the profile of the fronts in Theorem~\ref{th1} cannot increase. In
particular, if $\mu_+>\mu_-$ then the range of admissible values for
$\kappa_+$ is smaller
than $(0,\sqrt{\mu_+}]$, which is the range of admissible decaying rates for standard fronts of
the limiting problem as $t\to+\infty$. This can be viewed as the real reason why the set of
future speeds is smaller than that of admissible speeds for the limiting problem as $t\to+\infty$
in that case.

Based on the formulation~\eq{alternative}, let us now compare the
asymptotic past and future speeds~$c_{\pm}$ given by Theorem~\ref{th1}
with each other. If~$\mu_+>\mu_-$, the decreasing monotonicity of $k\mapsto
k+\mu_+/k$
in~$(0,\sqrt{\mu_+}]$ yields
\be\label{acceleration}
c_+=\kappa_++\frac{\mu_+}{\kappa_+}\geq\kappa_-+\frac{\mu_+}{\kappa_-}>
\kappa_-+\frac{\mu_-}{\kappa_-}=c_-,
\ee
whence the transition fronts always strictly globally accelerate, in the sense
that the asymptotic future speed~$c_+$ is strictly larger than the past
one~$c_-$. If
$\mu_+=\mu_-=\mu$ (as for instance if~$f$ does not depend on~$t$), then
$c_+=\kappa_++\mu/\kappa_+\geq \kappa_-+\mu/\kappa_-=c_-$, whence the transition fronts
always globally accelerate in this case. Lastly, if~$\mu_+<\mu_-$, then the transition fronts
may globally accelerate because~$c_+$ can be as large as wanted for a given $c_-$, but they
may also strictly decelerate, for any choice of admissible $c_-$: indeed, for any
$\kappa_-\in(0,\sqrt{\mu_-}]$, taking $\kappa_+=\min(\kappa_-,\sqrt{\mu_+})$ we deduce from
the fact that $k\mapsto k+\mu_-/k$ attains its minimum
at~$\kappa=\sqrt{\mu_-}$, that
$c_+=\kappa_++\mu_+/\kappa_+<\kappa_++\mu_-/\kappa_+\leq\kappa_-+\mu_-/\kappa_-=
c_-$.

As a matter of fact, because of the convergence~\eq{phi+-} to the limiting fronts
$\phi_{c_{\pm}}$ as $t\to\pm\infty$, the solutions constructed in Theorem~\ref{th1} above
satisfy more than~(\ref{cpm2}). Namely, except possibly when $\mu_+>\mu_-$ and the speeds $c_{\pm}$ satisfy $c_-=2\sqrt{\mu_-}$ and $c_+=\sqrt{\mu_-}+\mu_+/\sqrt{\mu_-}$, the solutions $u$ of Theorem~\ref{th1} are such that
\be\label{cpm2bis}\left\{\baa{l}
\displaystyle\lim_{\tau\to+\infty}\Big(\sup_{t\le-\tau}\Big|\frac{X(t+\tau)-X(t)}{\tau}-c_-\Big|\Big)
=0,\vspace{3pt}\\
\displaystyle\lim_{\tau\to+\infty}\Big(\sup_{t\ge0}\Big|\frac{X(t+\tau)-X(t)}{\tau}-c_+\Big|\Big)
=0,\eaa\right.
\ee
see Section~\ref{sec23} for the proof of this property. The speeds $c_{\pm}$ are thus truly
asymptotic {\it mean} speeds as~$t\to\pm\infty$. In part~1) of Theorem~2.3 of~\cite{NR1},
some transition fronts connecting $0$ and $1$ were constructed for problem~\eq{P=f} under
more general assumptions on $f$ (in particular,~$f(t,\cdot)$ is not assumed to converge to the
profiles $f_{\pm}$ as $t\to\pm\infty$). With the assumptions of the present paper, the transition
fronts given in~\cite{NR1} are special cases of the ones of Theorem~\ref{th1} above, namely
those for which~\eqref{eq:alternative} holds with~$\kappa_-=\kappa_+\in(0,\sqrt{\ul\mu})$.
Thus, on the one hand the assumptions of~\cite{NR1} are more general but, on the other hand,
the choice of possible asymptotic speeds $c_{\pm}$ provided in the present paper is larger
than in~\cite{NR1}. In particular, the speed~$c_-$ can here be as close as wanted to
$2\sqrt{\mu_-}$ and can even be equal to the critical value $2\sqrt{\mu_-}$, whereas the
corresponding~$c_-$ of~\cite{NR1} is strictly larger than
$\sqrt{\ul\mu}+\mu_-/\sqrt{\ul\mu}\;(\ge2\sqrt{\mu_-})$. Furthermore, once the past
speed~$c_-$ is assigned, the future speed~$c_+$ can be here as large as wanted whereas,
in~\cite{NR1}, the speed~$c_+$ is uniquely determined by~$c_-$ through the formula
$c_+=c_-+2(\mu_+-\mu_-)/(c_--\sqrt{c_-^2-4\mu_-})$. Let us also remark that, here, the future
speed~$c_+$ can also be equal to the critical speed $2\sqrt{\mu_+}$ for the limiting problem
as~$t\to+\infty$, provided~$\mu_+\le\mu_-$ and $c_-$ is not too large, whereas,
in~\cite{NR1},~$c_+$ is always strictly larger than~$2\sqrt{\mu_+}$. Actually, one should
think of the fronts of~\cite{NR1} as the analogues of the standard fronts~$\phi(x-ct)$ for
homogeneous equations, because they have constant exponential decaying rate. These fronts
are the keystone we use to construct other fronts which change their exponential
decay in time, at least in the supercritical case~$c_\pm>2\sqrt{\mu_\pm}$. The analysis of the
critical cases requires a special attention and the method used in the present paper to cover
these cases is actually completely different from~\cite{NR1}.


\subsection{A priori bounds on the asymptotic speeds, and asymptotic profiles}

The second main result is almost the converse of Theorem~\ref{th1}: it shows, in general, the
existence of the asymptotic speeds $c_{\pm}$ and the optimality of the relations~(\ref{cpm1}).

\begin{theorem}\label{th2}
Assume that $(\ref{eq:f/u})$ and~$(\ref{fpm})$ hold, that $f_-$ is of class~$C^2$ and concave
on $[0,1]$ and that there exists a continuous function $\zeta:(-\infty,0)\to\R$ such that
\be\label{hypf-}
\zeta\in L^1(-\infty,0)\ \hbox{ and }\ \sup_{s\in(0,1)}\Big|\frac{f(t,s)}{f_-(s)}-1\Big|\le\zeta(t)\hbox{ for all }t<0.
\ee
Then, for any transition front $u$ connecting $0$ and $1$ for problem~\eq{P=f}, there holds
\be\label{cpm2ter}\left\{\baa{l}
\displaystyle2\sqrt{\mu_-}\le c_-:=\liminf_{t\to-\infty}\frac{X(t)}{t}\le\limsup_{t\to-\infty}
\frac{X(t)}{t}<+\infty,\vspace{3pt}\\
\displaystyle\kappa+\frac{\mu_+}{\kappa}\le c_+:=\liminf_{t\to+\infty}\frac{X(t)}{t}\le
\limsup_{t\to+\infty}\frac{X(t)}{t}<+\infty,\eaa\right.
\ee
where $\kappa$ is as in~$(\ref{cpm1})$. Furthermore, if $c_->2\sqrt{\mu_-}$, then the liminf
and limsup in~$(\ref{cpm2ter})$ are limits, that is, $u$ has asymptotic past and future speeds
$c_{\pm}$ satisfying~\eqref{cpm1}. Lastly, if $c_->2\sqrt{\mu_-}$ and if there exists a
continuous function $\tilde{\zeta}:(0,+\infty)\to\R$ such that
\be\label{hypf+}
\tilde{\zeta}\in L^1(0,+\infty)\ \hbox{ and }\ \sup_{s\in(0,1)}\Big|\frac{f(t,s)}{f_+(s)}-1\Big|\le\tilde{\zeta}(t)\hbox{ for all }t>0,
\ee
then there exists a bounded function $\xi:\R\to\R$ for which~\eqref{eq:phi+-} holds true.
\end{theorem}

This result improves Theorem 1.7 of~\cite{hr}, which dealt with homogeneous equations. Two
of the main interests of Theorem~\ref{th2} are to prove the existence of the asymptotic past
and future speeds of any transition front connecting~$0$ and~$1$, provided
that~$\liminf_{t\to-\infty}X(t)/t$ is not critical (that is, larger than $2\sqrt{\mu_-}$) on one hand,
and to show the sharpness of the bounds~(\ref{cpm1}) for the asymptotic speeds on the other
hand. Thus, even if the family of fronts constructed in the proof of Theorem~\ref{th1} does not
exhaust the whole class of transition fronts connecting $0$ and~$1$, it completely covers the
range of possible asymptotic past and future speeds. Therefore the observations and the
comparisons between the speeds following the statement of Theorem~\ref{th1} apply to
arbitrary transition fronts connecting $0$ and $1$ for problem~\eq{P=f}.

When~$c_-:=\liminf_{t\to-\infty}X(t)/t$ is not critical, Theorem~\ref{th2} excludes in particular
the exis\-tence of more complex dynamics for which the set of limiting values of $X(t)/t$ as
$t\to-\infty$ and~$t\to+\infty$ would not be reduced to a singleton. The only open question is
the existence of the asymptotic speeds when $\liminf_{t\to-\infty}X(t)/t=2\sqrt{\mu_-}$. We
conjecture that the asymptotic speeds still exist in this case, and that~\eqref{eq:phi+-} holds
without any restrictions on $c_{\pm}$.

\begin{remark}\label{hypf+-}
{\rm The technical conditions~$(\ref{hypf-})$
and~$(\ref{hypf+})$
used in Theorem~$\ref{th2}$ mean that~$f(t,u)$ converges to $f_{\pm}(u)$ sufficiently fast as
$t\to\pm\infty$. We do not know whether~$(\ref{hypf-})$ is necessary or not for the first part of
Theorem~$\ref{th2}$ to hold, and $(\ref{hypf-})$ and~$(\ref{hypf+})$ for the
last part.
Notice that if, in addition to the hypotheses that $f$ and $f_{\pm}$ are of
class $C^1$ and satisfy~\eqref{eq:f/u},~\eqref{fpm} (which imply
$f'_{\pm}(0)>0$), one assumes that $f'_{\pm}(1)<0$ (which is automatically
fulfilled if $f_-$ is assumed to be concave) then the
functions~$\zeta_{\pm}:t\mapsto\sup_{u\in(0,1)}\big|f(t,u)/f_{\pm}(u)-1\big|$
are continuous in $\R$ and
the conditions~\eqref{hypf-} and~\eqref{hypf+} are therefore equivalent to $\zeta_{\pm}\in
L^1(\R_{\pm})$. This is immediately seen by noticing that the
functions~$(t,u)\mapsto f(t,u)/f_{\pm}(u)$ are continuously extended at $(t,0)$ by
$\partial_u f(t,0)/f_{\pm}'(0)$ and at $(t,1)$ by~$\partial_u f(t,1)/f_{\pm}'(1)$.
Lastly, a typical example for which these conditions~$(\ref{hypf-})$
and~$(\ref{hypf+})$ are fulfilled is when $f$ is of the type~$f(t,u)=\tilde{\mu}(t)\,\tilde{f}(u)$
with~$\lim_{t\to\pm\infty}\tilde{\mu}(t)=\tilde{\mu}_{\pm}\in(0,+\infty)$ and
$\tilde{\mu}-\tilde{\mu}_{\pm}\in L^1(\R_{\pm})$.}
\end{remark}


\subsection{The set of admissible asymptotic and global mean speeds}

We derive here an immediate corollary of Theorems~\ref{th1} and~\ref{th2},
and of formula~(\ref{cpm2bis}). It is about the characterization of the set of
all admissible asymptotic past and future speeds and global mean speeds of transition
fronts connecting $0$ and $1$ for problem~\eq{P=f}.

\begin{corollary}\label{cor1}
If~\eqref{eq:f/u},~\eqref{fpm},~\eqref{C1omega},~\eqref{hypf-} hold and $f_-$ is
of class~$C^2$ and concave, transition fronts connecting~$0$
and~$1$ and having asymptotic past and future speeds $c_{\pm}$ exist if and only
if~$c_-$ and~$c_+$ fulfill~$(\ref{cpm1})$, or equivalently~\eq{alternative}. Furthermore,
transition fronts connecting $0$ and $1$ and having a
global mean speed $\gamma$, in the sense of~\eqref{meanspeed}, exist if and only
if~$\mu_+\le\mu_-$ and~$\gamma\ge2\sqrt{\mu_-}$.
\end{corollary}

\noindent{\bf{Proof.}} The first sentence is an immediate consequence of Theorems~\ref{th1}
and~\ref{th2}.
Thus, a transition front connecting~$0$ and $1$ with a global
mean speed $\gamma$ exists only if $c_{\pm}:=\gamma$ satisfy~(\ref{cpm1}), or
equivalently~$c_{\pm}:=\gamma$ can be written in the
form~\eqref{eq:alternative}.
If~$\mu_+>\mu_-$, then $c_\pm$ in~\eqref{eq:alternative} always satisfy
$c_+>c_-$, as already emphasized
in~(\ref{acceleration}), and therefore a transition front with a global mean
speed cannot exist.

Suppose now that~$\mu_+\le\mu_-$ and take $\gamma\geq2\sqrt{\mu_-}$. Let us
show that $c_{\pm}:=\gamma$ can be written
in the form~\eqref{eq:alternative}.
The choice of $\kappa_-\in(0,\sqrt{\mu_-}]$ is uniquely determined by the
condition $c_-=\gamma$, i.e.,~$\kappa_-+\mu_-/\kappa_-=\gamma$. Then,
the equation
$\kappa_++\mu_+/\kappa_+=\gamma$ admits a
solution~$\kappa_+\in(0,\min(\kappa_-,\sqrt{\mu_+})]$ because the function
$\Gamma:\kappa\mapsto\kappa+\mu_+/\kappa$ is continuous on $(0,+\infty)$ and
satisfies~$\Gamma(0^+)=+\infty$,~$\Gamma(\kappa_-)=\kappa_- +\mu_+/\kappa_-\leq
\kappa_- +\mu_-/\kappa_-=\gamma$ and~$\Gamma(\sqrt{\mu_+})=2\sqrt{\mu_+}\leq
2\sqrt{\mu_-}\leq\gamma$.
It then follows from Theorem~\ref{th1} that there exists a transition front $u$
connecting~$0$ and~$1$ and having asymptotic past and future speeds both equal
to $\gamma$. Since $\mu_+\le\mu_-$, we know that the fronts
given by Theorem~\ref{th1} satisfy~\eqref{eq:phi+-}, which, in turn, yields
\eqref{cpm2bis}, as shown in Section \ref{sec24}.
This implies that, for the front $u$, \eqref{meanspeed} holds whenever $s$
and~$t$ have the same sign, whereas, in the case $s<0<t$, one writes
$$\left|\frac{X(t)-X(s)}{t-s}-\gamma\right|\leq
\frac{|X(t)-\gamma t|}{t-s}+\frac{|X(s)-\gamma s|}{t-s},$$
and readily sees that both terms in this sum go to $0$ as $t-s\to+\infty$ because
$X(t)/t\to\gamma$ as~$t\to\pm\infty$ by~\eqref{cpm2bis}, and $X$ is locally bounded
by~\eqref{Xtau}. The proof of Corollary~\ref{cor1} is thereby
complete.\hfill$\Box$


\subsection{A sufficient condition for an entire solution to be a transition front}

Our last main result provides a sufficient condition for an entire solution of~\eq{P=f} to be a
transition front.

\begin{theorem}\label{thm:decay}
Assume that $f$ satisfies~\eqref{eq:f/u},~\eqref{fpm},~\eqref{hypf-}  and~\eqref{hypf+},
with $f_-$ concave and in~$C^2([0,1])$. Let $0<u<1$ be a solution of~\eq{P=f} such
that
\be\label{hyphn}
\exists\,c>2\sqrt{\mu_-},\ \ \max_{[-c|t|,c|t|]}u(t,\cdot)\to0\ \hbox{ as }t\to-\infty.
\ee
Then the limit
$$\lambda:=-\lim_{x\to+\infty}\frac{\ln u(0,x)}{x}$$
exists, satisfies $\lambda\in\big[0,\sqrt{\mu_-}\big)$, and $u$ is a transition front connecting
$0$ and $1$ if and only if~$\lambda>0$. Furthermore, if $\lambda>0$, then the transition front
$u$ admits some asymptotic past and future speeds~$c_-$ and $c_+$ given by
\be\label{asympspeeds}\left\{\baa{l}
\displaystyle2\sqrt{\mu_-}<c_-=\sup\Big\{\gamma\ge0,\ \lim_{t\to-\infty}\max_{[-\gamma|t|,\gamma|t|]}u(t,\cdot)
=0\Big\},\vspace{3pt}\\
\displaystyle c_+=\min(\lambda,\sqrt{\mu_+})+\frac{\mu_+}{\min(\lambda,\sqrt{\mu_+})},
\eaa\right.
\ee
and~\eqref{eq:phi+-} holds for some bounded function $\xi:\R\to\R$.
\end{theorem}

The above result provides a characterization of transition fronts, in the class of entire
solutions~$0<u<1$ satisfying~\eqref{hyphn}, in terms of the profile of $u$ at time $0$ (or,
equivalently after shifting times, at another arbitrary time $t_0$; as a matter of fact, we show in
the proof of Theo\-rem~\ref{thm:decay} that $\ln u(t,x)\sim-\lambda x$ as~$x\to+\infty$ for all
$t\in\R$, with $\lambda\in[0,2\sqrt{\mu_-})$ independent of~$t$). The hypothesis~\eqref{hyphn}
is used to apply some results of~\cite{hn2}. This is not such a restrictive assumption in
general, because the limit in~\eqref{hyphn} automatically holds for any $0\le c<2\sqrt{\mu_-}$,
as an easy consequence of the spreading result.

\begin{remark}{\rm In~$\cite{nrrz}$, the authors consider the reaction-diffusion~\eq{P=f} with
$x$-dependent KPP type nonlinearities $f(x,u)$, instead of time-dependent ones $f(t,u)$.
Among other things they prove the following striking result: if, say,
$f_u(x,0):=\frac{\partial f}{\partial u}(x,0)\to m\in(0,+\infty)$ as~$x\to\pm\infty$
and~$f_u(\cdot,0)-m$ is nonnegative and compactly supported, then transition fronts
connec\-ting~$0$ and~$1$, in the sense of~\eqref{gtf}, do not exist if $\lambda>2m$, where
$\lambda$ is the supremum of the spectrum of the operator $\partial_{xx}+f_u(x,0)$. On the
other hand, transition fronts with global mean speed~$\gamma$ exist
when~$2m>\lambda(\ge m)$ for every speed
$\gamma\in(2\sqrt{m},\lambda/\sqrt{\lambda-m})$ $($the existence for the critical speeds is
unclear$)$. More general $x$-dependent equations have been considered in~$\cite{z1}$ and
more general existence results have been obtained under a similar smallness condition for a
quantity which is similar to $\lambda$. The time and space variables obviously play a different
role in equations of the type~\eq{P=f}. But if we wanted to make an analogy between the
transition fronts for~$x$ or $t$-dependent equations, we could make the following two
comparisons. Firstly, transition fronts always exist for our time-dependent equation~\eq{P=f}
whereas they do not exist in general for the associated $x$-dependent one. Secondly, under
the assumption~$f_u(t,0)\to\mu=\mu_{\pm}$ as~$t\to\pm\infty$ $($or more generally when
$\mu_+\le\mu_-)$ in~\eq{P=f}, transition fronts with global mean speeds always exist,
whatever the temporal range of the transition between the limiting profiles $f_{\pm}$ may be
and whatever the amplitude of $f(t,\cdot)$ for the intermediate times $t$ may be, and the set of
admissible speeds is a closed semi-infinite interval including the critical speed, whereas
transition fronts may not exist in general for the $x$-dependent equation and, even if they
exist, the set of admissible global mean speeds c a bounded interval.}
\end{remark}


\subsection{Time-dependent diffusivity}

We conclude this first section by showing that changing the time-variable, one can extend
Theorems~\ref{th1} and~\ref{th2} to equations with time-dependent diffusivities. Namely, we
consider the equation
\Fi{sigmaP=f}
u_t=\sigma(t)u_{xx}+f(t,u),\quad t\in\R,\ x\in\R,
\Ff
with $\sigma\in C^1(\R)$ being bounded from below away from $0$. Writing
$\t u(t,x)\!=\!u(\tau^{-1}(t),x)$ with~$\tau(t)\!:=\!\int_0^t\sigma(s)ds$, leads to the equation
\Fi{P=t f}
\t u_t=\t u_{xx}+\frac{f(\tau^{-1}(t),\t u)}{\sigma(\tau^{-1}(t))},\quad t\in\R,\ x\in\R.
\Ff
We can apply Theorems~\ref{th1} and~\ref{th2} to this equation, provided
the nonlinear term satisfies the hypotheses there, and then derive a
characterization of the asymptotic past and future speeds of transition fronts
connecting $0$ and $1$ for~\eqref{eq:P=t f}. Notice that $\t u$ is a transition
front for~\eqref{eq:P=t f} satisfying~\eqref{gtf} with $X =\t X$ if and only if
$u(t,x)=\t u(\tau(t),x)$ is a transition front for~\eqref{eq:P=f} with~$X(t)=\t
X(\tau(t))$. Therefore, if $\t u$ has past and future speeds equal to $\t c_\pm$
then $u$ has past and future speeds equal to
\Fi{tcpm}
c_\pm:=\lim_{t\to\pm\infty}\frac{X(t)}t=\lim_{t\to\pm\infty}\frac{\t X(\tau(t))}{\tau(t)}\frac{\tau(t)}t
=\t c_\pm\lim_{t\to\pm\infty}\frac1t\int_0^t\sigma(s)ds,
\Ff
provided the latter exist (notice that $\tau(t)\to\pm\infty$ as
$t\to\pm\infty$ because $\inf\sigma>0$).

A situation where these arguments apply is when $f$ satisfies the
hypotheses~\eqref{eq:f/u},~\eqref{fpm},~\eqref{C1omega} and~\eqref{hypf-} of
Theorems~\ref{th1} and~\ref{th2} and $\sigma\in C^1(\R)$ is such that
\Fi{sigma}
\sigma>0\hbox{ in }\R,\ \ \sigma(t)\to\sigma_\pm>0\text{ as }t\to\pm\infty\ \ \text{and}\ \
t\mapsto(\sigma(t)-\sigma_-)\in L^1(-\infty,0).
\Ff
Indeed, the new nonlinear term $\tilde{f}(t,u)=f(\tau^{-1}(t),u)/\sigma(\tau^{-1}(t))$
satisfies~\eqref{eq:f/u} and~\eqref{fpm} with $f_\pm$ replaced by $f_\pm/\sigma_\pm$,
as well as~\eqref{C1omega} with $C$ replaced by $C/\inf\sigma$. It also fulfils the
hypothesis~\eqref{hypf-} of Theorem~\ref{th2}, since
$$\sup_{s\in(0,1)}\left|\frac{f(\tau^{-1}(t),s)}{\sigma(\tau^{-1}(t))}
\frac{\sigma_-}{f_-(s)} -1\right|\leq\frac{\sigma_-}{\inf\sigma}\zeta(\tau^{-1}(t))+
\left|\frac{\sigma_-}{\sigma(\tau^{-1}(t))}-1\right|,$$
and this term belongs to $L^1(-\infty,0)$ because $\zeta$ does and $\int_{-\infty}^0\left|\sigma_-/\sigma(\tau^{-1}(t))-1\right|dt
=\int_{-\infty}^0|\sigma_--\sigma(z)|\,dz<\infty$.
One can therefore derive the results for~\eqref{eq:sigmaP=f} from the ones
for~\eqref{eq:P=t f}, noticing that, in virtue of~\eqref{eq:tcpm}, when coming back to the
original time-variable, the asymptotic speeds are multiplied by $\sigma_\pm$. For instance,
applying Corollary~\ref{cor1} we can characterize the admissible past and future speeds
$\t c_\pm$ for~\eqref{eq:P=t f} by~$\t c_-\geq 2\sqrt{\mu_-/\sigma_-}$ and
$\t c_+\geq\kappa+\mu_+/(\kappa\sigma_+)$, where
$\kappa=\min\big(\sqrt{\mu_+/\sigma_+},\big(\t c_--\sqrt{\t c_-^2-4\mu_-/\sigma_-}\big)/2\big)$,
and thus derive the following:

\begin{corollary}\label{corsigma}
Assume that $f$ satisfies~\eqref{eq:f/u},~\eqref{fpm},~\eqref{C1omega} and~\eqref{hypf-}
and that $\sigma\in C^1(\R)$ satisfies~\eqref{eq:sigma}. Then transition fronts connecting $0$
and~$1$ for~\eqref{eq:sigmaP=f} having asymptotic past and future speeds $c_{\pm}$ exist if
and only $c_-\geq 2\sqrt{\sigma_-\mu_-}$ and $c_+\geq\kappa+(\sigma_+\mu_+)/\kappa$,
where $\kappa=\min\big(\sqrt{\sigma_+\mu_+},(\sigma_+/\sigma_-)\times\big(c_--\sqrt{c_-^2-4\sigma_-\mu_-}\big)/2\big)$.
\end{corollary}

\noindent{\bf{Outline of the paper.}} Section~\ref{sec:ex} is concerned with the proof of
Theorem~\ref{th1}. More precisely, Section~\ref{sec:super-critical} deals with the existence of
transition fronts of~\eq{P=f} for non-critical asymptotic speeds~$c_{\pm}$, while the critical
cases are considered in Section~\ref{sec:critical}. Section~\ref{sec3} is devoted to the proofs
of Theorems~\ref{th2} and~\ref{thm:decay}: after recalling in
Section~\ref{sec31} some known
useful results on the transition fronts in the homogeneous case, we
prove in Section~\ref{sec32} the a priori bounds on the asymptotic speeds of any transition
front connecting $0$ and~$1$ for~\eq{P=f}, as well as the asymptotic behavior of
all non-critical fronts. Lastly, Theorem~\ref{thm:decay}
is proved in Section~\ref{sec33}.


\SE{Existence of transition fronts: proof of Theorem~\ref{th1}}\label{sec:ex}

This section is dedicated to the proof of Theorem~\ref{th1}. In Section~\ref{sec:super-critical},
we derive the existence of transition fronts having supercritical past and future speeds
$c_\pm$, that is, such that $c_\pm>2\sqrt{\mu_\pm}$. Recall that the set of asymptotic past
and future speeds $c_\pm$ satisfying~\eqref{cpm1} can be expressed by~\eq{alternative}, with
$(\kappa_-,\kappa_+)$ belonging to the set
$$\mc{K}:=\big\{(k_1,k_2);\ 0<k_1\leq\sqrt{\mu_-},\
0<k_2\leq\min(k_1,\sqrt{\mu_+})\big\},$$
see~Figure \ref{fig:Kappa}.
Supercritical speeds are the ones for which $(\kappa_-,\kappa_+)\in\mc{K}$ and
$\kappa_\pm<\sqrt{\mu_\pm}$. Actually, the case~$\kappa_-=\kappa_+<\sqrt{\ul\mu}=
\min(\sqrt{\mu_-},\sqrt{\mu_+})$, that is the oblique open segment in Figure~\ref{fig:Kappa},
has been treated in~\cite{NR1}. In Section~\ref{sec:super-critical} of the present paper,
we will construct a supercritical front associated with any
choice of $0<\kappa_-<\sqrt{\mu_-}$ and~$0<\kappa_+<\min(\kappa_-,\sqrt{\mu_+})$ using
two distinct fronts of~\cite{NR1}.

In Section~\ref{sec:critical}, we deal with the case where at least one between $c_-$ and
$c_+$ is critical. We start with~$c_->2\sqrt{\mu_-}$ and $c_+=2\sqrt{\mu_+}$. Next, we make
use of the critical front provided by the recent paper~\cite{n2}, which, roughly speaking, has
the slowest admissible past and future speeds. In particular, being slower than any front with
supercritical speed, its past and future speeds~$c_\pm$ satisfy the equalities
in~\eqref{cpm1}, namely~$c_-=2\sqrt{\mu_-}$ and $c_+=\sqrt{\ul\mu}+\mu_+/\sqrt{\ul\mu}$.
Finally, using the critical front and the same method as in Section~\ref{sec:super-critical}, we
construct fronts such that $c_-=2\sqrt{\mu_-}$ and $c_+$ satisfies the strict inequality
$c_+>\sqrt{\ul\mu}+\mu_+/\sqrt{\ul\mu}$ in~\eqref{cpm1}.

To summarize, the construction of the fronts corresponding to the different
portions of the set $\mc{K}$ in Figure \ref{fig:Kappa} is derived in:
\begin{itemize}
 \item \cite{NR1} (see also Proposition \ref{pro:NR1} below): the oblique open
segment;
 \item Section \ref{sec:super-critical}: the interior of $\mc{K}$;
 \item Section \ref{sec:BC}:  the segment $(BC]$ (in the case $\mu_+<\mu_-$),
without the point $B$;
 \item Section \ref{sec:B}:  the point $B$;
 \item Section \ref{sec:AB}:  the segment $(AB)$;
\end{itemize}

In Section~\ref{sec23}, we show some exponential lower bounds which are used in
the
construction of the transition front with critical asymptotic speeds and which are also of
independent interest. Finally, the slightly stronger properties~\eqref{cpm2bis}
are proved in
Section~\ref{sec24}.


\subsection{Interior of $\mc{K}$ : supercritical
speeds}\label{sec:super-critical}

We will make use of the existence result of~\cite{NR1}, Theorem~2.3 part~1). That result
applies to a general time-dependent nonlinearity $f$ and it is expressed in terms of the {\em
least mean} of the function $\mu(t):=\partial_uf(t,0)$,
see Definition~2.2 in~\cite{NR1}. Under the hypotheses~\eqref{fpm},~\eqref{defmupm}
considered in the present paper, the least mean of $\mu$ coincides with $\ul\mu=
\min(\mu_-,\mu_+)$. This is a consequence of the fact that~$\mu(t)\to\mu_\pm$ as
$t\to\pm\infty$, which, in turn, follows immediately from~\eqref{fpm}, after writing $f(t,u)/u=(f(t,u))/f_\pm(u))\times(f_\pm(u)/u)$.
Actually, in~\cite{NR1} it is assumed that $f(t,u)>0$ for $u\in(0,1)$. However, it is shown
in~\cite{rr} - where more general coefficients depending also on $x$ are considered - that the
construction in~\cite{NR1} works only requiring that~$f(\.,u)$ is nonnegative and has positive
least mean for any $u\in(0,1)$. Under the hypothesis~\eqref{fpm}, the least mean of $f(\.,u)$ is
equal to $\min(f_-(u),f_+(u))$, and then it is positive for $u\in(0,1)$. Thus, Theorem 1.3
in~\cite{rr} yields the following

\begin{proposition}\label{pro:NR1}
Assume that $f$ satisfies~\eqref{eq:f/u},~\eqref{C1omega}, that $\mu(t)=\partial_u f(t,0)$
admits positive limits~$\mu(\pm\infty)$ as $t\to\pm\infty$ and that, for every $u\in(0,1)$, the
least mean of $f(\cdot,u)$ is positive. Then, for every $\kappa\in\big(0,\sqrt{\min(\mu(-\infty),
\mu(+\infty))}\big)$, there exists a transition front $u$ connecting~$0$ and~$1$ such that
\Fi{Xnr1}
X(t)=\int_0^t \left(\kappa+\frac{\mu(s)}{\kappa}\right)ds\ \hbox{ for all }t\in\R.
\Ff
Furthermore, $u_x(t,x)<0$ for all $(t,x)\in\R\times\R$ and
\Fi{expdecay}
u(t,x+X(t))\,e^{\kappa x}\to1\text{ as }x\to+\infty,\quad\text{uniformly in }t\in\R.
\Ff
\end{proposition}

\begin{remark}\label{rk:NR1}{\rm The fronts constructed in~\cite{NR1} still
satisfy~\eqref{eq:Xnr1}, even when~\eqref{fpm} does not hold (and $\mu$ is bounded). Hence,
they admit a global mean speed $\gamma$ if and only if
$(1/t)\int_s^{s+t}\!\!\mu(\tau)d\tau\to(\gamma-\kappa)\,\kappa$ as $t\to\pm\infty$
uniformly in $s\in\R$.
In order to admit past and future speeds it is instead sufficient that the above limits exist, not
necessarily coinciding, for a given $s\in\R$ (and then for every $s\in\R$ because $\mu$ is
bounded). This shows that in the case of a general time-dependent reaction term, the
asymptotic past and future speeds may or may not exist, and even if they exist and coincide
this does not imply the existence of a global mean speed.}
\end{remark}

We now derive the existence result in the supercritical case.

\begin{proposition}\label{pro:exsupercrit}
Under the assumptions~\eqref{eq:f/u},~\eqref{fpm} and~\eqref{C1omega}, for every
$(\kappa_-,\kappa_+)\in\mc{K}$ such that~$\kappa_\pm<\sqrt{\mu_\pm}$,
equation~\eq{P=f} admits a transition front $u$ connecting $0$ and $1$ with
asymptotic past and future speeds $c_\pm:=\kappa_\pm+\mu_\pm/\kappa_\pm$.
Furthermore, $u_x(t,x)<0$ for all $(t,x)\in\R\times\R$ and there exists a bounded function
$\xi:\R\to\R$ such that~\eqref{eq:phi+-} holds true.
\end{proposition}

\begin{proof}
The proof is divided into four steps.

{\em Step 1: construction of the transition front.} Let $(\kappa_-,\kappa_+)\in\mc{K}$ satisfy
$\kappa_\pm<\sqrt{\mu_\pm}$. If~$\kappa_-=\kappa_+$, the front is directly provided by
Proposition~\ref{pro:NR1}. Let us consider the other case, that is,~$\kappa_+<\kappa_-$. We
introduce the following symmetrization of $f$: $\t f(t,u):=f(-|t|,u)$. The function
$\t\mu(t):=\partial_u \t f(t,0)$ satisfies $\t\mu(\pm\infty)=\mu_-$. By Proposition~\ref{pro:NR1},
there exists a transition front $u_1$ for the nonlinearity $\t f$ such that~\eqref{gtf} holds
with~$X=X_1$ given by
\be\label{defX1}
X_1(t):=\int_0^t \left(\kappa_-+\frac{\t\mu(s)}{\kappa_-}\right)ds\ \hbox{ for all }t\in\R.
\ee
In particular, $u_1$ is a solution of~\eq{P=f} for $t<0$. Since $0<\kappa_+=\min(\kappa_-,
\kappa_+)<\sqrt{\min(\mu_-,\mu_+)}$, Proposition~\ref{pro:NR1} also provides a
transition front $u_2$ for the original equation~\eq{P=f}, which satisfies~\eqref{gtf} with
$X=X_2$ given by
\be\label{defX2}
X_2(t):=\int_0^t \left(\kappa_++\frac{\mu(s)}{\kappa_+}\right)ds\ \hbox{ for all }t\in\R.
\ee
By~\eq{expdecay} we know that $u_1$ and $u_2$ satisfy
\Fi{u12exp}
u_1(t,x+X_1(t))e^{\kappa_- x}\to1\ \text{and}\ u_2(t,x+X_2(t))e^{\kappa_+ x}\to1
\text{ as }x\to+\infty,\ \text{uniformly in }t\in\R.
\Ff
Let us set $\ul u(t,x):=\max(u_1(t,x), u_2(t,x))$ and $\ol u(t,x):=\min(u_1(t,x)+u_2(t,x)),1)$.
The function $\ul u$ is a generalized subsolution of~\eq{P=f} for $t<0$. On the other hand,
since by~\eq{f/u},
$$\forall\,t\in\R,\ \forall\, 0<\alpha\leq\beta\le1\!-\!\alpha,\ \ f(t,\alpha\!+\!\beta)\!\leq\!
\frac{f(t,\beta)}\beta(\alpha\!+\!\beta)\!=\!\frac{f(t,\beta)}\beta\alpha\!+\!f(t,\beta)\!\leq\!f(t,\alpha)\!
+\!f(t,\beta),$$
it follows that $\ol u$ is a generalized supersolution of~\eq{P=f} for $t<0$.
For
$n\in\N$, let $u^n$ denote the (bounded) solution of~\eq{P=f} for $t>-n$, with initial datum
$u^n(-n,x)=\ul u(-n,x)$. The parabolic comparison principle yields $\ul u\leq u^n\leq\ol u$ in
$(-n,0)\times\R$, and, moreover, $u^n\geq u_2$ in the whole $(-n,+\infty)\times\R$. Using
interior parabolic estimates we see that, up to extraction of a subsequence, the
sequence~$(u^n)_{n\in\N}$ converges locally uniformly to an entire solution $u$ of~\eq{P=f}
satisfying $0<\ul u\leq u\leq\ol u\le 1$ in $(-\infty,0)\times\R$, as well as $u\geq u_2$
in~$\R\times\R$. Furthermore, $u$ is nonincreasing in $x$
because this is true for $\ul u$ whence for the $u^n$ by the comparison principle.

We now claim that $u$ is a transition front connecting $0$ and $1$ for~\eq{P=f}, such
that~\eqref{gtf} holds with
\be\label{defX12}
X(t)=\begin{cases}X_1(t) & \text{ if }t<0\\
                    X_2(t) & \text{ if }t\geq0.
\end{cases}
\ee
Since $\mu(\pm\infty)=\mu_\pm$, as seen at the beginning of the section,
\eqref{defX12} will then imply
that $u$ has the desired past and future speeds $c_\pm=\kappa_{\pm}+\mu_{\pm}/
\kappa_{\pm}$. Let us start to check it for large negative times. The
inequalities~$0<\kappa_+<\kappa_-<\sqrt{\mu_-}$ yield
$$\lim_{t\to-\infty}\frac{X_1(t)}t=\kappa_-+\frac{\mu_-}{\kappa_-}<
\kappa_++\frac{\mu_-}{\kappa_+}=\lim_{t\to-\infty}\frac{X_2(t)}t.$$
Thus, since $u_2(t,x+X_2(t))\to0=\inf u_1$ as $x\to+\infty$ and $u_1(t,x+X_1(t))\to1=
\sup u_2$ as~$x\to-\infty$ uniformly in $t\in\R$, we infer that
$\sup_{x\in\R}\big|\ul u(t,x)-u_1(t,x)\big|\to0$ and $\sup_{x\in\R}
\big|\ol u(t,x)-u_1(t,x)\big|\to0$ as~$t\to-\infty$.
Hence $u(t,x)-u_1(t,x)\to0$ as $t\to-\infty$ uniformly in $x\in\R$, because
$\ul u\leq u\leq\ol u$ in $(-\infty,0)\times\R$. It follows that, for given $\e>0$, there
exist $T_\e<0$ and $R_{\epsilon}>0$ such that
\Fi{interface}
\inf_{t<T_\e,\,x<-R_\e}u(t,x+X_1(t))>1-\e\ \hbox{ and }\ \sup_{t<T_\e,\,x>R_\e}u(t,x+X_1(t))<\e.
\Ff
We now focus on positive times. Using~\eqref{eq:u12exp} with $0<\kappa_+<\kappa_-$ and
the limit $u_2(t,x+X_2(t))\to1$ as~$x\to-\infty$,  we deduce the existence of $M_\e\geq1$ such
that $M_\e u_2(T_\e,x)\geq u_1(T_\e,x)$ for all $x\in\R$.
Consequently,~$u(T_\e,x)\leq\ol u(T_\e,x)\leq\min\big((M_\e+1)u_2(T_\e,x),1\big)$
for all $x\in\R$.
As for $\ol u$, the function $\min((M+1)u_2,1)$ is a generalized supersolution of~\eq{P=f} by
the last hypothesis in~\eqref{eq:f/u}. Hence, by comparison,
\be\label{uu2}
\forall\,t\in[T_\e,+\infty),\ \forall\,x\in\R,\quad
u_2(t,x)\leq u(t,x)\leq\min\big((M_\e+1)u_2(t,x),1\big).
\ee
Since $u_2$ satisfies~\eqref{gtf} with $X=X_2$, we find a constant $R'_\e>0$ such that
$$\inf_{t\geq T_\e,\,x<-R'_\e}u(t,x+X_2(t))>1-\e,\qquad
\sup_{t\geq T_\e,\,x>R'_\e}u(t,x+X_2(t))<\e.$$
This and~\eq{interface} prove the claim, because $X_1$ and $X_2$ are locally bounded.

{\em Step 2: $u_x<0$ in $\R\times\R$.} We know that $u_x(t,x)\leq0$ for all
$(t,x)\in\R\times\R$. Differentiating~\eq{P=f} with respect to $x$, we find that the function
$u_x$ is an entire solution of a linear parabolic equation. Being nonpositive, the parabolic
strong maximum principle implies that it is either strictly negative, or identically equal to $0$.
The latter case is ruled out because, as a transition front connecting $0$ and $1$, $u$ is such
that~$u(t,-\infty)=1$, $u(t,+\infty)=0$ for every $t\in\R$.

{\em Step 3: convergence to a standard front as $t\to+\infty$.} We finally need to show that
there exists a  bounded function $\xi$ such that~\eq{phi+-} holds when $t\to\pm\infty$, where
$u$, $X$ and $c_\pm$ are defined in Step~1. The function $\xi$ is chosen in such a way that
$u(t,X(t)+\xi(t))=1/2$ for~$t\in\R$. Notice that $\xi$ is bounded by~\eqref{gtf}. Let us first prove
here the convergence~\eq{phi+-} as~$t\to+\infty$. To do so, consider an arbitrary sequence
$\seq{t}$ in $\R$ diverging to $+\infty$. As $n\to+\infty$, the functions~$u(t+t_n,x+X(t_n)+
\xi(t_n))$ converge (up to subsequences) locally uniformly in~$(t,x)\in\R\times\R$ to a solution
$0\le\t u\le1$ of~\eq{P=f} with $f$ replaced by $f_+$, satisfying~$\t u(0,0)=1/2$ and $\tilde{u}_x\le0$ in~$\R\times\R$.\par
We now derive the exponential decay of $\t u$.
By~\eq{u12exp} and~\eqref{uu2}, the profile of $u$ decays as~$x\to+\infty$
with exponential rate $\kappa_+$, in the sense that there are $R>0$ and $M>1$
such that
\Fi{expbounds}
\forall\,t\ge0,\ \forall\,x>R,\quad M^{-1}\leq
u(t,x+X_2(t))\,e^{\kappa_+x}\leq M.
\Ff
For fixed $(t,x)\in\R\times\R$, we want to estimate
$\t u(t,x+c_+t) =\limn u(t+t_n,x+c_+t+X(t_n)+\xi(t_n))$.
We know that $X(t_n)=X_2(t_n)$ for $n$ large enough and, by \eqref{defX2}
and $c_+=\kappa_++\mu_+/\kappa_+$, we get that
$c_+t+X_2(t_n)-X_2(t+t_n)\to0$ as $n\to+\infty$.
As a consequence, if $x>R+\|\xi\|_{L^{\infty}(\R)}$, we can apply
\eqref{eq:expbounds} and deduce
\Fi{M}
\tilde{M}^{-1}\leq
\t u(t,x+c_+t)\,e^{\kappa_+x}\leq \tilde{M},\Ff
with $\tilde{M}=M\,e^{\kappa_+\|\xi\|_{L^{\infty}(\R)}}$. Let $\phi_{c_+}(x-c_+t)$ be the
standard traveling front for the non\-linearity~$f_+$, connecting~$0$ and~$1$ and moving with
speed $c_+=\kappa_++\mu_+/\kappa_+>2\sqrt{\mu_+}$, normalized by~$\phi_{c_+}(0)=1/2$.
We know from~\cite{aw} that $\phi_{c_+}$ has the same exponential decay
$\kappa_+$ as
$\t u$, in the sense that there exists $A\ge1$ such that
$A^{-1}\le\phi_{c_+}(y)\,e^{\kappa_+y}\le A$ for all $y\ge0$. It then follows from Proposition 4.3
in~\cite{NR1} that there is~$a\geq0$ such that
$\phi_{c_+}(x-c_+t+a)\leq\t u(t,x)\leq
\phi_{c_+}(x-c_+t-a)$ for all $(t,x)\in\R\times\R$.
This allows us to apply a Liouville type result from~\cite{bh1} (Theorem 3.5
there, see also Lemma~8.2 of~\cite{hnrr2}, adapted here to the homogeneous case)
and infer the existence of $b\in\R$ such that $\t u(t,x)=
\phi_{c_+}(x-c_+t-b)$ for all $(t,x)\in\R\times\R$. Since~$\t
u(0,0)=1/2=\phi_{c_+}(0)$ and $\phi_{c_+}$ is
strictly decreasing, we derive $b=0$. We have shown in particular that (up to subsequences)
$u(t_n,x+X(t_n)+\xi(t_n))\to\phi_{c_+}(x)$ as $n\to+\infty$ locally uniformly in $x\in\R$. Since
the limit $\phi_{c_+}(x)$ does not depend on the particular sequence~$\seq{t}$ diverging to
$+\infty$, we deduce that~$u(t,x+X(t)+\xi(t))\to\phi_{c_+}(x)$ as $t\to+\infty$ locally
uniformly in $x\in\R$. The convergence actually holds uniformly in $x\in\R$ - whence in
$C^2(\R)$ by parabolic estimates - because $0<u(t,x)<1$ is decreasing with respect to
$x\in\R$ for any $t\in\R$ and~$\phi_{c_+}(-\infty)=1$, $\phi_{c_+}(+\infty)=0$.

{\em Step 4: convergence to a standard front as $t\to-\infty$.} Consider here a sequence
$\seq{t}$ diverging to~$-\infty$ and let $\t u$ be as in Step~3. In order to apply the previous
arguments to show that $\t u$ coincides with the standard traveling front $\phi_{c_-}(x-c_-t)$
for the nonlinearity $f_-$, normalized by $\phi_{c_-}(0)=1/2$, and thus to conclude the proof, it
is sufficient to check that there exists $\tilde{M}\geq1$ such
that, for $x$ large enough, \eqref{eq:M} holds with $c_+$ and $\kappa_+$
replaced by $c_-$ and $\kappa_-$. Since
$u(t,x)-u_1(t,x)\to0$ as $t\to-\infty$ uniformly in $x\in\R$ and $X=X_1$ on
$\R_-$, for fixed~$(t,x)\in\R\times\R$ we see that
\Fi{limitdecay}\begin{split}
   \t u(t,x+c_-t)\,e^{\kappa_-x} &= \limn u(t+t_n,x+c_-t+X(t_n)+\xi(t_n))\,e^{\kappa_-x}\\
&= \limn u_1(t+t_n,x+c_-t+X_1(t_n)+\xi(t_n))\,e^{\kappa_-x}.
  \end{split}
\Ff
Now, we know from the one hand that $c_-t+X_1(t_n)-X_1(t+t_n)\to0$ as
$n\to+\infty$
by~\eqref{defX1} and~$c_-=\kappa_-+\mu_-/\kappa_-$, and from the other
hand that there exists $R'>0$ such that
$$\forall\,t\in\R,\ \forall\,n\in\N,\ \forall\,y>R',\quad\frac12\leq
u_1(t+t_n,y+X_1(t+t_n))e^{\kappa_-y}\leq2.$$
Thus, for all $t\in\R$ and $x>R'+\|\xi\|_{L^\infty(\R)}$,~\eqref{eq:limitdecay}
yields
$e^{-\kappa_-\|\xi\|_{L^\infty(\R)}}/2\leq   \t u(t,x+c_-t)\,e^{\kappa_-x}\leq
2\,e^{\kappa_-\|\xi\|_{L^\infty(\R)}}$,
that is,~\eqref{eq:M} holds with $c_+$ and $\kappa_+$ replaced by $c_-$ and $\kappa_-$.
The proof of Proposition~\ref{pro:exsupercrit} is thereby complete.
\end{proof}


\subsection{Critical asymptotic past or future speeds}\label{sec:critical}

In this subsection, we construct transition fronts connecting $0$ and $1$ with either critical
past speed $c_-$ or critical future speed $c_+$, that is, $c_-=2\sqrt{\mu_-}$ or
$c_+=2\sqrt{\mu_+}$. This will conclude the proof of Theorem~\ref{th1}.

Until the end of Section \ref{sec:critical}, we assume that $f$
satisfies~\eqref{eq:f/u},~\eqref{fpm} and~\eqref{C1omega}.

\subsubsection{A lower bound on the asymptotic past and future speeds}

We first derive an easy consequence of the spreading result of~\cite{aw}.

\begin{proposition}\label{pro:spreading}
Any transition front $u$ connecting $0$ and $1$ for~\eq{P=f} satisfies
$\liminf_{t\to\pm\infty}X(t)/t\geq2\sqrt{\mu_\pm}$.
In particular, if $u$ has some asymptotic past or future speeds $c_{\pm}$, then
$c_{\pm}\ge2\sqrt{\mu_{\pm}}$.
\end{proposition}

\begin{proof}
By hypothesis~\eqref{fpm}, for any $\e\in(0,1/2)$, there is $T_\e\in\R$ such that
$f(t,u)\geq(1-\e)f_+(u)$ for all~$t>T_\e$ and $u\in(0,1)$. Let $u$ be a transition front
connecting $0$ and $1$ for~\eq{P=f}. Hence, $u$ is a supersolution of the problem
\Fi{P=f+e}
w_t= w_{xx}+(1-\e)f_+(w),
\Ff
for $t>T_\e$, $x\in\R$. It then follows from~\cite{aw} that
$u(t,2\sqrt{(1-2\e)\mu_+}\, t)\to1$ as $t\to+\infty$,
whence, by~\eqref{gtf},~$\liminf_{t\to+\infty}\big(X(t)-2\sqrt{(1-2\e)\mu_+}\, t\big)>-\infty$. The
inequality $\liminf_{t\to+\infty}X(t)/t\ge2\sqrt{\mu_+}$ then follows from the
arbitrariness of $\e\in(0,1/2)$.

The case $t\to-\infty$ is similar: take $\e\in(0,1/2)$ and let $T_\e\in\R$ be such that $u$ is a
supersolution of the problem
\Fi{P=f-e}
w_t= w_{xx}+(1-\e)f_-(w),
\Ff
for $t<T_\e$, $x\in\R$. By~\eqref{gtf}, there exists a continuous function $v_0:\R\to[0,1]$
which is not identically equal to~$0$ and satisfies $u(t,X(t)+x)\ge v_0(x)$ for all $t\in\R$ and
$x\in\R$. Let $v$ be the solution of $v_t=v_{xx}+(1-\e)f_-(v)$ for $t>0$ and $x\in\R$,
emerging from the initial datum $v_0$. By~\cite{aw} we know that
$v(t,2\sqrt{(1-2\e)\mu_-}\, t)\to1$ as~$t\to+\infty$. Therefore, one infers by comparison that,
for any $s<T_\e$,
$$1\ge u\big(T_\e,X(s)+2\sqrt{(1-2\e)\mu_-}(T_\e-s)\big)\geq
v\big(T_\e-s,2\sqrt{(1-2\e)\mu_-}(T_\e-s)\big)\to1\text{ as }s\to-\infty.$$
It follows then from~\eqref{gtf} that
$\limsup_{s\to-\infty}\big(X(s)+2\sqrt{(1-2\e)\mu_-}(T_\e-s)\big)<+\infty$, which concludes the
proof of the proposition due to the arbitrariness of $\e\in(0,1/2)$.
\end{proof}

\subsubsection{Segment $(BC]$: supercritical past speed and critical future
speed}
\label{sec:BC}

We now construct fronts with asymptotic speeds $c_\pm$ satisfying
\eqref{eq:alternative} with the restrictions
$$\kappa_-<\sqrt{\mu_-}\ \hbox{ and }\ \kappa_+=\sqrt{\mu_+}.$$
Since $\sqrt{\mu_+}=\kappa_+\leq\kappa_-$, this case is allowed only if
$\mu_+<\mu_-$. We
know from Proposition~\ref{pro:NR1} that, for~$\kappa\in(0,\sqrt{\mu_+})$, there is a transition
front $u_\kappa$ connecting 0 and 1 satisfying~\eqref{gtf} with
$$X(t)=X_\kappa(t):=\int_0^t \left(\kappa+\frac{\mu(s)}{\kappa}\right)ds\ \hbox{ for all }t\in\R,$$
and such that $u_{\kappa}$ is decreasing with respect to $x$. We need some additional
properties of the transition fronts $(u_\kappa)_{0<\kappa<\sqrt{\mu_+}}$, which are derived
in their construction in the proof of Theorem~2.3 part~1) of~\cite{NR1} or Theorem~1.3
of~\cite{rr}. Let us recall the construction. For $\kappa\in(0,\sqrt{\mu_+})$, set
\be\label{defUkappa}
U_\kappa(x):=\min(e^{-\kappa x},1).
\ee
For $n\in\N$, let $u_\kappa^n$ be the bounded solution of~\eq{P=f} for $t>-n$, emerging from
$$u_\kappa^n(-n,x)=U_\kappa(x-x_n),\quad\text{where }\
x_n:=\int_0^{-n}\left(\kappa+\frac{\mu(s)}{\kappa}\right)ds.$$
Then $u_\kappa$ is the locally uniform limit of (a subsequence of) $(u_\kappa^n)_{n\in\N}$.

We claim that, performing the same construction, but taking
$\kappa=\kappa_-\in[\sqrt{\mu_+},\sqrt{\mu_-})$, one obtains a transition front with the desired
asymptotic past and future speeds $c_\pm$. Call $u:=u_{\kappa_-}$ the function constructed
in such a way. The exponential decay $\kappa=\kappa_-$ is admissible in the construction
of~\cite{NR1,rr} if one replaces the nonlinearity $f$ with $\t f(t,s):=f(-|t|,s)$. For $t<0$,~$\t f=f$
and therefore $u$ coincides with a transition front connecting $0$ and $1$ for the nonlinearity~
$\t f$ and it satisfies~\eqref{gtf} with $X$ such that $X'(t)=\kappa_-+\mu(t)/\kappa_-$
for~$t<0$. This implies that $u$ satisfies~\eqref{gtf} for $t<0$, with $X$ such that
$X(t)/t\to\kappa_-+\mu_-/\kappa_-$ as $t\to-\infty$. In particular, $u$ has an asymptotic past
speed equal to $\kappa_-+\mu_-/\kappa_-=c_-$.

In order to investigate the properties of $u$ for positive times, consider the family
$(U_\kappa)_{0<\kappa<\sqrt{\mu_+}}$ defined in~\eqref{defUkappa}. Fix
$\kappa\in(0,\sqrt{\mu_+})$ and, for $\rho\in\R$, call $U_{\kappa}^\rho$ the translated of
$U_\kappa$ by $\rho$, that is,~$U_{\kappa}^\rho(x):=U_{\kappa}(x+\rho)$. Since
$0<\kappa<\sqrt{\mu_+}=\kappa_+\le\kappa_-$, any translated $U_{\kappa}^\rho$ is less
steep than~$U_{\kappa_-}(x):=\min(e^{-\kappa_-x},1)$, in the sense that there is
$\zeta^\rho\in\R$ such that
$ U_{\kappa_-}\geq U_\kappa^\rho$ in $(-\infty,\zeta^\rho]$ and
$U_{\kappa_-}<U_\kappa^\rho$ in $(\zeta^\rho,+\infty)$.
Thus, the classical result about the number of zeros of solutions of linear parabolic equations
(see~\cite{a}, and also~\cite{dm,dgm,lau}) implies that, for any $\rho$ and
$t\in\R$, there exists
$\zeta^\rho_t\in\R\cup\{\pm\infty\}$ for which
\Fi{intersection}
u(t,x)\geq u_\kappa(t,x+\rho)\text{ if }x\leq\zeta^\rho_t\ \hbox{ and }\
u(t,x)\leq u_\kappa(t,x+\rho)\text{ if }x\geq\zeta^\rho_t.
\Ff
This readily implies that, for any $t\in\R$ and $0<a<b<1$, the diameter of the transition zone
$\{x\in\R;\ a\leq u(x,t)\leq b\}$ cannot be bigger than that of $\{x\in\R;\ a\leq u_\kappa(x,t)\leq
b\}$. Thus, since transition fronts are characterized by the uniform boundedness in time of
transition zones, and $0<u<1$ by the strong \MP, we
deduce that $u$ is a transition front for~\eq{P=f}, as $u_{\kappa}$ is.
Namely,~\eqref{gtf} holds for some function $X:\R\to\R$. Moreover, the second inequality
in~\eqref{eq:intersection} implies that we can choose a large negative $\rho$
in such a way that
$u(0,x)\leq u_\kappa(0,x+\rho)$ for all $x\geq0$.
On the other hand, $u_\kappa(0,x+\rho)\geq u_\kappa(0,\rho)>0$ for $x\leq0$
since $u_{\kappa}$ is
decreasing with respect to $x$. As a consequence, there exists $M\geq1$ such
that~$u(0,x)\leq M u_\kappa(0,x+\rho)$ for all $x\in\R$. Hence, by comparison,
$$u(t,x)\leq\min(M u_\kappa(t,x+\rho),1)\ \hbox{ for all
}(t,x)\in\R_+\times\R,$$
because $\min(M u_\kappa,1)$ is a generalized supersolution of~\eq{P=f} by the last
hypothesis in~\eqref{eq:f/u}. Therefore, from the fact that $u_\kappa$ admits future speed
equal to $\kappa+\mu_+/\kappa$, one easily gets that the function $X$ for which $u$
satisfies~\eqref{gtf} verifies
\be\label{limsupkappa}
\limsup_{t\to+\infty}\frac{X(t)}t\leq\kappa+\frac{\mu_+}\kappa.
\ee
Since this holds for all $\kappa\in(0,\sqrt{\mu_+})$, we derive $\limsup_{t\to+\infty}X(t)/t
\leq2\sqrt{\mu_+}$. Owing to Proposition~\ref{pro:spreading}, we eventually infer that $u$ has
a future speed equal to $c_+=2\sqrt{\mu_+}$.

It remains to show that $u$ satisfies $u_x(t,x)<0$ for all $(t,x)\in\R\times\R$, as well
as~\eq{phi+-} for some bounded function $\xi$. Consider again the family of supercritical fronts
$(u_\kappa)_{0<\kappa<\sqrt{\mu_+}}$ given by Proposition~\ref{pro:NR1}. These are the
same fronts provided by Proposition~\ref{pro:exsupercrit} in the cases
where~$\kappa_-=\kappa_+$, and therefore they satisfy the same type of properties we want
to derive for $u$. The fact that $u$ satisfies everywhere $u_x<0$ then follows immediately
from~\eqref{eq:intersection}. Indeed, for a given~$(t,x)\in\R\times\R$, let
$\rho\in\R$ be such that
$u(t,x)=u_\kappa(t,\rho)$. Then~$u_x(t,x)>(u_\kappa)_x(t,\rho)$ would
violate~\eqref{eq:intersection}, whence $u_x(t,x)\leq(u_\kappa)_x(t,\rho)<0$.

The convergence in~\eq{phi+-} as $t\to-\infty$ is a consequence of
Proposition~\ref{pro:exsupercrit}. Indeed $u$ coincides for $t<0$ with the supercritical front for
the nonlinearity $\t f$ with past and future speeds both equal to
$c_-=\kappa_-+\mu_-/\kappa_-$ given by Proposition~\ref{pro:NR1} or, equivalently, by
Proposition~\ref{pro:exsupercrit} and such a front satisfies the desired convergence as
$t\to-\infty$ by the last statement of Proposition~\ref{pro:exsupercrit}.

We now deal with the convergence as $t\to+\infty$. Let $X$ be such that $u$
satisfies~\eqref{gtf}. Up to perturbing $X$ by adding a bounded function, we can assume
without loss of genera\-lity that $u(t,X(t))=1/2$ for all $t\in\R$. For given $\seq{t}$ diverging to
$+\infty$, the functions~$(u(\.+t_n,\.+X(t_n)))_{n\in\N}$ converge (up to subsequences) locally
uniformly to a solution~$\t u$ of the limit equation with nonlinearity $f_+$. Moreover,
$\t u(0,0)=1/2$. Let~$(X_\kappa)_{0<\kappa<\sqrt{\mu_+}}$ be the family of functions
for which the transition fronts $(u_\kappa)_{0<\kappa<\sqrt{\mu_+}}$ satisfy~\eqref{gtf},
together with~$u_\kappa(t,X_\kappa(t))=1/2$ for all $t\in\R$ (the real numbers $X_{\kappa}(t)$
are then uniquely defined, since the functions $u_{\kappa}$ are continuously decreasing in
$x$). We know by Proposition~\ref{pro:exsupercrit} that there exists a family of bounded
functions $(\xi_\kappa)_{0<\kappa<\sqrt{\mu_+}}$ such that
\Fi{xi-kappa}
\forall\,\kappa\in(0,\sqrt{\mu_+}),\quad u_\kappa(t,X_k(t)+\xi_\kappa(t)+\cdot)\to
\phi_{c_\kappa}\ \hbox{ in }C^2(\R)\ \hbox{ as }t\to+\infty,
\Ff
where $c_\kappa=\kappa+\mu_+/\kappa$ and $\phi_{c_\kappa}(x-c_\kappa t)$ is a standard
traveling front for the equation with nonlinearity $f_+$. Fix any given
$\kappa\in(0,\sqrt{\mu_+})$. By adding a constant to $\xi_\kappa$ if need be, we can reduce
without loss of generality to the case where $\phi_{c_\kappa}$ satisfies
$\phi_{c_\kappa}(0)=1/2$. We then have
$1/2=\lim_{t\to+\infty}u_\kappa(t,X_k(t))=
\lim_{t\to+\infty}\phi_{c_\kappa}(-\xi_\kappa(t))$,
whence $\xi_\kappa(+\infty)=0$ by the strict monotonicity of $\phi_{c_\kappa}$. It then follows
from the uniform continuity of the~$u_\kappa$ and their space derivatives up to order 2,
that~\eqref{eq:xi-kappa} holds with $\xi_\kappa\equiv0$. Now, for any~$\e>0$, we
have~$u(t,X(t)+\e)<1/2=u_\kappa(t,X_\kappa(t))$ for all $t\in\R$ and thus, owing
to~\eqref{eq:intersection},
$$\forall\,t\in\R,\ \forall\,x\geq0,\quad u(t,x+X(t)+\e)\leq u_\kappa(t,x+X_\kappa(t)).$$
The arbitrariness of $\e>0$ implies that $u(t,x+X(t))\leq u_\kappa(t,x+X_\kappa(t))$ for all
$t\in\R$ and $x\geq0$. The reverse inequality for $x\leq0$ is obtained in analogous way. Using
these inequalities at $t=t_n$ and letting $n\to+\infty$, we eventually derive
by~\eqref{eq:xi-kappa} (with $\xi_\kappa\equiv0$),
$$\forall\,\kappa\in(0,\sqrt{\mu_+}),\quad \t u(x,0)\geq\phi_{c_\kappa}(x)\text{ for }x\leq0
\ \hbox{ and }\ \t u(x,0)\leq\phi_{c_\kappa}(x)\text{ for }x\geq0.$$
On the other hand, as $\kappa\to\sqrt{\mu_+}$, the standard traveling fronts $\phi_{c_\kappa}$
converge uniformly in $\R$ to the (unique) critical traveling front $\phi_{2\sqrt{\mu_+}}$ for the
nonlinearity $f_+$, normalized by $\phi_{2\sqrt{\mu_+}}(0)=1/2$. We infer that
\be\label{tildeuphi}
\t u(x,0)\geq\phi_{2\sqrt{\mu_+}}(x)\ \text{ for }x\leq0,\qquad \t
u(x,0)\leq\phi_{2\sqrt{\mu_+}}(x)\ \text{ for }x\geq0.
\ee
This means that $\t u$ is steeper than $\phi_{2\sqrt{\mu_+}}$ at time $0$. But it is known
that~$\phi_{2\sqrt{\mu_+}}(x-2\sqrt{\mu_+}t)$ is the steepest entire solution of the equation
with nonlinearity $f_+$. Indeed, it is the limit of the solutions $v_n(t,x)$ of the corresponding
Cauchy problems emerging from $v_n(-n,\cdot)=\1_{(-\infty,x_n)}$ at time $-n$ for some
sequence $(x_n)_{n\in\N}$, where $\1_{(-\infty,x_n)}$ denotes the characteristic function of
the interval $(-\infty,x_n)$, and the Heaviside function is steeper than any function ranging
in~$[0,1]$ (see e.g.~\cite{hr,kpp,n2}). Therefore, the reverse inequalities
of~\eqref{tildeuphi} hold true as well, because~$\t u(0,0)=1/2=\phi_{2\sqrt{\mu_+}}(0)$. We
have eventually shown that, up to subsequences,~$u(t_n,X(t_n)+\cdot)$ converges locally
uniformly to $\phi_{2\sqrt{\mu_+}}$ as $n\to+\infty$. Since this holds for any  sequence
$\seq{t}$ diverging to $+\infty$, we deduce that $u(t,\.+X(t))\to\phi_{2\sqrt{\mu_+}}$ as
$t\to+\infty$ locally uniformly in $x\in\R$. The convergence actually holds uniformly in
$x\in\R$~-~whence in $C^2(\R)$ by parabolic estimates~-~because $0<u(t,\.)<1$ is
decreasing for any $t\in\R$ and $\phi_{2\sqrt{\mu_+}}(-\infty)=1$,
$\phi_{2\sqrt{\mu_+}}(+\infty)=0$. This concludes the proof of Theorem~\ref{th1} in this case.

\subsubsection{Point $B$: minimal past and future
speeds}\label{sec223}
\label{sec:B}

We deal with the case $\kappa_-=\sqrt{\mu_-}$ and $\kappa_+=\min(\sqrt{\mu_-},
\sqrt{\mu_+})=\sqrt{\ul\mu}$\,, that is, asymptotic speeds $c_\pm$ given by
\Fi{speeds*}
c_-=2\sqrt{\mu_-}\ \ \hbox{ and }\ \ c_+=\begin{cases}
2\sqrt{\mu_+} & \!\!\text{if }\mu_-\geq\mu_+\\
\displaystyle\sqrt{\mu_-}+\frac{\mu_+}{\sqrt{\mu_-}} & \!\!\text{otherwise.}\end{cases}
\Ff
To this aim, we make use of the ``critical'' transition front $u$ connecting $0$ and $1$
for~\eq{P=f}, which can be constructed as in~\cite{n2}. Namely, there is a sequence
$(x_n)_{n\in\N}$ in $\R$ such that the solutions~$u_n$ of~\eq{P=f} for~$t>-n$ and emerging
from $u_n(-n,\cdot)=\1_{(-\infty,x_n)}$ at time $-n$, converge locally uniformly in
$\R\times\R$, up to extraction of a subsequence, to an entire solution $0<u<1$ of~\eq{P=f}
such that $u(0,0)=1/2$ and $u$ is decreasing in $x$. Furthermore, since the Heaviside
function is steeper than any function ranging in~$[0,1]$, it follows from~\cite{a} that $u$ is
critical in the following sense: if $0<v<1$ is any solution of~\eq{P=f} coinciding with $u$ at
some~$(t_0,x_0)\in\R\times\R$, then either $v\equiv u$ in $\R\times\R$, or $u(t_0,x)>v(t_0,x)$
for $x<x_0$ and~$u(t_0,x)<v(t_0,x)$ for $x>x_0$. In other words, $u$ is steeper than any
entire solution~$0<v<1$ of~\eq{P=f}.
Taking as $v$ a suitable translation of one of the transition fronts
$(u_\kappa)_{0<\kappa<\sqrt{\ul\mu}}$ connecting~$0$ and $1$
for~\eq{P=f} provided by Proposition~\ref{pro:NR1}, we conclude, as
in~\eq{intersection}, that $u$ is a transition front connecting $0$
and $1$.

From this one claims that, if $X$ is a function for which the critical front $u$
satisfies~\eqref{gtf}, and $v$ is transition front connecting $0$ and $1$ with asymptotic past
and future speeds $\t c_\pm$, then
\be\label{tildecpm}
\limsup_{t\to\pm\infty}\frac{X(t)}t\leq \t c_\pm.
\ee
More precisely, the inequality $\limsup_{t\to+\infty}X(t)/t\le\t c_+$ can be established as
in~\eqref{limsupkappa}. On the other hand, if we assume by contradiction that
$\limsup_{t\to-\infty}X(t)/t>\t c_-$ and if $\t X:\R\to\R$ denotes a function for which~\eqref{gtf}
holds for the transition front $v$, then there is a sequence~$(t_n)_{n\in\N}$ in $\R$ diverging to
$-\infty$ and such that $X(t_n)-\t X(t_n)\to-\infty$ as $n\to+\infty$.
Let~$\rho>0$ be any
positive real number. It follows then from~\eqref{infsup} for~$X$ and~\eqref{gtf} for~$\t X$
that~$u(t_n,X(t_n))<v(t_n,X(t_n)+\rho)$ for $n$ large enough, whence
$u(t_n,x)<v(t_n,x+\rho)$ for
all~$x\ge X(t_n)$ by the criticality of $u$. Since $u<1$ in $\R\times\R$ and since
$\inf_{x\le X(t_n)}v(t_n,x+\rho)\to1$ as~$n\to+\infty$ by~\eqref{gtf} for $\t
X$, one infers that
$u(t_n,\cdot)\le 2 v(t_n,\cdot +\rho)$ in $\R$ for all $n$ large enough. Hence,
for~$n$ large
enough,~$u(t,x)\le\min\big(2v(t,x+\rho),1\big)$ for all $t\ge t_n$ and $x\in\R$
by the maximum
principle, since~$\min\big(2v(t,x+\rho),1\big)$ is a generalized supersolution
of~\eq{P=f} by the
last hypothesis in~\eqref{eq:f/u}. By letting~$n\to+\infty$, one concludes in particular that
$0<u(0,0)\le 2v(0,\rho)$ and the limit~$\rho\to+\infty$ leads to a
contradiction, since~$v(0,+\infty)=0$.

As a consequence of~\eqref{tildecpm}, by considering all fronts $v$ given by
Proposition~\ref{pro:exsupercrit}, one derives $\limsup_{t\to\pm\infty}X(t)/t\leq c_\pm$,
with $c_\pm$ given by~\eqref{eq:speeds*}. Proposition~\ref{pro:spreading} therefore implies
that $u$ has the desired asymptotic past speed $c_-=2\sqrt{\mu_-}$, as well as the desired
asymptotic future speed $c_+=2\sqrt{\mu_+}$ in the case~$\mu_-\ge\mu_+$.

Suppose now that $\mu_-<\mu_+$. In this case the future speed of
$u$ will not coincide
with the critical speed~$2\sqrt{\mu_+}$ for the limiting problem as $t\to+\infty$, and thus the
spreading property given by Proposition~\ref{pro:spreading} does not allow us to conclude. In
order to show that $u$ has the desired future speed in this case, we make use of the following
lower bound on the exponential decay of the transition front~$u$:
\Fi{decay>}
\forall\,t\in\R,\ \forall\,\lambda>\sqrt{\mu_-},\quad
\inf_{x>0}e^{\lambda x}u(t,x)>0.
\Ff
We postpone the proof of this estimate until the next subsection (see Corollary~\ref{cor:exp}
below) and we complete the proof of the limit
$$\frac{X(t)}{t}\to c_+=\sqrt{\mu_-}+\frac{\mu_+}{\sqrt{\mu_-}}\ \hbox{ as }t\to+\infty.$$
By hypothesis~\eqref{fpm}, for $\e\in(0,1)$, there exists $T_\e\in\R$
such that $f(t,s)\ge(1-\e)f_+(s)$ for all $t>T_\e$ and~$s\in(0,1)$. For any given
$\kappa\in(0,\sqrt{(1-\e)\mu_+})$, the homogeneous problem~\eqref{eq:P=f+e} admits a
standard traveling front $\phi_{\kappa,\e}(x-c_{\kappa,\e}t)$ connecting $0$ and $1$ with
speed
$$c_{\kappa,\e}=\kappa+\frac{(1-\e)\mu_+}\kappa.$$
The function $\phi_{\kappa,\e}(x-c_{\kappa,\e}t)$ is thus a subsolution to~\eq{P=f} for
$t>T_\e$, $x\in\R$. It is also well known~\cite{aw,u} that $\phi_{\kappa,\e}$ decays like
$e^{-\kappa x}$ as $x\to+\infty$. Consequently, for all choices of $\e\in(0,1)$ and~$\kappa$
satisfying~$\sqrt{\mu_-}<\kappa<\sqrt{(1-\e)\mu_+}$, we deduce from~\eqref{eq:decay>} that
$u(T_\e,x)>\phi_{\kappa,\e}(x)$ for $x$ larger than some~$x_{\kappa,\e}\in\R$. Hence, since
$\inf_{x\le x_{\kappa,\e}}u(T_\e,x)>0$ by~\eqref{gtf} and the continuity and positivity of~$u$,
there exists~$\delta_{\kappa,\e}\in(0,1)$ such
that~$u(T_\e,\cdot)>\delta_{\kappa,\e}\phi_{\kappa,\e}$ in $\R$. Notice that
$\delta_{\kappa,\e}\phi_{\kappa,\e}(x-c_{\kappa,\e}t)$ is also a subsolution to~\eq{P=f}
for~Ê$t>T_\e$ and $x\in\R$, by the last hypothesis in~\eqref{eq:f/u}. The parabolic comparison
principle then yields~$u(t,x)\geq\delta_{\kappa,\e}\phi_{\kappa,\e}(x-c_{\kappa,\e}(t-T_\e))$ for
all $t>T_\e$ and $x\in\R$. This eventually implies
\be\label{ckappaeps}
\liminf_{t\to+\infty}\frac{X(t)}t\geq c_{\kappa,\e}=\kappa+\frac{(1-\e)\mu_+}\kappa
\ee
for all $0<\e<1$ and $\kappa$ satisfying $\sqrt{\mu_-}<\kappa<\sqrt{(1-\e)\mu_+}$, from which
the desired result follows by letting~$\kappa\to\sqrt{\mu_-}$ and then $\e\to0$.

The convergence to the profile of the critical standard front for the homogeneous
non\-linearities~$f_-$ as~$t\to-\infty$ follows from the criticality property, analogously to the
convergence as~$t\to+\infty$ in the case of critical future speed treated above. The same
situation holds for the convergence as $t\to+\infty$ in the case~$\mu_+\leq\mu_+$.
As mentioned after the statement of Theorem~\ref{th1}, in the case
$\mu_+>\mu_-$ we derive the following weaker result.

\begin{proposition}\label{pro:criticprofile}
In addition to the hypotheses of Theorem~$\ref{th1}$, assume that $f_+$ is $C^2$
and concave and that $\mu_+>\mu_-$. Then there exists a sequence $\seq{t}$
diverging to $+\infty$ and a bounded function $\xi:\R\to\R$ such that the
critical transition front $u$ with asymptotic speeds
$c_-=2\sqrt{\mu_-}$ and $c_+=\sqrt{\mu_-}+\mu_+/\sqrt{\mu_-}$
given by Theorem~$\ref{th1}$ satisfies
$u(t_n,X(t_n)+\xi(t_n)+\cdot)\to\phi_{c_+}$ in $C^2(\R)$ as
$n\to+\infty$.
\end{proposition}

\begin{proof}
We keep the same notation as above. Property \eqref{Xtau} implies that the
linear interpolation of the
function $\Z\ni z\mapsto X(z)$ is Lipschitz-continuous and that the difference
between such function and $X$ is bounded on $\R$. Hence, it is not restrictive
to assume that $u$ satisfies \eqref{gtf} with $X$ Lipschitz-continuous.
It makes then sense to consider the a.e.~defined derivative of $X$, which can
be interpreted as an instantaneous speed of $u$. We will apply to the function
$X'$ a result quoted from \cite{NR2} concerning its {\em upper mean}. We recall that
the upper mean of a function $g\in L^\infty(\R)$ is defined by
$$\lceil g\rceil:=\lim_{t\to+\infty}\,\sup_{\tau\in\R}
\frac{1}{t}\int_\tau^{\tau+t}g(s)ds.$$
In order to focus on positive times, we actually define
$$c(t):=\begin{cases}X'(t) & \text{for a.e. }t>0\\
         0 & \text{for }t\leq0.
        \end{cases}$$
Applying Proposition 4.4 of \cite{NR2} to the function $t\mapsto-c(-t)$, we
infer the existence of a sequence~$\seq{t}$ such that $c(t_n+\.)$ converges as
$n\to+\infty$, in the $L^\infty(\R)$ weak-$\star$ topology,
to some function~$\t c\in L^{\infty}(\R)$ such that
$(1/t)\int_0^t \t c(s)ds\to\lceil c\rceil$ as
$t\to-\infty$. Since
$$\lceil c\rceil\geq\lim_{t\to+\infty}
\frac{1}{t}\int_0^tc(s)ds=\lim_{t\to+\infty}
\frac{X(t)-X(0)}{t}=c_+>0,$$
the sequence $\seq{t}$ necessarily diverges to $+\infty$ as $n\to+\infty$.
As $n\to+\infty$, (a subsequence of)~$u(t_n+\cdot,X(t_n)+\cdot)$ converges
locally uniformly in $\R\times\R$ to a solution $0\le\t u\le1$ of
the limiting equation with nonlinearity $f_+$. We claim that $\t u$ is a transition front
connecting~$0$ and~$1$ for that equation, satis\-fying~\eqref{gtf}
with
$\t X(t)=\int_0^t\t c(s)ds$.
Indeed, for any given $\epsilon\in(0,1)$, let~$M\ge0$ be such
that~$u(\tau,X(\tau)+x)\ge1-\epsilon$ (resp.~$u(\tau,X(\tau)+x)\le\epsilon$) for all $t\in\R$ and
$x\le-M$ (resp.~$x\ge M$). Now, for any~$t\in\R$ and $x\le-M$,
$$\t u(t,\t X(t)+x)=\lim_{n\to+\infty}u\Big(t_n+t,X(t_n)+\!\int_0^t\!\!\t c(s)ds+x\Big)=
\lim_{n\to+\infty}u\Big(t_n+t,X(t_n)+\!\int_0^t\!\!c(t_n\!+\!s)ds+x\Big).$$
But, for $n$ large enough, one has $X(t_n)+\int_0^tc(t_n+s)ds=X(t_n+t)$ since $t_n\to+\infty$,
whence $u(t_n+t,X(t_n)+\int_0^tc(t_n+s)ds+x)=u(t_n+t,X(t_n+t)+x)\ge1-\epsilon$ for $n$
large enough. Therefore,~$\t u(t,\t X(t)+x)\ge1-\epsilon$. Similarly, one can prove that
$\t u(t,\t X(t)+x)\le\epsilon$ for all $t\in\R$ and $x\ge M$.

As a consequence, the transition front $\t u$ admits an asymptotic past speed equal
to $\lim_{t\to-\infty}\t X(t)/t=\lceil c\rceil$, with $\lceil c\rceil\geq c_+
=\sqrt{\mu_-}+\mu_+/\sqrt{\mu_-}>2\sqrt{\mu_+}$ (remember that $0<\mu_-<\mu_+$ here).
We then know from Theorem~1.7 of~\cite{hr} that the transition front $\t u$ also admits an
asymptotic future
speed $\t c_+=\lim_{t\to+\infty}\t X(t)/t$ which is larger than or equal to the past speed,
namely,~$\t c_+\geq\lceil c\rceil$. The reverse inequality
is a consequence of the definition of upper mean, and therefore~$\t u$ has past
and
future speeds both equal to $\lceil c\rceil$. It follows then from
Remark~4.1 of~\cite{hr} that~$\t u$ is a standard traveling front of the type
$\t u(t,x)=\phi_{\lceil c\rceil}(x-\lceil c\rceil t)$ for the limiting equation
with nonlinearity $f_+$.
Since $\phi_{\lceil c\rceil}(-\infty)=1$, $\phi_{\lceil c\rceil}(+\infty)=0$
and $0<u<1$ is decreasing in~$x$, the (subsequence of)
$u(t_n,X(t_n)+\xi(t_n)+\cdot)$ actually converges to $\phi_{\lceil c\rceil}$
uniformly in $\R$, and then in~$C^2(\R)$ by parabolic estimates.

It then only remains to show that $\lceil c\rceil=c_+$, i.e., that $\lceil
c\rceil\leq c_+$. Consider the same family $(u_\kappa)_{0<\kappa<\sqrt{\mu_-}}$ as in
Proposition~\ref{pro:NR1}.
For $\kappa\in(0,\sqrt{\mu_-})$, call
$$X_\kappa(t):=\int_0^t \left(\kappa+\frac{\mu(s)}{\kappa}\right)ds\ \hbox{
for all }t\in\R,$$
and let $\xi_\kappa$ be a bounded function such that
$u_\kappa(t,X_\kappa(t)+\xi_\kappa(t))=1/2$ for all $t\in\R$.
It follows from the monotonicity in $x$ of $u_\kappa$ that
$u_\kappa(t,X_\kappa(t)+\xi_\kappa(t)+x)\geq1/2$ for all
$(t,x)\in\R\times\R_-$.
Moreover, letting~$L\in\R$ be such that $u(t,X(t)+L)<1/2$ for all
$t\in\R$, the steepness property of the critical front $u$
yields~$u_\kappa(t,X_\kappa(t)+\xi_\kappa(t)+x)\geq u(t,X(t)+L+x)$ for all
$(t,x)\in\R\times\R_+$.
Hence, for all $\tau\in\R$ and $x\in\R$,
$2 u_\kappa(\tau,X_\kappa(\tau)+\xi_\kappa(\tau)+x)\geq
u(\tau,X(\tau)+L+x)$. The comparison principle then yields
$$2 u_\kappa(\tau+t,X_\kappa(\tau)+\xi_\kappa(\tau)+x)\geq
u(\tau+t,X(\tau)+L+x)\quad\text{for all }\tau\in\R,\
(t,x)\in\R_+\times\R.$$
From this inequality and the fact that $u_\kappa$ and $u$ fulfil \eqref{gtf}-\eqref{infsup}
with $X_\kappa$ and $X$ respectively, one readily deduces the existence of a
positive constant
$C$ such that
$$X(\tau+t)-X(\tau)\leq X_\kappa(\tau+t)-X_\kappa(\tau)+C=
\int_\tau^{\tau+t} \left(\kappa+\frac{\mu(s)}{\kappa}\right)ds+C
\quad\text{for all
}\tau\in\R,\ t\in\R_+.$$
As a consequence, recalling that $c=0$ on $\R_-$, we derive $\lceil c\rceil\leq
\kappa+\mu_+/\kappa$, from which
$\lceil c\rceil\leq c_+$ follows by letting $\kappa\nearrow\sqrt{\mu_-}$.
\end{proof}

\subsubsection{Segment $(AB)$: critical past speed and non-minimal future speed}
\label{sec:AB}

It remains to consider the case $\kappa_-=\sqrt{\mu_-}$ and
$0<\kappa_+<\min(\sqrt{\mu_-},\sqrt{\mu_+})$, that is,
$c_-=2\sqrt{\mu_-}$ and~$c_+=\kappa_++\mu_+/\kappa_+$
with $0<\kappa_+<\sqrt{\ul\mu}$.
We define the front $u$ as in the proof of Proposition~\ref{pro:exsupercrit}, taking as $u_1$
the front with minimal past and future speeds of the previous case and as $u_2$ the front
given by Proposition~\ref{pro:NR1} with $\kappa=\kappa_+$. Notice that $u_1$ has strictly
slower asymptotic past and future speeds than~$u_2$. Moreover, by the criticality property,
$u_1(0,x)\leq u_2(0,x)$ for $x$ larger than some $x_0$, whence there is~$M\geq1$ such that
$u_1(0,x)\leq M u_2(x,0)$ for all $x\in\R$. These are the properties that allow one to apply the
arguments of the first step in the proof of Proposition~\ref{pro:exsupercrit} and to conclude that
$u$ has the desired past and future speeds.

Property~\eq{phi+-} is also a consequence of the arguments in the proof of
Proposition~\ref{pro:exsupercrit}. More precisely, we have seen there that
$|u(t,x)-u_1(t,x)|\to0$ as $t\to-\infty$ uniformly in $x\in\R$, and thus the convergence
as~$t\to-\infty$ in~\eq{phi+-} follows because we know from the previous case that it is
satisfied by $u_1$. The convergence as $t\to+\infty$ is proved in Step~3 of the proof of
Proposition~\ref{pro:exsupercrit}.


\subsection{Exponential behaviour of supersolutions}\label{sec23}

In this subsection we derive a sharp lower bound on the exponential decay of supersolutions
of homogeneous equations, in the spirit of Lemma 3.1 in~\cite{rr}. Let us mention that the
arguments can be extended to higher dimensional cases.

\begin{theorem}\label{thm:exp}
Let $u$ be a nonnegative classical supersolution of
$u_t\ge u_{xx}+g(u)$ for $t\leq0$ and $x\in\R$,
with $g:\R\to\R$ of class $C^1$ such that $g(0)=0$ and $g'(0)>0$, and
assume that there exists a function~$X:(-\infty,0]\to\R$ for which
\Fi{gamma}
\inf_{t\leq0,\;x\leq0}u(t,x+X(t))>0\ \hbox{ and }\
c:=\limsup_{t\to-\infty}\frac{X(t)}t<+\infty.
\Ff
Then, either $c<2\sqrt{g'(0)}$ and $\inf_{t\leq0,\;x\in\R}u(t,x)>0$, or $c\geq2\sqrt{g'(0)}$ and
$$\forall\,\lambda>\frac{c+\sqrt{c^2-4g'(0)}}2,\quad \inf_{x\geq0}e^{\lambda x}u(0,x)>0.$$
\end{theorem}

\begin{proof} We distinguish the two cases. The interesting case is
$c\geq2\sqrt{g'(0)}$, the other one being a consequence of the
standard spreading result.

{\em Case $c<2\sqrt{g'(0)}$.}
Call
$m:=\inf_{t\leq0,\;x\leq0}u(t,x+X(t))>0$, and let $v$ be the solution
of the Cauchy pro\-blem $v_t=v_{xx}+g(v)$, $t>0$, $x\in\R$, with initial datum
$v(0,x)=m\1_{(-\infty,0]}(x)$, $x\in\R$.
Take $\gamma\in(c,2\sqrt{g'(0)})$. The spreading result \cite{aw} and the fact
that~$v$ is nonincreasing in $x$, because so is its initial datum, imply
\Fi{spreading1d}
m':=\liminf_{t\to+\infty}\Big(\inf_{x\le\gamma t}\,v(t,x)\Big)>0.
\Ff
For any $s\leq 0$, the function $v^{-s}(t,x):=v(t-s,x-X(s))$
lies below $u$
at time $t=s$. Then, by the comparison principle, we derive that, for all $s\leq t\leq0$ and
$x\in\R$, $u(t,x)\geq v(t-s,x-X(s))$.
From this, since for fixed $t\leq0$ and~$x\in\R$,
$x-X(s)\leq\gamma(t-s)$ for $-s$ large enough, letting
$s\to-\infty$ and using~\eq{spreading1d} we get $u(t,x)\geq m'$.

{\em Case $c\geq2\sqrt{g'(0)}$.} Take $\lambda>(c+\sqrt{c^2-4g'(0)})/2>0$. Let
$(\eta_n)_{n\ge 2}$ be the functions defined by
$$\eta_n(t,x):=\left(1-\frac x{\ln(t+2n)}\right)^{\lambda\ln(t+2n)}\quad\hbox{for }
t\geq-n,\ 0\leq x<\ln(t+2n).$$
For $\e>0$ and $n\ge 2$, calling for short $\rho:=1-x/\ln(t+2n)\in(0,1]$, we
find that $\eta_n$
satisfies in its domain of definition
\[\begin{split}
\frac{\partial_t\eta_n-\partial_{xx}\eta_n-(c+\e)\partial_x\eta_n}{\eta_n} &=
\frac{\lambda}{t+2n}(\ln\rho+\rho^{-1}-1)-\lambda^2\rho^{-2}
+\frac\lambda{\ln(t+2n)}\rho^{-2}+\lambda(c+\e)\rho^{-1}\\
&\leq \lambda\rho^{-2}\left(\frac{\rho-\rho^2}{t+2n}+\frac1{\ln(t+2n)}\right)-\lambda^2\rho^{-2}
+\lambda(c+\e)\rho^{-1}.
\end{split}\]
For any given $n\ge 2$, there exists then $h_n>0$ independent of $t\ge-n$
and~$x\in[0,\ln(t+2n))$ such that
$$\frac{\partial_t\eta_n-\partial_{xx}\eta_n-(c+\e)\partial_x\eta_n}{\eta_n}\leq
\lambda\frac{c+\e}\rho-\frac{\lambda^2- \lambda h_n}{\rho^2},\quad t>-n,\ 0\le x<\ln(t+2n),$$
with $h_n\to0$ as $n\to+\infty$. If $h_n<\lambda$ then the right-hand side above is increasing
in~$\rho\in(0,2(\lambda-h_n)/(c+\e)]$. Notice that $2\lambda/(c+\e)>1$ for $\e>0$ small
enough, because~$\lambda>c/2$. Hence, for $\e>0$ small
enough,~$2(\lambda-h_n)/(c+\e)>1$ for $n$ large and thus, since $\rho\in(0,1]$, under such
conditions we find that
$$\frac{\partial_t\eta_n-\partial_{xx}\eta_n-(c+\e)\partial_x\eta_n}{\eta_n}
\leq c\lambda-\lambda^2+\lambda(\e+h_n),\quad t>-n,\ 0\le x<\ln(t+2n).$$
On the other hand, one has $c\lambda-\lambda^2<g'(0)$, whence there exist $k>0$ and
$\e>0$ small enough, and $n_0\in\N$ large enough, such that, for all $n\ge n_0$, the function
$\eta_n$ satisfies
$$\partial_t\eta_n-\partial_{xx}\eta_n-(c+\e)\partial_x\eta_n
\leq(g'(0)-k)\eta_n,\quad t>-n,\ 0\le x<\ln(t+2n).$$
Notice that $\eta_n$ is bounded by $1$, and thus, for $\beta>0$ small enough independent of
$n$ (large enough), the function $u_n$ defined by
$$u_n(t,x):=
\begin{cases}\beta\,\eta_n(t,x-(c+\e)t) & \text{if }t\geq-n,\ 0\leq x-(c+\e)t<\ln(t+2n)\\
 0 & \text{if }t\geq-n,\ x-(c+\e)t\ge\ln(t+2n).\end{cases}$$
is a generalized subsolution of $v_t=v_{xx}+g(v)$ in the domain $-n<t<0$, $x>(c+\e)t$. We
claim that~$\beta$ can be chosen in such a way that, for $n$ large enough, $u_n$ lies below
$u$ on the parabolic boundary of this~$(t,x)$-domain. For the initial time $t=-n$, we see that
$u_n(-n,x)\leq\beta\times\1_{[-(c+\e)n,-(c+\e) n+\ln n]}(x)$ for every~$x\geq-(c+\e)n$,
whereas, at the boundary $x=(c+\e) t$, $u_n(t,(c+\e) t)=\beta$ for all $-n<t<0$. On the other
hand, by~\eqref{eq:gamma} there exists $T<0$ such that
$$\inf_{t\leq T,\,(c+\e)t\leq x\leq(c+\e)t+\ln(-t)}u(t,x)>0.$$
Lastly, $\inf_{-T\le t\le 0}u(t,(c+\e)t)>0$ by the strong maximum principle, and thus the claim
follows. We can therefore apply the comparison principle and infer that, for $n$ large enough,~
$u_n(t,x)\leq u(t,x)$ for all~$-n\leq t\leq0$ and $x\geq (c+\e)t$. In particular,
$$\forall\,x\in[0,\ln(2n)),\quad u(0,x)\geq\beta\eta_n(0,x)
=\beta\left(1-\frac{x}{\ln(2n)}\right)^{\lambda\ln(2n)},$$
from which we eventually derive $u(0,x)\geq\beta e^{-\lambda x}$ for $x\geq0$, by letting
$n\to+\infty$.
\end{proof}

\begin{remark}{\rm One cannot expect to get in general a better lower bound for the
exponential decay rate of supersolutions than the one in \thm{exp}. Indeed,
if $g$ is a positive
constant, then for all $c\geq2\sqrt g$, the function $u$ defined by $u(t,x):=e^{-\lambda(x-ct)}$,
with $\lambda=(c+\sqrt{c^2-4g})/2$, satis\-fies~$u_t=u_{xx}+gu$. Instead, if one restricts to
transition fronts, the result is far to be optimal: if~$g$ is a KPP-type nonlinearity such that
$g(0)=g(1)=0$ and $0<g(s)\le g'(0)s$ for all~$s\in(0,1)$, standard traveling fronts $\phi(x-ct)$
with speed $c$ decay exponentially as $x\to+\infty$ with exponent~$(c-\sqrt{c^2-4g'(0)})/2$ in
the sense that $\lim_{x\to+\infty}(\ln\phi(x))/x=-(c-\sqrt{c^2-4g'(0)})/2$. This exponent
coincides with our bound only in the critical case $c=2\sqrt{g'(0)}$. The smallest
bound~$\sqrt{g'(0)}$ in \thm{exp} is provided by the slowest traveling front, i.e.~the critical
one, for which the bound is sharp. Then, as shown in Corollary~\ref{cor:exp} below, the
property of the critical front allows one to derive the same bound for all
other transition fronts, but it is not sharp in those cases.}
\end{remark}

\begin{corollary}\label{cor:exp}
Under the assumptions~\eqref{eq:f/u} and~\eqref{fpm}, any solution $0<u<1$ to~\eq{P=f}
satisfies~\eqref{eq:decay>}.
\end{corollary}

\begin{proof}
Suppose first that $u$ is the ``critical'' transition front of~\cite{n2},
introduced in Section \ref{sec223} for constructing a front with minimal past
and future speeds. In particular, letting
$X$ be the function for which~$u$ satisfies~\eqref{gtf}, the first condition in~\eqref{eq:gamma}
holds by~\eqref{gtf} and~\eqref{infsup}, and the second one holds with~$c=2\sqrt{\mu_-}$.
Moreover, by~\eqref{fpm}, for any $0<\e<1$, there exists $T_\e\in\R$ such that $u$ is a
supersolution of $u_t=u_{xx}+(1-\e)f_-(u)$ for $t\leq T_\e$ and $x\in\R$. Thus, for
$\tau\leq T_\e$, we can apply \thm{exp} to the function~$u(\.+\tau,\.)$ and, since
$c=2\sqrt{\mu_-}\ge2\sqrt{(1-\e)f'_-(0)}$, we infer that $\inf_{x\geq0}e^{\lambda x}u(\tau,x)>0$
for any $\lambda$ such that
$$\lambda>\frac{2\sqrt{\mu_-}+\sqrt{4\mu_--4(1-\e)\mu_-}}2=(1+\sqrt\e)\sqrt{\mu_- }.$$
It is easy to see that the property $\inf_{x\geq0}e^{\lambda x}u(\tau,x)>0$ for such a
$\lambda$ is preserved for $\tau>T_\e$, by comparing $u$ with the function
$k\,e^{-\lambda x}$ for $k>0$ small enough, on the domain $t\in(T_\e,\tau)$, $x>0$. Indeed,
such a function is a subsolution of~\eq{P=f} and $k>0$ can be chosen in such a way that this
function lies below $u$ on the parabolic boundary
$(\{T_\e\}\times[0,+\infty))\cup([T_\e,\tau]\times\{0\})$. Due to the arbitrariness of $\e\in(0,1)$,
this concludes the proof of the corollary in the case where $u$ is the critical transition front.

Let now $0<v<1$ be a solution to~\eq{P=f}. Fix $t\in\R$. Up to translating the critical transition
front $u$ in space, it is not restrictive to assume that $u(t,0)=v(t,0)$. Hence, the criticality
property recalled in the previous subsection yields $v(t,x)\geq u(t,x)$ for $x\geq0$,
whence~$v$ satisfies~\eqref{eq:decay>} because $u$ does.
\end{proof}


\subsection{Proof of~\eqref{cpm2bis}}\label{sec24}

In Theorem~\ref{th1}, we proved the existence of transition fronts connecting $0$ and $1$ for~\eq{P=f} and having asymptotic past and future speeds $c_{\pm}$ given by~\eqref{cpm1}. We prove here the stronger property~\eqref{cpm2bis}, except possibly when $\mu_+>\mu_-$ and the speeds $c_{\pm}$ satisfy $c_-=2\sqrt{\mu_-}$ and~$c_+=\sqrt{\mu_-}+\mu_+/\sqrt{\mu_-}$. Actually, in the case where $c_{\pm}>2\sqrt{\mu_\pm}$, the limits~\eqref{cpm2bis} follow immediately from the definitions~\eqref{defX1},~\eqref{defX2} and~\eqref{defX12} used in the proof of Proposition~\ref{pro:exsupercrit}. In the general case, which is treated here, the property will follow from the convergence~\eq{phi+-}.

Let us only prove the second limit in~\eqref{cpm2bis}, since the first one can
be shown similarly. Let~$u$ be a transition front constructed in
Theorem~\ref{th1} (in all cases except when $\mu_+>\mu_-$ and the speeds~$c_{\pm}$
satisfy $c_-=2\sqrt{\mu_-}$ and $c_+=\sqrt{\mu_-}+\mu_+/\sqrt{\mu_-}$)
and let $\xi:\R\to\R$ be a bounded function such
that~$u(t,X(t)+\xi(t)+\cdot)\to\phi_{c_+}$ in $C^2(\R)$ as $t\to+\infty$. As
done in Step~3 of the proof of Proposition~\ref{pro:exsupercrit}, for any
sequence $(t_n)_{n\in\N}$ in $\R$ diverging to~$+\infty$, the
functions~$(\tau,x)\mapsto u(t_n+\tau,X(t_n)+\xi(t_n)+x)$ converge, up to
extraction of a subsequence, locally uniformly in $\R\times\R$ to an entire
solution $0\le u_{\infty}(\tau,x)\le 1$ of~\eq{P=f} with nonlinearity $f_+$,
such that $u_{\infty}(0,x)=\phi_{c_+}(x)$.
By uniqueness of the solution of the Cauchy problem associated with this limiting equation,
one infers that~$u_{\infty}(\tau,x)
=\phi_{c_+}(x-c_+\tau)$ for all $\tau>0$ and $x\in\R$. Since the limit is uniquely determined, it follows that, for any $\tau>0$,
\be\label{convtau}
u(t+\tau,X(t)+\xi(t)+x)\to\phi_{c_+}(x-c_+\tau)\ \hbox{ as }t\to+\infty
\ee
locally uniformly in $x\in\R$ (and then uniformly in $\R$, by~\eqref{gtf},~\eqref{Xtau},~$\phi_{c_+}(-\infty)=1$ and~$\phi_{c_+}(+\infty)=0$). This limit, together with~\eq{phi+-} applied with $t+\tau$ and the fact that the function $\phi_{c_+}$ is decreasing, implies that, for any $\tau>0$,
$X(t+\tau)-X(t)+\xi(t+\tau)-\xi(t)\to c_+\tau$ as $t\to+\infty$,\footnote{Notice that this limit also holds for $\tau<0$, by setting $t'=t+\tau$ and writing $t=t'+|\tau|$.}
whence
\be\label{convtau2}
\limsup_{t\to+\infty}\big|X(t+\tau)-X(t)-c_+\tau\big|\le2\|\xi\|_{L^{\infty}(\R)}.
\ee

Assume now by contradiction that the second property in~\eqref{cpm2bis} does not hold. Since $X$ is locally bounded and since $X(s)/s\to c_+$ as $s\to+\infty$, this means that there exist $\epsilon>0$ and some sequences $(t_n)_{n\in\N}$ and $(\tau_n)_{n\in\N}$ of positive real numbers diverging to $+\infty$ and such that
\be\label{tntaun}
\left|\frac{X(t_n+\tau_n)-X(t_n)}{\tau_n}-c_+\right|\ge\epsilon\ \hbox{ for all }n\in\N.
\ee
Choose now $\tau>0$ such that $2\|\xi\|_{L^{\infty}(\R)}/\tau<\epsilon/2$. By~\eqref{convtau2}, let $T>0$ be such that
$$\big|X(t+\tau)-X(t)-c_+\tau\big|\le2\|\xi\|_{L^{\infty}(\R)}+\frac{\epsilon\tau}{2}\ \hbox{ for all }t\ge T,$$
and let $n_0\in\N$ such that $t_n\ge T$ for all $n\ge n_0$. For such $n$, write $\tau_n=k_n\tau+\tau'_n$ with $k_n\in\N$ and $0\le\tau'_n<\tau$. It follows that, for all $n\ge n_0$,
$$\baa{l}
\big|X(t_n+\tau_n)-X(t_n)-c_+\tau_n\big|\vspace{3pt}\\
\qquad\quad\displaystyle\le\big|X(t_n+\tau_n)-X(t_n+k_n\tau)-c_+\tau'_n\big|+\sum_{k=0}^{k_n-1}\big|X(t_n+(k+1)\tau)-X(t_n+k\tau)-c_+\tau\big|\vspace{3pt}\\
\qquad\quad\displaystyle\le\big|X(t_n+k_n\tau+\tau'_n)-X(t_n+k_n\tau)-c_+\tau'_n\big|+\Big(2\|\xi\|_{L^{\infty}(\R)}+\frac{\epsilon\tau}{2}\Big)k_n.\eaa$$
Since the sequence $(\tau'_n)_{n\ge n_0}$ is bounded and since $\tau_n\to+\infty$ and $k_n/\tau_n\to1/\tau$ as $n\to+\infty$, one infers from~\eqref{Xtau} that
$$\limsup_{n\to+\infty}\left|\frac{X(t_n+\tau_n)-X(t_n)-c_+\tau_n}{\tau_n}\right|\le\Big(2\|\xi\|_{L^{\infty}(\R)}+\frac{\epsilon\tau}{2}\Big)\times\frac{1}{\tau}=\frac{2\|\xi\|_{L^{\infty}(\R)}}{\tau}+\frac{\epsilon}{2}<\epsilon,$$
the last inequality being due to the choice of $\tau$. This contradicts~\eqref{tntaun} and the proof is thereby complete.~\hfill$\Box$


\SE{A priori bounds and asymptotic limits of transition fronts}\label{sec3}

This section is chiefly devoted to the proof of Theorem~\ref{th2} (done in Section~\ref{sec32}) on the optimality of the bounds~(\ref{cpm1}) for the asymptotic past and future speeds $c_{\pm}$ as $t\to\pm\infty$ of any transition front connecting $0$ and $1$ for~\eq{P=f}. We also show the existence of the asymptotic speeds and the convergence to some asymptotic profiles for any supercritical front. We first recall in Section~\ref{sec31} some useful results of~\cite{hn2,hr} on transition fronts in the case of homogeneous concave nonlinearities $f=f(u)$. Finally, Section~\ref{sec33} is devoted to the proof of Theorem~\ref{thm:decay}.


\subsection{Transition fronts in the time-independent case}\label{sec31}

In this section, we focus on a particular time-independent version of~\eq{P=f}. Namely, let~$g:[0,1]\to\R$ be any~$C^2$ concave function such that $g(0)=g(1)=0$ and $g(u)>0$ for all $u\in(0,1)$. Consider the equation~\eq{P=f} with $f(t,u)=g(u)$, that is
\be\label{eqg}
u_t=u_{xx}+g(u),\ \ t\in\R,\ x\in\R.
\ee\par
For~(\ref{eqg}), standard traveling fronts~$\varphi_c(x-ct)$ such that $\varphi_c(-\infty)=1>\varphi_c>\varphi_c(+\infty)=0$ exist if and only if $c\ge c^*:=2\sqrt{g'(0)}$, see~\cite{aw,kpp}. Furthermore, the functions $\varphi_c$ are decreasing, unique up to shifts and one can assume without loss of generality that they satisfy
\be\label{asymvarphi}\left\{\baa{ll}
\varphi_c(\xi)\sim e^{-\lambda_c\xi} & \hbox{for }c>c^*\vspace{3pt}\\
\varphi_{c^*}(\xi)\sim \xi\,e^{-\lambda_{c^*}\xi} & \hbox{for }c=c^*\eaa\right.\hbox{ as }\xi\to+\infty,
\ee
where
\be\label{lambdac}
\lambda_c=\frac{c-\sqrt{c^2-4g'(0)}}{2}\ \hbox{ for }c\ge c^*.
\ee
Notice in particular that $\lambda_{c^*}=\sqrt{g'(0)}=c^*/2=:\lambda^*$. With the normalization~(\ref{asymvarphi}), it is known that~$\varphi_c(\xi)\le e^{-\lambda_c\xi}$ for all $c>c^*$ and for all $\xi\in\R$. Lastly, let $\theta:\R\to(0,1)$ be the unique solution of~$\theta'(t)=g(\theta(t)),\ t\in\R$ such that $\theta(t)\sim e^{g'(0)t}$ as $t\to-\infty$.\par
The standard traveling fronts $\varphi_c(\pm x-ct)$ are entire solutions ranging in $(0,1)$ and they are the keystones in the construction of many other solutions. More precisely, following~\cite{hn2}, let $\Psi$ be the bijection defined by
$$\baa{rcl}
\Psi:[-\lambda^*,\lambda^*]=[-\sqrt{g'(0)},\sqrt{g'(0)}] & \to & \mathcal{X}:=\big(\R\backslash(-c^*,c^*)\big)\cup\{\infty\}\vspace{3pt}\\
\lambda\neq 0 & \mapsto & \displaystyle\lambda+\frac{g'(0)}{\lambda}\vspace{3pt}\\
\lambda=0 & \mapsto & \infty\eaa$$
and let us endow $\mathcal{X}$ with the topology induced by the image by $\Psi$ of the Borel topology of~$[-\lambda^*,\lambda^*]$. In other words, a subset $O$ of $\mathcal{X}$ is open if $\Psi^{-1}(O)$ is open relatively in $[-\lambda^*,\lambda^*]$. Now, let $\mathcal{M}$ be the set of all nonnegative Borel measures~$\mu$ on~$\mathcal{X}$ such that~$0<\mu(\mathcal{X})<+\infty$. It follows then from Theorem~1.2 of~\cite{hn2} and formula~(30) of~\cite{hn2} that there is a one-to-one map
$$\mu\mapsto u_{\mu}$$
from $\mathcal{M}$ to the set of solutions $0<u<1$ of~(\ref{eqg}). Furthermore, for each~$\mu\in\mathcal{M}$, calling $M=\mu\big(\mathcal{X}\backslash\{-c^*,c^*\}\big)$, the solution $u_{\mu}$ satisfies
\be\label{umu}\baa{l}
\max\Big(\varphi_{c^*}\big(x-c^*t-c^*\ln\mu(c^*)\big),\,\varphi_{c^*}\big(\!-x-c^*t-c^*\ln\mu(-c^*)\big),\vspace{3pt}\\
\qquad\displaystyle\ \ M^{-1}\!\!\int_{\R\backslash[-c^*,c^*]}\!\!\varphi_{|c|}\big(({\rm{sgn}}\,c)x\!-\!|c|t\!-\!|c|\ln M\big)\,d\mu(c)\,+\,M^{-1}\theta(t\!+\!\ln M)\,\mu(\infty)\Big)\vspace{3pt}\\
\le\,u_{\mu}(t,x)\,\le\,\varphi_{c^*}\big(x-c^*t-c^*\ln\mu(c^*)\big)+\varphi_{c^*}\big(\!-x-c^*t-c^*\ln\mu(-c^*)\big)\vspace{3pt}\\
\qquad\qquad\qquad\displaystyle+M^{-1}\!\!\int_{\R\backslash[-c^*,c^*]}\!\!e^{-\lambda_{|c|}(({\rm{sgn}}\,c)x-|c|t-|c|\ln M)}d\mu(c)\,+\,M^{-1}e^{g'(0)(t+\ln M)}\mu(\infty)\eaa
\ee
for all $(t,x)\in\R^2$, under the convention that the terms involving $M$ are
not present if~$M=0$, that~$\mu(\pm c^*)=\mu\big(\{\pm c^*\}\big)$ and
$\mu(\infty)=\mu\big(\{\infty\}\big)$, and that $\ln 0=-\infty$. The
estimate~(\ref{umu}) reflects diffe\-rent types of contributions weighted by
$\mu$: critical standard fronts, supercritical standard fronts and the spatially
homogeneous solution. The solutions~$u_{\mu}$ are decreasing (resp. increasing)
with respect to $x$ if~$\mu\big((-\infty,-c^*]\cup\{\infty\}\big)=0$, that is,
if $u_{\mu}$ is a measurable interaction of right-moving spatially decreasing
traveling fronts $\varphi_c(x-ct)$ (resp.
if~$\mu\big([c^*,+\infty)\cup\{\infty\}\big)=0$, that is, if~$u_{\mu}$ is a
measurable interaction of left-moving spatially increasing traveling
fronts~$\varphi_{|c|}(-x-|c|t)$). Lastly, we point out that these solutions~$u_{\mu}$
almost describe the set of all solutions of~(\ref{eqg}). Indeed, on
the one hand, it follows from~\cite{aw} that any solution~$0<u<1$ of~(\ref{eqg})
is such that~$\max_{[-
c|t|,c|t|]}u(t,\cdot)\to0$ as $t\to-\infty$ for every $c\in[0,c^*)$, while, on
the other hand, the following almost-uniqueness result was proved in~\cite{hn2}:
if a solution~$0<u(t,x)<1$ of~(\ref{eqg}) is such that
\be\label{hyphn2}
\exists\,c>c^*,\ \ \max_{[-c|t|,c|t|]}u(t,\cdot)\to0\ \hbox{ as }t\to-\infty,
\ee
then there is a measure~$\mu\in\mathcal{M}$ such that $u=u_{\mu}$ and the support of $\mu$ does not intersect the interval $(-c,c)$. For the one-dimensional equation~(\ref{eqg}), it is conjectured that any solution~$0<u(t,x)<1$ is of the type~$u_{\mu}$, even without~(\ref{hyphn2}).\par
In~\cite{hr} (Theorems~1.11 and 1.14), we showed a necessary and sufficient condition for a solution of the type $u_{\mu}$ to be a transition front connecting~$0$ and~$1$ for problem~(\ref{eqg}), and we characterized the asymptotic past and future speeds in this case:

\begin{theorem}\label{thsupport}{\rm{\cite{hr}}}
Let $g:[0,1]\to\R$ be any $C^2$ concave function such that $g(0)=g(1)=0$ and~$g>0$ on the interval $(0,1)$.\par
{\rm{(i)}} Under the above notations, a solution $u_{\mu}$ of~$(\ref{eqg})$
associated with a measure~$\mu\in\mathcal{M}$ is a transition front connecting
$0$ and $1$ if and only if the support of~$\mu$ is bounded and is included
in~$[c^*,+\infty)=[2\sqrt{g'(0)},+\infty)$.\par
{\rm{(ii)}} Assume here that the support of $\mu$ is compactly included in $[c^*,+\infty)$ and let $c_-$ and~$c_+$ denote the leftmost and rightmost points of the support of $\mu$. Then $u_{\mu}$ has asymptotic past and future speeds equal to~$c_{\pm}$ in the sense of~\eqref{gtf} and~\eqref{cpm2}. Furthermore, if $c_->c^*$, then there is a bounded function~$\xi:\R\to\R$ such that $u_{\mu}(t,X(t)+\xi(t)+\cdot)\to\varphi_{c_{\pm}}$ in $C^2(\R)$ as $t\to\pm\infty$.
\end{theorem}

It has also been proved in~\cite{hr} that, under the assumptions of part~(ii)
above, there holds $\limsup_{t\to-\infty}|X(t)-c_-t|<+\infty$
if~$\mu(c_-):=\mu\big(\{c_-\}\big)>0$ and $X(t)-c_-t\to-\infty$ as~$t\to-\infty$
if $\mu(c_-)=0$, while $\limsup_{t\to+\infty}|X(t)-c_+t|<+\infty$
if~$\mu(c_+):=\mu\big(\{c_+\}\big)>0)$ and $X(t)-c_+t\to-\infty$
as~$t\to+\infty$ if~$\mu(c_+)=0$. On the other hand, Theorem~\ref{thsupport}
also implies that not all
solutions $0<u(t,x)<1$ of~\eqref{eqg} such that $u(t,-\infty)=1$ and
$u(t,+\infty)=0$ for all $t\in\R$ are transition fronts: namely, any solution of
the type $u=u_{\mu}$ for which the support of $\mu$ is included in
$[c^*,+\infty)$ but is not compact satisfies the above limits (roughly speaking,
for such solutions, the transition region between $0$ and $1$ is not uniformly
bounded in time). As another corollary of Theorem~\ref{thsupport} (see
Theorem~1.6~in~\cite{hr}), it follows that transition fronts connecting
$0$ and $1$ for~\eqref{eqg} and having asymptotic past and future speeds
$c_{\pm}$ exist if and only if $c^*\le c_-\le c_+<+\infty$. For further results
and comments on the transition fronts for homogeneous equation~\eqref{eqg}, we
refer to~\cite{hr}.\par
To complete this subsection, we include an additional comparison result
(see Proposition~2.1 in~\cite{hr}) which has its own interest and will be used
in particular in the proof of Theorem~\ref{th2} in Section~\ref{sec32}. It
provides uniform lower and upper bounds of a solution $u_{\mu}$ of~(\ref{eqg})
on the left and right of its level sets, when the measure~$\mu\in\mathcal{M}$ is
compactly supported in $[c^*,+\infty)$. These bounds say that the~$u_{\mu}$ is
steeper than the standard front associated to any speed larger than its support.

\begin{proposition}\label{pro3}{\rm{\cite{hr}}}
Let $g:[0,1]\to\R$ be any $C^2$ concave function such that $g(0)=g(1)=0$
and~$g>0$ on the interval $(0,1)$. Let $\mu$ be any measure in~$\mathcal{M}$
that is supported in~$[c^*,\gamma]=[2\sqrt{g'(0)},\gamma]$ for
some~$\gamma\in[c^*,+\infty)$, and let $0<u_{\mu}<1$ be the solution of~\eqref{eqg}
that is associated to the measure $\mu$. Then, for every $(t,y)\in\R^2$,
\be\label{umusteeper3}\left\{\baa{ll}
u_{\mu}(t,y+x)\ge\varphi_{\gamma}\big(\varphi_{\gamma}^{-1}(u_{\mu}(t,y))+x\big)
& \hbox{for all }x\le 0,\vspace{3pt}\\
u_{\mu}(t,y+x)\le\varphi_{\gamma}\big(\varphi_{\gamma}^{-1}(u_{\mu}(t,
y))+x\big)
& \hbox{for all }x\ge 0,\eaa\right.
\ee
where $\varphi_{\gamma}^{-1}:(0,1)\to\R$ denotes the reciprocal of the function $\varphi_{\gamma}$.
\end{proposition}


\subsection{Proof of Theorem~\ref{th2}}\label{sec32}

In this section, $f:\R\times[0,1]\to\R$ is any function satisfying the assumptions of Theorem~\ref{th2} and $u$ denotes any transition front connecting $0$ and $1$ for problem~\eq{P=f}, satis\-fying~(\ref{gtf}) for some~$X:\R\to\R$. First of all, it follows from~\eqref{Xtau} and Proposition~\ref{pro:spreading} that
\be\label{boundedspeed}
2\sqrt{\mu_\pm}\le c_{\pm}:=\liminf_{t\to\pm\infty}\frac{X(t)}{t}\le\limsup_{t\to\pm\infty}\frac{X(t)}{t}<+\infty.
\ee
The general strategy to prove Theorem~\ref{th2} will be to compare $u$ on some time-intervals of the type~$(-\infty,\tau]$ with the solution $v$ of the homogeneous reaction-diffusion equation~\eqref{eqg} with~$g=f_-$. Then we derive a uniform exponential lower bound for $v$, and hence for $u$, by applying some results of Section~\ref{sec31}. Furthermore, the results on the qualitative properties of solutions of~\eqref{eqg} with $g=f_-$ will provide the exact exponential decay rate of $u(t,\cdot)$ at $+\infty$ if~$\liminf_{t\to-\infty}X(t)/t>2\sqrt{\mu_-}$. Lastly, by comparing $u$, for large positive times, with the solutions of some homogeneous reaction-diffusion equations with nonlinearities close to~$f_+$, we get the bound~(\ref{cpm2ter}) for~$c_+$, and some passages to the limits will yield~\eq{phi+-} in the supercritical case.

\subsubsection*{Step 1: construction of a solution $v$ of a time-independent equation}

For every $n\in\N$, let $v_n$ be the solution of the Cauchy problem
\be\label{defvn}\left\{\baa{rcll}
(v_n)_t & = & (v_n)_{xx}+f_-(v_n), & t>-n,\ x\in\R,\vspace{3pt}\\
v_n(-n,x) & = & u(-n,x), & x\in\R.\eaa\right.
\ee
Since $0<u<1$ in $\R^2$ and $f_-(0)=f_-(1)=0$, the maximum principle yields $0<v_n(t,x)<1$ for all~$n\in\N$ and $(t,x)\in[-n,+\infty)\times\R$. From standard parabolic estimates, it follows that, up to extraction of a subsequence, the functions $v_n$ converge locally uniformly in $\R^2$ to a solution $v$ of the homogeneous equation
\be\label{eqv}
v_t=v_{xx}+f_-(v),\ \ t\in\R,\ x\in\R
\ee
such that $0\le v(t,x)\le 1$ for all $(t,x)\in\R^2$.

\subsubsection*{Step 2: comparisons between $v$ and $u$}

Let us recall here the existence of a continuous $L^1(-\infty,0)$ function
$\zeta$ such that the assumption~(\ref{hypf-}) holds. Furthermore, since the
$C^1(\R\times[0,1])$ function $f$ vanishes on $\R\times\{0,1\}$ and since $f_-$
is positive in $(0,1)$ and concave in $[0,1]$ (one has in particular
$f'_-(0)>0>f'_-(1)$), it follows, as said in Remark~\ref{hypf+-},  that the
function
$\zeta_-:t\mapsto\sup_{s\in(0,1)}\big|f(t,s)/f_-(s)-1\big|$ is continuous.
Therefore, even if it means changing $\zeta$ into~$\zeta_-$ on the whole $\R$,
one can assume without loss of generality that there exists a nonnegative
continuous function~$\zeta:\R\to\R$ such that~(\ref{hypf-}) holds for all
$t\in\R$, that is
\be\label{hypf-bis}
(1-\zeta(t))\,f_-(s)\le f(t,s)\le(1+\zeta(t))\,f_-(s)\hbox{ for all }t\in\R\hbox{ and }s\in[0,1],
\ee
and $\zeta\in L^1(-\infty,\tau)$ for all $\tau\in\R$. Let us now define
\be\label{defTheta}
\Theta(t)=\int_{-\infty}^t\zeta(\tau)d\tau\ \hbox{ for }t\in\R.
\ee

\begin{lemma}\label{lemuv}
There holds
\be\label{comparuv}
u(t,x)\,e^{-\mu_-\Theta(t)}\le v(t,x)\le u(t,x)\,e^{\mu_-\Theta(t)}\ \hbox{ for all }(t,x)\in\R^2.
\ee
\end{lemma}

\noindent{\bf{Proof.}} For every $n\in\N$, let $\varphi_n$ be the function defined in $[-n,+\infty)$ by $\varphi_n(t)=\exp\big(\mu_-\int_{-n}^t\zeta(\tau)d\tau\big)$.
For the proof of the upper bound in~(\ref{comparuv}), denote $w_n(t,x)\!=\!u(t,x)\varphi_n(t)$ for~$(t,x)\!\in\![-n,+\infty)\times\R$. Let us check that the function $w_n$ is a supersolution for the equation~(\ref{defvn}) satis\-fied by~$v_n$ in~$[-n,+\infty)\times\R$. First of all, we observe that $w_n(-n,x)=u(-n,x)=v_n(-n,x)$ for all~$x\in\R$, and that $v_n<1$ in $[-n,+\infty)\times\R$. For every $(t,x)\in(-n,+\infty)\times\R$ such that $w_n(t,x)<1$, there holds
$$\baa{rcl}
(w_n)_t(t,x)\!-\!(w_n)_{xx}(t,x)\!-\!f_-(w_n(t,x)) & \!\!\!=\!\!\! & u_t(t,x)\varphi_n(t)\!+\!u(t,x)\varphi_n'(t)\!-\!u_{xx}(t,x)\varphi_n(t)\!-\!f_-\big(u(t,x)\,\varphi_n(t)\big)\vspace{3pt}\\
& \!\!\!=\!\!\! & f\big(t,u(t,x)\big)\,\varphi_n(t)+\mu_-\,\zeta(t)\,u(t,x)\,\varphi_n(t)-f_-\big(u(t,x)\,\varphi_n(t)\big).\eaa$$
But, for such $(t,x)$, one has $f(t,u(t,x))\ge(1-\zeta(t))\,f_-(u(t,x))$ by~(\ref{hypf-bis}), while
$$f_-\big(u(t,x)\,\varphi_n(t)\big)\le f_-\big(u(t,x)\big)\,\varphi_n(t)\ \hbox{ and }\ f_-(u(t,x))\le f'_-(0)u(t,x)=\mu_-u(t,x)$$
since $0<u(t,x)\le u(t,x)\varphi_n(t)<1$ and $s\mapsto f_-(s)/s$ is nonincreasing on $(0,1)$. Therefore,
$$\baa{rcl}
(w_n)_t(t,x)\!-\!(w_n)_{xx}(t,x)\!-\!f_-(w_n(t,x)) & \!\!\!\ge\!\!\! & (1-\zeta(t))f_-(u(t,x))\varphi_n(t)\!+\!\mu_-\zeta(t)u(t,x)\varphi_n(t)\!-\!f_-(u(t,x))\varphi_n(t)\vspace{3pt}\\
& \!\!\!=\!\!\! & \big(\mu_-u(t,x)-f_-(u(t,x))\big)\,\zeta(t)\,\varphi_n(t)\vspace{3pt}\\
& \!\!\!\ge\!\!\! & 0\eaa$$
for every $(t,x)\in(-n,+\infty)\times\R$ such that $w_n(t,x)<1$. It follows then from the maximum principle that $w_n(t,x)\ge v_n(t,x)$ for all~$(t,x)\in[-n,+\infty)\times\R$, that is
$v_n(t,x)\le u(t,x)\,\exp\big(\mu_-\int_{-n}^t\zeta(\tau)d\tau\big)$ for all~$(t,x)\in[-n,+\infty)\times\R$.
By passing to the limit as $n\to+\infty$ and using the definition~(\ref{defTheta}), one infers that~$v(t,x)\le u(t,x)\,e^{\mu_-\Theta(t)}$ for all $(t,x)\in\R^2$, that is, the upper bound in~\eqref{comparuv} has been shown.\par
The proof of the lower bound in~(\ref{comparuv}) uses a similar method. Namely, we set $z_n(t,x)=v_n(t,x)\,\varphi_n(t)$ for $(t,x)\in[-n,+\infty)\times\R$ and we will check that~$z_n$ is a supersolution for the equation satisfied by $u$ in~$[-n,+\infty)\times\R$. First of all, we note that~$z_n(-n,x)=v_n(-n,x)=u(-n,x)$ for all $x\in\R$ and that $u<1$. Moreover, for every~$(t,x)\in(-n,+\infty)\times\R$ such that $z_n(t,x)<1$, there holds
$$\baa{l}
(z_n)_t(t,x)\!-\!(z_n)_{xx}(t,x)\!-\!f(t,z_n(t,x))\vspace{3pt}\\
\qquad\qquad=(v_n)_t(t,x)\,\varphi_n(t)\!+\!v_n(t,x)\,\varphi_n'(t)\!-\!(v_n)_{xx}(t,x)\,\varphi_n(t)-f\big(t,v_n(t,x)\,\varphi_n(t)\big)\vspace{3pt}\\
\qquad\qquad=f_-\big(v_n(t,x)\big)\,\varphi_n(t)+\mu_-\zeta(t)\,v_n(t,x)\,\varphi_n(t)-f(t,v_n(t,x)\,\varphi_n(t)).\eaa$$
But, as it was done for $w_n$ in the last paragraph,
$$\baa{rcl}
f\big(t,v_n(t,x)\,\varphi_n(t)\big) & \le & (1+\zeta(t))\,f_-\big(v_n(t,x)\,\varphi_n(t)\big)\vspace{3pt}\\
& \le & f_-\big(v_n(t,x)\,\varphi_n(t)\big)+\zeta(t)\,\mu_-\,v_n(t,x)\,\varphi_n(t)\vspace{3pt}\\
& \le & f_-\big(v_n(t,x)\big)\,\varphi_n(t)+\zeta(t)\,\mu_-\,v_n(t,x)\,\varphi_n(t),\eaa$$
whence $(z_n)_t(t,x)-(z_n)_{xx}(t,x)-f(t,z_n(t,x))\ge0$ for every $(t,x)\in(-n,+\infty)\times\R$ such that $z_n(t,x)<1$. It follows then from the maximum principle that $z_n(t,x)\ge u(t,x)$ for all~$(t,x)\in[-n,+\infty)\times\R$, that is,~$v_n(t,x)\ge u(t,x)\,\exp\big(\!\!-\mu_-\!\int_{-n}^t\zeta(\tau)d\tau\!\big)$ for all $(t,x)\in[-n,+\infty)\times\R$.
By passing to the limit as $n\to+\infty$ and using the definition~(\ref{defTheta}), one concludes that~$v(t,x)\ge u(t,x)\,e^{-\mu_-\Theta(t)}$ for all $(t,x)\in\R^2$ and the proof of Lemma~\ref{lemuv} is thereby complete.~\hfill$\Box$

\begin{remark}{\rm
One could wonder whether a comparison similar to that of Lemma~\ref{lemuv} would hold or not for the functions $\widetilde{u}:=1-u$ and $\widetilde{v}:=1-v$. Actually, such a comparison between~$\widetilde{u}$ and~$\widetilde{v}$ is not clear. Indeed, for instance, the function $\widetilde{v}$ obeys $\widetilde{v}_t=\widetilde{v}_{xx}+g_-(\widetilde{v})$ in~$\R^2$ where~$g_-(s)=-f_-(1-s)$ on $[0,1]$. But $0$ is a stable point of $g_-$ and the arguments used in the proof of Lemma~\ref{lemuv} to compare $u$ and $v$ do not work as such for $\widetilde{u}$ and $\widetilde{v}$. As a matter of fact, a comparison of the type~(\ref{comparuv}) for $1-u$ and $1-v$ is not needed, since only the exponential decay of $u(t,x)$ and $v(t,x)$ as $x\to+\infty$ will determine the asymptotic speeds of~$u$ and $v$, as will be shown in the following steps.}
\end{remark}

\subsubsection*{Step 3: $v$ is a transition front for~\eqref{eqv} in $(-\infty,\tau]\!\times\!\R$ for all $\tau\!\in\!\R$ with the same family~$(X(t))_{t\le\tau}$}

Since~$0<u<1$ in $\R^2$ and $u(t,+\infty)=0$ for every $t\in\R$, it follows immediately from Lemma~\ref{lemuv} that~$0<v(t_0,x_0)<1$ for at least a point $(t_0,x_0)\in\R^2$. Since~$0\le v\le 1$ in $\R^2$, the strong maximum principle actually implies that $0<v(t,x)<1$ for all $(t,x)\in\R^2$. The estimates~(\ref{comparuv}) and the stability of~$1$ will then imply that~$v$ is a transition front connecting~$0$ and~$1$ for~(\ref{eqv}) with the same family $(X(t))$ as $u$ in any set of the type~$(-\infty,\tau]\times\R$, in the sense of the following lemma.

\begin{lemma}\label{lemvfront}
For every $\tau\in\R$, there holds
$$\left\{\baa{ll}
v(t,X(t)+x)\to 1 & \hbox{as }x\to-\infty\vspace{3pt}\\
v(t,X(t)+x)\to 0 & \hbox{as }x\to+\infty\eaa\right.\hbox{ uniformly in }t\in(-\infty,\tau].$$
\end{lemma}

\noindent{\bf{Proof.}} First of all, since the function $\Theta$ given in~(\ref{defTheta}) is bounded in $(-\infty,\tau]$, it follows immediately from~(\ref{gtf}) and Lemma~\ref{lemuv} that $v(t,X(t)+x)\to0$ as $x\to+\infty$, uniformly in~$t\in(-\infty,\tau]$.\par
In order to show the other part of the conclusion, let us assume on the contrary
that there are $\epsilon>0$ and a sequence~$(t_n,x_n)_{n\in\N}$ in
$(-\infty,\tau]\times\R$ such that $x_n\to-\infty$ as $n\to+\infty$ and
$v(t_n,X(t_n)+x_n)\le 1-\epsilon$ for all $n\in\N$. Define, for all $n\in\N$ and
$(t,x)\in\R^2$, $\tilde{v}_n(t,x)=v(t+t_n,x+X(t_n)+x_n)$. Up to extraction of a
subsequence, the functions $\tilde{v}_n$ converge locally uniformly in $\R^2$ to
a solution $v_{\infty}$ of~(\ref{eqv}) such that~$0\le v_{\infty}\le 1$ in
$\R^2$ and~$v_{\infty}(0,0)\le1-\epsilon.$ On the
other hand, since $x_n\to-\infty$ as~$n\to+\infty$, it follows from~(\ref{gtf})
and~\eqref{Xtau} that
$$u(t+t_n,x+X(t_n)+x_n)\to1\ \hbox{ as }n\to+\infty,\hbox{ locally uniformly in }(t,x)\in\R^2.$$
Since $\Theta$ is bounded in $(-\infty,\tau]$, whence $\Theta\le A_{\tau}$ in
$(-\infty,\tau]$ for some $A_{\tau}\in\R$, the first inequality in~\eqref{comparuv}
yields
$v_{\infty}(t,x)\ge e^{-\mu_-A_{\tau}}$ for all
$(t,x)\in(-\infty,0]\times\R$.
One infers from the maximum principle that, for every
$s<t\leq0$ and $x\in\R$, one has $v_{\infty}(t,x)\ge\rho(t-s)$, where
$\rho:\R\to(0,1)$ denotes the solution of the ordinary differential equation
$\rho'=f_-(\rho)$ in~$\R$ with $\rho(0)=e^{-\mu_-A_{\tau}}\in(0,1)$.
Since~$\rho(+\infty)=1$, one concludes, by passing to the limit as
$s\to-\infty$, that $v_{\infty}(t,x)\ge 1$ (and then $=1$) for all
$(t,x)\in(-\infty,0]\times\R$.
This contradicts~$v_{\infty}(0,0)\le1-\epsilon<1$. Finally, the
proof of Lemma~\ref{lemvfront} is complete.\hfill$\Box$

\subsubsection*{Step 4: completion of the proof of Theorem~$\ref{th2}$ in the case $c_-=2\sqrt{\mu_-}$}

In this case, remembering~(\ref{boundedspeed}), it only remains to show the second assertion in~(\ref{cpm2ter}), namely
$$c_+:=\liminf_{t\to+\infty}\frac{X(t)}{t}\ge\kappa+\frac{\mu_+}{\kappa}$$
where
$\kappa=\min\big(\sqrt{\mu_+},(c_--\sqrt{c_-^2-4\mu_-})/2\big)=\min\big(\sqrt{
\mu_+},\sqrt{\mu_-}\big)=\sqrt{\ul\mu}$. In the case~$0<\mu_+\le\mu_-$,
then~$\kappa=\sqrt{\mu_+}$, while $c_+\ge2\sqrt{\mu_+}=\kappa+\mu_+/\kappa$ by
Proposition~\ref{pro:spreading}. Hence, the proof is done in this case.\par
Consider now the case~$0<\mu_-<\mu_+$. In other words, $\kappa=\sqrt{\mu_-}$
and one shall show in this case that~$c_+\ge\sqrt{\mu_-}+\mu_+/\sqrt{\mu_-}$.
Let $\epsilon$ be any real number such that $0<\epsilon<1$ and
$(1-\epsilon)\,\mu_+>\mu_-$, and let $T_\e\in\R$ be such
that~$f(t,s)\ge(1-\epsilon)f_+(s)$ for all $(t,s)\in[T_\e,+\infty)\times[0,1]$.
One makes use of the estimate~\eq{decay>} provided by Corollary \ref{cor:exp},
which, applied with $t=T_\e$ and $\lambda=\kappa$
satisfying $\sqrt{\mu_-}<\kappa<\sqrt{(1-\e)\mu_+}$,
allows one to derive~\eqref{ckappaeps}, as done in Section~\ref{sec223}.
The desired result follows by letting $\kappa\to\sqrt{\mu_-}$ and then~$\e\to0$
in~\eqref{ckappaeps}. To sum up, the proof of Theorem~\ref{th2} is
complete in the case $c_-=2\sqrt{\mu_-}$.

\subsubsection*{Step 5: completion of the proof of Theorem~$\ref{th2}$ if $c_->2\sqrt{\mu_-}$}

In this case, one has to show that the asymptotic past and future speeds of $u$
exist and are equal to~$c_{\pm}$, in the sense that~$X(t)\sim c_{\pm}t$ as
$t\to\pm\infty$, and that the solution $u$ converges to some well-identified
profiles along its level sets as $t\to\pm\infty$. We assume throughout this subsection
that $c_->2\sqrt{\mu_-}$ and that~\eqref{hypf+} holds.
The proof of Step~5 is itself divided into several
substeps. Under the notations of Section~\ref{sec31}, we first identify a
measure $\mu$ such that $v=u_{\mu}$ for the equation~(\ref{eqv}), and we show
some properties of the support of $\mu$. Then, we identify the leftmost point of
the support of $\mu$ and we apply Theorem~\ref{thsupport} to determine the
asymptotic past speed of~$u$ and its asymptotic behavior along its level sets.
Lastly, we will compare~$u$ at large positive times to a solution of the
homogeneous equation~(\ref{eqv}) with nonlinearity~$f_+$ instead of~$f_-$ and we
will again apply Theorem~\ref{thsupport} to determine the asymptotic future
speed of $u$.\hfill\break

\noindent{\it{Substep 5.1: identification and elementary properties of a measure $\mu$ such that $v=u_{\mu}$}}\hfill\break

\noindent{}Since $c_-:=\liminf_{t\to-\infty}X(t)/t>2\sqrt{\mu_-}$,
property~(\ref{gtf}) implies that $\sup_{[ct,+\infty)}u(t,\cdot)\to0$
as~$t\to-\infty$ for every $2\sqrt{\mu_-}<c<c_-$. Hence,
$\sup_{[ct,+\infty)}v(t,\cdot)\to0$ as $t\to-\infty$
from Lemma~\ref{lemuv} and the fact that the function $\Theta$ is bounded in,
say, $(-\infty,0]$.
In particular, $v$ satisfies \eqref{hyphn2} with
$c^*=2\sqrt{\mu_-}=2\sqrt{f_-'(0)}$ and thus, since it is a time-global
solution of~(\ref{eqv}) and $f_-$ is a $C^2$ concave function such that
$f_-(0)=f_-(1)=0$ and $f_->0$ on $(0,1)$, it follows from Theorem~1.4
of~\cite{hn2} that, for problem~(\ref{eqv}), i.e.~(\ref{eqg}) with~$g=f_-$, $v$
can be represented as $v=u_{\mu}$ with $\mu\in\mathcal{M}$ supported in
$(-\infty,c]\cup[c,+\infty)\cup\{\infty\}$ for every $2\sqrt{\mu_-}<c<c_-$.
Therefore, $\mu$ is supported in~$(-\infty,c_-]\cup[c_-,+\infty)\cup\{\infty\}$
and, in this case, the inequalities~(\ref{umu}) amount to
$$\baa{l}
\displaystyle M^{-1}\!\!\int_{(-\infty,-c_-]\cup[c_-,+\infty)}\!\!\varphi^-_{|c|}\big(({\rm{sgn}}\,c)x-|c|t-|c|\ln M\big)\,d\mu(c)+M^{-1}\theta^-(t+\ln M)\,\mu(\infty)\vspace{3pt}\\
\qquad\qquad\le\,v(t,x)\,\le\,\displaystyle M^{-1}\!\!\int_{(-\infty,c_-]\cup[c_-,+\infty)}\!\!e^{-\lambda^-_{|c|}(({\rm{sgn}}\,c)x-|c|t-|c|\ln M)}\,d\mu(c)\,+\,M^{-1}e^{\mu_-(t+\ln M)}\mu(\infty)\eaa$$
for all $(t,x)\in\R^2$, where $M=\mu\big((-\infty,c_-]\cup[c_-,+\infty)\cup\{\infty\}\big)>0$ and the functions $\varphi^-_c(x-ct)$ with~$c\ge c^*_-=2\sqrt{f'_-(0)}=2\sqrt{\mu_-}$ denote the traveling fronts connecting $0$ and $1$ for~(\ref{eqv}) with the normalization~\eqref{asymvarphi} and
\be\label{lambdac-}
\lambda^-_c=\frac{c-\sqrt{c^2-4\mu_-}}{2}
\ee
in place of $\lambda_c$ in~\eqref{lambdac}. Furthermore,~$\theta^-:\R\to(0,1)$ denotes the unique solution of $(\theta^-)'(t)=f_-(\theta^-(t))$ in $\R$ such that~$\theta^-(t)\sim e^{\mu_-t}$ as $t\to-\infty$. Now, if $\mu\big((-\infty,-c_-]\cup\{\infty\}\big)>0$, then Lebesgue's dominated convergence theorem implies that $\liminf_{x\to+\infty}v(t,x)\ge M^{-1}\mu\big((-\infty,-c_-]\big)+M^{-1}\theta^-(t+\ln M)\mu(\infty)>0$ for every $t\in\R$, which is ruled out by Lemma~\ref{lemvfront}. Therefore, $\mu\big((-\infty,-c_-]\cup\{\infty\}\big)=0$, that is, the support of $\mu$ is included in $[c_-,+\infty)$ and
\be\label{ineqvmu}
M^{-1}\!\!\int_{[c_-,+\infty)}\!\!\varphi^-_{c}\big(x-ct-c\ln M\big)\,d\mu(c)\le v(t,x)\le M^{-1}\!\!\int_{[c_-,+\infty)}\!\!e^{-\lambda^-_c(x-ct-c\ln M)}\,d\mu(c)
\ee
for all $(t,x)\in\R^2$, with $M=\mu\big([c_-,+\infty)\big)>0$.\par
The following lemma gives an exponential lower bound of $u(t,x)$ and $v(t,x)$ as $x\to+\infty$ in terms of the support of $\mu$. It will used several times in the sequel.

\begin{lemma}\label{lemuvexp} If $\mu\big([c,+\infty)\big)>0$ for some $c\ge c_-$, then
$\liminf_{x\to+\infty}\big(e^{\lambda^-_cx}u(t,x)\big)>0$ and
$\liminf_{x\to+\infty}\big(e^{\lambda^-_cx}v(t,x)\big)>0$ for every $t\in\R$,
where $\lambda^-_c$ is given in~\eqref{lambdac-}.
\end{lemma}

\begin{remark}{\rm
Lemma \ref{lemuvexp} improves Corollary \ref{cor:exp} under the
additional assumption that $u$ is non-critical as~$t\to-\infty$,
i.e.~$c_-:=\liminf_{t\to-\infty}X(t)/t>2\sqrt{\mu_-}$. Indeed the upper bound
for the exponential decay rate given by \eqref{eq:decay>} is $\sqrt{\mu_-}$,
whereas the one in Lemma \ref{lemuvexp} is
$\lambda^-_{c_-}=(c_--\sqrt{c_-^2-4\mu_-})/2$, which is smaller than~$\sqrt{\mu_-}$
because $c_->2\sqrt{\mu_-}$.}
\end{remark}

\noindent{\bf{Proof of Lemma~\ref{lemuvexp}.}} From Lemma~\ref{lemuv}, it is sufficient to show the conclusion for the function~$v$. Let~$t\in\R$ be arbitrary. It follows from~(\ref{ineqvmu}) that
$$e^{\lambda^-_cx}v(t,x)\ge M^{-1}\int_{[c,+\infty)}\min\Big(e^{\lambda_c^-x}\varphi_{c'}^-(x-c't-c'\ln M),1\Big)\,d\mu(c')\ \hbox{ for all }x\in\R.$$
For every $c'>c\,(\ge c_->2\sqrt{\mu_-})$, one has $\varphi_{c'}^-(x-c't-c'\ln M)\sim e^{-\lambda_{c'}^-(x-c't-c'\ln M)}$ as $x\to+\infty$ with~$0<\lambda_{c'}^-<\lambda_c^-\,(\le\lambda_{c_-}<\sqrt{\mu_-})$, whence $\min\big(e^{\lambda_c^-x}\varphi_{c'}^-(x-c't-c'\ln M),1\big)\to1$ as $x\to+\infty$. On the other hand, $\min\big(e^{\lambda_c^-x}\varphi_{c}^-(x-ct-c\ln M),1\Big)\to\min\big(e^{\lambda^-_c(ct+c\ln M)},1\big)$ as $x\to+\infty$. Finally, Lebesgue's dominated convergence theorem implies that
$$\liminf_{x\to+\infty}\big(e^{\lambda^-_cx}v(t,x)\big)\ge M^{-1}\min\big(e^{\lambda^-_c(ct+c\ln M)},1\big)\,\mu\big([c,+\infty)\big)>0,$$
that is the desired conclusion.
\hfill$\Box$\break

Next, we show that the measure $\mu$ is compactly supported in $[c_-,+\infty)$.

\begin{lemma}\label{lemsupport} There is $\tilde{c}_+\in[c_-,+\infty)$ such that
the support of $\mu$ is included in $[c_-,\tilde{c}_+]$. Furthermore, without
loss of generality, $\tilde{c}_+$ can be chosen as the rightmost point of the
support of~$\mu$, in the sense that~$\mu\big((\tilde{c}_+,+\infty)\big)=0$ and
$\mu\big([c',+\infty)\big)>0$ for every
$c'<\tilde{c}_+$.
\end{lemma}

\begin{remark}{\rm Notice that Lemma~$\ref{lemsupport}$ does not follow
immediately from Theorem~$\ref{thsupport}$, since one does not know a priori
that $v=u_{\mu}$ is a transition front connecting $0$ and $1$ for~\eqref{eqv}.
We point out that Lemma~$\ref{lemvfront}$ above only shows that the limits $(0$
and $1)$ of $v(t,X(t)+x)$ as $x\to\pm\infty$ are uniform in any time-interval of
the type~$(-\infty,\tau]$ with $\tau\in\R$. But the uniformity of the limits may
depend on $\tau$ in general. Actually, with similar ideas as in the proof of
Lemma~\ref{lemsupport}, it is not complicated to show the existence of
solutions~$0<\tilde{v}<1$ of~\eqref{eqv} satisfying Lemma~$\ref{lemvfront}$ for some
family $(\tilde{X}(t))_{t\in\R}$ and which are not transition fronts. For
instance, consider the function $\tilde{v}=u_{\tilde{\mu}}$ associated to the
measure $\tilde{\mu}=\sum_{n=0}^{+\infty}2^{-n}\delta_{c_0+n}$, where $c_0$ is
any real number in~$[2\sqrt{\mu_-},+\infty)$. The solution~$\tilde{v}$ satisfies
Lemma~$\ref{lemvfront}$ with, say,~$\tilde{X}(t)=c_0t$ for all $t\in\R$,
but~$\inf_{(-\infty,ct]}\tilde{v}(t,\cdot)\to1$ as~$t\to+\infty$ for every
$c\in\R$. Roughly speaking, the function $\tilde{v}$ has an infinite asymptotic
future speed. Because of~\eqref{Xtau}, it cannot be a transition front
for~\eqref{eqv} in the sense of~\eqref{gtf}, for any family~$(X(t))_{t\in\R}$.}
\end{remark}

\noindent{\bf{Proof of Lemma~\ref{lemsupport}.}} Assume by way of contradiction that
$\mu$ is not compactly supported. Pick
any~$c\in[c_-,+\infty)$. Then $\mu\big([c,+\infty)\big)>0$ and therefore it
follows from Lemma~\ref{lemuvexp}, together with \eqref{gtf} and the positivity
and continuity of $u$, that for any $T\in\R$, there is $\beta\in(0,1)$ such
that
\Fi{u>beta}
\forall\,x\in\R,\quad
u(T,x)\ge\min\big(\beta,\beta\,e^{-\lambda^-_cx}\big).
\Ff
Notice that $\lambda^-_c\searrow0$ as $c\to+\infty$. Thus, for
$c$ large enough, we have that
$\lambda^-_c<\sqrt{\mu_+/2}=\sqrt{f_+'(0)/2}$,
whence equation \eqref{eqg} with $g=f_+/2$ admits a standard traveling front
$\t\vp(x-\gamma_c t)$ connecting~$0$ and~$1$
satisfying~$\t\varphi(\xi)\sim e^{-\lambda_c^-\xi}$ as $\xi\to+\infty$, where
$\gamma_c:=\lambda^-_c+\mu_+/(2\lambda^-_c)\to+\infty$ as $c\to+\infty$.
By~(\ref{fpm}), there exists $T\in\R$ such that
$f(t,s)\ge f_+(s)/2$ for all $(t,s)\in[T,+\infty)\times[0,1]$. This implies
that $\t\vp(x-\gamma_c t)$ is a subsolution to \eqref{eq:P=f} for $t>T$, and
then the same is true for~$\t\beta\t\vp(x-\gamma_c t)$, for any
$\t\beta\in(0,1)$.
Furthermore, by~\eqref{eq:u>beta},~$\t\beta\in(0,1)$ can be chosen in such a
way that $\t\beta\t\vp(x-\gamma_c T)\leq u(T,x)$ for all $x\in\R$.
It then follows from the maximum principle that
$u(t,\gamma_c t)\ge\t\beta\t\vp(0)>0$ for
all~$t>T$, from which, owing to~(\ref{gtf}) or~(\ref{infsup}), we
derive~$c_+=\liminf_{t\to+\infty}X(t)/t\ge\gamma_c$.
Since this inequality holds for all $c$ large enough, letting~$c\to+\infty$ we
eventually get $\liminf_{t\to+\infty}X(t)/t=+\infty$, which is ruled out
by~(\ref{boundedspeed}).\par
Finally, the support of $\mu$ is bounded and there is
$\tilde{c}_+\in[c_-,+\infty)$ as in Lemma~\ref{lemsupport}, that
is~$\tilde{c}_+$ is the rightmost point of the support of $\mu$. The proof of
Lemma~\ref{lemsupport} is thereby complete.~\hfill$\Box$\break

\noindent{\it{Substep 5.2: $c_-$ is the asymptotic past speed of $u$ and $u$ converges to $\varphi_{c_-}^-$ along its level sets as~$t\to-\infty$}}\hfill\break

\noindent{}Let us first show here that $X(t)/t\to c_-$ as $t\to-\infty$. We
already know by definition that $\liminf_{t\to-\infty}X(t)/t=c_-$. From part~(i)
of Theorem~\ref{thsupport} and from Lemma~\ref{lemsupport}, the
solution~$v=u_{\mu}$ of~(\ref{eqv}) is actually a transition front connecting
$0$ and $1$ for this equation. In other words, there is a
family~$(\tilde{X}(t))_{t\in\R}$ of real numbers such that~(\ref{gtf}) holds
for~$v$ and~$(\tilde{X}(t))_{t\in\R}$. Furthermore, from part~(ii) of
Theorem~\ref{thsupport}, the transition front~$v$ has some asymptotic past and
future speeds, which are the leftmost and rightmost points of the support of
$\mu$. Having in hand that the measure~$\mu$ is supported
in~$[c_-,\tilde{c}_+]\subset(2\sqrt{\mu_-},+\infty)$ from
Lemma~\ref{lemsupport}, the leftmost point~$\tilde{c}_-$ of the support of $\mu$
satisfies~$2\sqrt{\mu_-}=c^*_-<c_-\le\tilde{c}_-\le\tilde{c}_+$. Part~(ii) of
Theo\-rem~\ref{thsupport} implies in particular that
$\tilde{X}(t)/t\to\tilde{c}_-$ as~$t\to-\infty$ and that there is a bounded
function~$\tilde{\xi}:\R\to\R$ such that
\be\label{tildexi}
v(t,\tilde{X}(t)+\tilde{\xi}(t)+x)\to\varphi^-_{\tilde{c}_-}(x)\
\hbox{ as }t\to-\infty\hbox{ uniformly in }x\in\R.
\ee\par
But $0<\inf_{t\in\R}v(t,\tilde{X}(t))\le\sup_{t\in\R}v(t,\tilde{X}(t))<1$
from~(\ref{infsup}) applied to the transition front~$v$,
whence~$0<\liminf_{t\to-\infty}u(t,\tilde{X}(t))\le\limsup_{t\to-\infty}u(t,\tilde{X}
(t))<1$ by Lemma~\ref{lemuv} and the fact that~$\Theta(t)\to0$ as $t\to-\infty$.
Together with~(\ref{gtf}), one infers that
\be\label{XtildeX}
\limsup_{t\to-\infty}\big|X(t)-\tilde{X}(t)\big|<+\infty.
\ee
Therefore, $X(t)/t\to\tilde{c}_-$ as $t\to-\infty$. Remembering
that~$c_-=\liminf_{t\to-\infty}X(t)/t$, one gets that
$\tilde{c}_-=c_-$ and $X(t)/t\to c_-$ as $t\to-\infty$.
In other words, $u$ has an asymptotic past speed, which is equal to $c_-$.\par
Finally, setting
\be\label{defxi1}
\xi(t)=\tilde{X}(t)-X(t)+\tilde{\xi}(t)\ \hbox{ for }t\le0,
\ee
it follows from~(\ref{XtildeX}), the boundedness of~$\tilde{\xi}$ and the local
boundedness of $X$ and $\tilde{X}$ that $\xi$ is bounded in~$(-\infty,0]$.
Furthermore,
$$\baa{rcl}
u(t,X(t)+\xi(t)+x)-\varphi^-_{c_-}(x) & = & u(t,\tilde{X}(t)+
\tilde{\xi}(t)+x)-v(t,\tilde{X}(t)+\tilde{\xi}(t)+x)\vspace{3pt}\\
& & +\,v(t,\tilde{X}(t)+\tilde{\xi}(t)+x)-\varphi^-_{c_-}(x)\vspace{3pt}\\
& \to & 0\ \hbox{ as }t\to-\infty\hbox{ uniformly in }x\in\R,\eaa$$
from~(\ref{tildexi}) and Lemma~\ref{lemuv}, together with
$\lim_{t\to-\infty}\Theta(t)=0$. Standard parabolic estimates also imply that
\be\label{convphi-}
u(t,X(t)+\xi(t)+\cdot)\to\varphi^-_{c_-}
\hbox{ in }C^2(\R)\hbox{ as }t\to-\infty.
\ee

\vskip 0.3cm
\noindent{\it{Substep 5.3: $c_+$ is the asymptotic future speed of $u$}}\hfill
\break

\noindent{}Here, we prove the existence of the asymptotic future speed of $u$.
This speed, that will be equal to $c_+$, will be determined obviously by the
limiting nonlinearity $f_+$ (namely, by~$\mu_+=f'_+(0)$), but also by the
rightmost point $\tilde{c}_+$ of the support of $\mu$. To do so, we will
identify the exponential decay rate of $u(t,x)$ and $v(t,x)$ as~$x\to+\infty$ in
terms of $\tilde{c}_+$, for some suitably chosen times $t$.

\begin{lemma}\label{lemfuture}
There holds
\be\label{Xtc+}
\frac{X(t)}{t}\to c_+=\tilde{\kappa}+\frac{\mu_+}{\tilde{\kappa}}\ge\kappa+
\frac{\mu_+}{\kappa}\ \hbox{ as }t\to+\infty
\ee
with
\be\label{tildekappa}
0<\tilde{\kappa}=\min\Big(\frac{\tilde{c}_+-\sqrt{(\tilde{c}_+)^2-4\mu_-}}{2},
\sqrt{\mu_+}\Big)\le\kappa=\min\Big(\frac{c_--\sqrt{c_-^2-4\mu_-}}{2},\sqrt{
\mu_+}\Big)\le\sqrt{\mu_+}.
\ee
\end{lemma}

\noindent{\bf{Proof.}} The strategy consists in establishing lower and upper
bounds of $X(t)/t$ as $t\to+\infty$ by comparing~$u$ at large time with some
solutions of reaction-diffusion equations with nonlinearities of the type
$(1\pm\epsilon)f_+$. In the proof of the lemma, $0<\epsilon<1$ is arbitrary.
From~(\ref{fpm}), there is~$T_{\epsilon}\in\R$ such that, for all
$t\ge T_{\epsilon}$ and~$s\in[0,1]$,
$(1-\epsilon)\,f_+(s)\le f(t,s)\le(1+\epsilon)\,f_+(s)$.\par
Let us first prove the upper bound for $\limsup_{t\to+\infty}X(t)/t$.
Lemma~\ref{lemuv} yields $u(T_{\epsilon},x)\le
v(T_{\epsilon},x)\,e^{\mu_-\Theta(T_{\epsilon})}$ for all $x\in\R$. On the other
hand, since~$v$ is a solution of~(\ref{eqv}) of the type $v=u_{\mu}$ and since
the support of $\mu$ is included
in~$[c_-,\tilde{c}_+]\subset[2\sqrt{\mu_-},+\infty)$, it follows from
Proposition~\ref{pro3}
that~$v(T_{\epsilon},x)=u_{\mu}(T_{\epsilon},x)\le\varphi^-_{\tilde{c}_+}
\big((\varphi^-_{\tilde{c}_+})^{-1}(v(T_{\epsilon},0))+x\big)$ for all
$x\ge0$.
Remember also that, since~$\tilde{c}_+\ge c_->2\sqrt{\mu_-}$, one has
$\varphi_{\tilde{c}_+}^-(\xi)\le e^{-\lambda^-_{\tilde{c}_+}\xi}$ for all
$\xi\in\R$, where
$\lambda^-_{\tilde{c}_+}=(\tilde{c}_+-\sqrt{\tilde{c}_+^2-4\mu_-})/2$.
Therefore, there is a real number $\gamma_{\epsilon}>0$ such that
$u(T_{\epsilon},x)\le\min\big(\gamma_{\epsilon}\,e^{-\lambda^-_{\tilde{c}_+}x},
1\big)$ for all $x\in\R$. Let $\overline{u}^{\epsilon}$ be the solution of the
Cauchy problem
$$\left\{\baa{rcl}
\overline{u}^{\epsilon}_t  & = & \overline{u}^{\epsilon}_{xx}+(1+\epsilon)\,
f_+(\overline{u}^{\epsilon}),\ \ t>0,\ x\in\R,\vspace{3pt}\\
\overline{u}^{\epsilon}(0,x) & = & \min\big(\gamma_{\epsilon}\,
e^{-\lambda^-_{\tilde{c}_+}x},1\big),\ \ x\in\R.\eaa\right.$$
The maximum principle implies that
$0<u(t,x)\le\overline{u}^{\epsilon}(t-T_{\epsilon},x)$ for all
$(t,x)\in[T_{\epsilon},+\infty)\times\R$. But, by~\cite{u}, the function
$\overline{u}^{\epsilon}$ spreads to the right with the speed
$\kappa_{\epsilon}+(1+\epsilon)\mu_+/\kappa_{\epsilon}$,
where~$\kappa_{\epsilon}=\min\big(\lambda^-_{\tilde{c}_+},\sqrt{
(1+\epsilon)\mu_+}\big)$. In particular,
$\sup_{[(\kappa_{\epsilon}+(1+\epsilon)\mu_+/\kappa_{\epsilon}+\eta)\,t,+\infty)
}\overline{u}^{\epsilon}(t,\cdot)\to0$ as~$t\to+\infty$ for every $\eta>0$,
whence
$$\sup_{\big[(\kappa_{\epsilon}+(1+\epsilon)\mu_+/\kappa_{\epsilon}+\eta)\,
(t-T_{\epsilon}),+\infty\big)}u(t,\cdot)\to0\ \hbox{ as }t\to+\infty.$$
It follows then from~(\ref{infsup}) that
$X(t)-\big(\kappa_{\epsilon}+(1+\epsilon)
\mu_+/\kappa_{\epsilon}+\eta\big)\,(t-T_{\epsilon})\to-\infty$ as $t\to+\infty$,
whence
$\limsup_{t\to+\infty}X(t)/t\le\kappa_{\epsilon}+(1+\epsilon)\mu_+/\kappa_{
\epsilon}+\eta$. Since $\eta>0$ and $\epsilon>0$ can be arbitrarily small, one
gets that
\be\label{kappatilde1}
\limsup_{t\to+\infty}\frac{X(t)}{t}\le
\tilde{\kappa}+\frac{\mu_+}{\tilde{\kappa}},
\ee
where $\tilde{\kappa}=\min\big(\lambda^-_{\tilde{c}_+},\sqrt{\mu_+}\big)>0$
is as in~\eqref{tildekappa}.\par
Let us now prove the lower bound of $\liminf_{t\to+\infty}X(t)/t$. For every
$\delta>0$, set $c_{\delta}=\max(c_-,\tilde{c}_+-\delta)\,(\ge c_-)$. Since
$c_-$ and $\tilde{c}_+$ are the leftmost and rightmost points of the support of
$\mu$, one has $\mu\big([c_{\delta},+\infty)\big)>0$ for every~$\delta>0$ and
Lemma~\ref{lemuvexp}
yields~$\liminf_{x\to+\infty}e^{\lambda^-_{c_{\delta}}x}u(T_{\epsilon},x)>0$.
Since $u(T_{\epsilon},-\infty)=1$ and $u(T_{\epsilon},\cdot)$ is positive and
continuous in $\R$, it follows that there is a positive real number
$\rho_{\epsilon}$ such that
$u(T_{\epsilon},x)\ge\min\big(\rho_{\epsilon},\rho_{\epsilon}\,
e^{-\lambda^-_{c_{\delta}}x}\big)$ for all $x\in\R$. Let
$\underline{u}^{\epsilon}$ be the solution of the Cauchy problem
$$\left\{\baa{rcl}
\underline{u}^{\epsilon}_t  & = & \underline{u}^{\epsilon}_{xx}+(1-\epsilon)\,
f_+(\underline{u}^{\epsilon}),\ \ t>0,\ x\in\R,\vspace{3pt}\\
\underline{u}^{\epsilon}(0,x) & = & \min\big(\rho_{\epsilon},\rho_{\epsilon}\,
e^{-\lambda^-_{c_{\delta}}x}\big),\ \ x\in\R.\eaa\right.$$
The maximum principle implies that $1>u(t,x)\ge\underline{u}^{\epsilon}
(t-T_{\epsilon},x)$ for all $(t,x)\in[T_{\epsilon},+\infty)\times\R$. But,
by~\cite{u}, the function $\underline{u}^{\epsilon}$ spreads to the right with
the speed $\kappa_{\epsilon,\delta}+(1-\epsilon)\mu_+/\kappa_{\epsilon,\delta}$,
where~$\kappa_{\epsilon,\delta}=\min\big(\lambda^-_{c_{\delta}},\sqrt{
(1-\epsilon)\mu_+}\big)$. In particular,
$\inf_{(-\infty,(\kappa_{\epsilon,\delta}+(1-\epsilon)\mu_+/\kappa_{\epsilon,
\delta}-\eta)\,t]}\underline{u}^{\epsilon}(t,\cdot)\to1$ as~$t\to+\infty$ for
every $\eta>0$, whence
$\inf_{(-\infty,(\kappa_{\epsilon,\delta}+(1-\epsilon)\mu_+/
\kappa_{\epsilon,\delta}-\eta)\,(t-T_{\epsilon})]}u(t,\cdot)\to1$ as
$t\to+\infty$.
It follows then from~(\ref{infsup}) that
$$X(t)-\big(\kappa_{\epsilon,\delta}+
(1-\epsilon)\mu_+/\kappa_{\epsilon,\delta}-\eta\big)\,(t-T_{\epsilon}
)\to+\infty\ \hbox{ as }t\to+\infty,$$
whence
$\liminf_{t\to+\infty}X(t)/t\ge\kappa_{\epsilon,\delta}
+(1-\epsilon)\mu_+/\kappa_{\epsilon,\delta}-\eta$. Since $\eta>0$, $\delta>0$
and $\epsilon>0$ can be arbitrarily small and since
$\lambda^-_{c_{\delta}}\to\lambda^-_{\tilde{c}_+}$ as~$\delta\to0$, one gets
that $\liminf_{t\to+\infty}X(t)/t\ge\tilde{\kappa}+\mu_+/\tilde{\kappa}$ with
$\tilde{\kappa}>0$ as in~(\ref{tildekappa}). Together with~(\ref{kappatilde1}),
one concludes that
$X(t)/t\to\tilde{\kappa}+\mu_+/\tilde{\kappa}$ as $t\to
+\infty$.
Since $c_+$ was defined as the liminf of $X(t)/t$ as $t\to+\infty$, this means
that $c_+=\tilde{\kappa}+\mu_+/\tilde{\kappa}$. Furthermore, the inequalities
$2\sqrt{\mu_-}<c_-\le\tilde{c}_+$
yield~$0<\lambda^-_{\tilde{c}_+}\le\lambda^-_{c_-}=(c_--\sqrt{c_-^2-4\mu_-})/2$ and
$0<\tilde{\kappa}=\min(\lambda^-_{\tilde{c}_+},\sqrt{\mu_+})\le\min(\lambda^-_{
c_-},\sqrt{\mu_+})=\kappa\le\sqrt{\mu_+}$, where $\kappa$ is as
in~(\ref{tildekappa}). Finally,~\eqref{Xtc+} holds and the proof of
Lemma~\ref{lemfuture} is complete.\hfill$\Box$\break

\noindent{\it{Substep 5.4: $u$ converges to a well identified
profile~$\varphi^+_{c_+}$ along its level sets as~$t\to+\infty$}}\hfill\break

\noindent{}The proof is based on some comparisons between $u$ and some
solutions
of the homogeneous equation~(\ref{eqf+}) below with the reaction term exactly
equal to~$f_+$. A Liouville type result about the classification of solutions of
homogeneous equations which are asymptotically trapped between two shifts of a
standard traveling front will also be used, as in Step~3 of the proof of
Proposition~\ref{pro:exsupercrit}.\par
First of all, let $0<\varphi_{c_+}^+(x-c_+t)<1$ be a standard traveling front
connecting $0$ and $1$ for the homogeneous equation
\be\label{eqf+}
w_t=w_{xx}+f_+(w),
\ee
with the speed $c_+=\tilde{\kappa}+\mu_+/\tilde{\kappa}\ge2\sqrt{\mu_+}=
2\sqrt{f'_+(0)}$ as in~(\ref{Xtc+})-(\ref{tildekappa}). Up to normalization, one
can assume that, as
$\xi\to+\infty$, $\varphi^+_{c_+}(\xi)\sim e^{-\tilde{\kappa}\xi}$
if~$c_+>2\sqrt{\mu_+}$ while
$\varphi^+_{c_+}(\xi)\sim\xi\,e^{-\tilde{\kappa}\xi}=\xi\,e^{-\sqrt{\mu_+}\xi}$
if $c_+=2\sqrt{\mu_+}$. For every $t\in\R$, since $u(t,\cdot)$ is continuous and
converges to $0$ and $1$ at $\pm\infty$, the real number
$$X^+(t)=\max\big\{x\in\R;\ u(t,x)=\varphi_{c_+}^+(0)\big\}$$
is well defined, and $u(t,X^+(t))=\varphi_{c_+}^+(0)$. Denote
\be\label{defxi2}
\xi(t)=X^+(t)-X(t)\ \hbox{ for }t>0.
\ee
The function $\xi$ is bounded in $(0,+\infty)$, by~(\ref{gtf}). It will be the
one used in the desired conclusion~\eq{phi+-} for~$t>0$, together
with~(\ref{defxi1}) and~\eqref{convphi-} for $t\le0$.\par
Let us now introduce some auxiliary functions $u^+$ and $v^+$. From
assumption~(\ref{hypf+}), there holds, for all~$t>0$ and $s\in[0,1]$,
$(1-\tilde{\zeta}(t))\,f_+(s)\le f(t,s)\le(1+\tilde{\zeta}(t))\,f_+(s)$.
Set $\tilde{\Theta}(t)=\int_{0}^t\tilde{\zeta}(\tau)\,d\tau$ for $t\ge0$ and
notice that
\be\label{tildeTheta}
0\le\tilde{\Theta}(t)\le\tilde{\Theta}_{\infty}:=\int_{0}^{+\infty}
\tilde{\zeta}(\tau)\,d\tau<+\infty\ \hbox{ for all }t\ge0.
\ee
Firstly, let $u^+$ be the solution of the Cauchy problem
\be\label{defu+}\left\{\baa{rcl}
u^+_t & = & u^+_{xx}+f_+(u^+),\ \ t>0,\ x\in\R,\vspace{3pt}\\
u^+(0,x) & = & u(0,x),\ \ x\in\R.\eaa\right.
\ee
As in the proof of Lemma~\ref{lemuv}, since $s\mapsto f_+(s)/s$ is
nonincreasing
on $(0,1]$ and~$f'_+(0)=\mu_+>0$, one can prove that
\be\label{u+u}
0<u^+(t,x)\,e^{-\mu_+\tilde{\Theta}(t)}\le u(t,x)\le u^+(t,x)\,
e^{\mu_+\tilde{\Theta}(t)}\ \hbox{ for all }(t,x)\in[0,+\infty)\times\R.
\ee
Secondly, let $v^+$ be the solution of the Cauchy problem
\be\label{defv+}\left\{\baa{rcl}
v^+_t & = & v^+_{xx}+f_+(v^+),\ \ t>0,\ x\in\R,\vspace{3pt}\\
v^+(0,x) & = & v(0,x),\ \ x\in\R,\eaa\right.
\ee
where we recall that $0<v<1$ solves~(\ref{eqv}). Setting
$\alpha=e^{-\mu_-\Theta(0)}\in(0,1]$, Lemma~\ref{lemuv} then implies
that~$0<\alpha\,v(0,\cdot)\le u(0,\cdot)\le\min\big(\alpha^{-1}v(0,\cdot),1\big)$ in $\R$,
that is, $0<\alpha\,v^+(0,\cdot)\le u^+(0,\cdot)\le\min
\big(\alpha^{-1}v^+(0,\cdot),1\big)$. Since $s\mapsto f_+(s)/s$ is
nonincreasing on $(0,1]$, it follows that $\alpha\,v^+$ is a subsolution and
$\min\big(\alpha^{-1}v^+,1\big)$ is a supersolution of the
equation~(\ref{defu+}) satisfied by $u^+$. Therefore, the maximum principle
yields
\be\label{u+v}
0<\alpha\,v^+(t,x)\le u^+(t,x)\le\min\big(\alpha^{-1}v^+(t,x),1\big)\
\hbox{ for all }(t,x)\in[0,+\infty)\times\R.
\ee
Gathering~(\ref{tildeTheta}),~(\ref{u+u}) and~(\ref{u+v}) together with the
fact that $u$ ranges in~$(0,1)$, one infers that
\be\label{ubeta}
0<\beta\,v^+(t,x)\le u(t,x)\le\min\big(\beta^{-1}v^+(t,x),1\big)\
\hbox{ for all }(t,x)\in[0,+\infty)\times\R,
\ee
where $\beta=\alpha\,e^{-\mu_+\tilde{\Theta}_{\infty}}=e^{-\mu_-\Theta(0)-\mu_+
\tilde{\Theta}_{\infty}}\in(0,1]$.\par
The following lemma establishes the exact decay rate of $v(0,x)$ as
$x\to+\infty$
(we point out that the same lemma would actually hold at any fixed time
$t\in\R$). We recall that the solution~$0<v<1$ of~(\ref{eqv}) is of the type
$v=u_{\mu}$ and that $\tilde{c}_+\in[c_-,+\infty)$ denotes the rightmost point
of the measure~$\mu$.

\begin{lemma}\label{lemmev}
We recall that $\lambda^-_{\tilde{c}_+}=(\tilde{c}_+-\sqrt{\tilde{c}_+^2-4\mu_-})/2$. For every $x_0\in\R$, there holds
\be\label{vtildeT}
\frac{v(0,x+x_0)}{v(0,x)}\to e^{-\lambda^-_{\tilde{c}_+}x_0}\ \hbox{ as }
x\to+\infty.
\ee
\end{lemma}

We point out that Proposition \ref{pro3} and Lemma \ref{lemuvexp}
immediately
imply that $v$ cannot have exponential decay rate respectively smaller and
larger than $\lambda^-_{\tilde{c}_+}$ as $x\to+\infty$, but this is not enough
for applying the results of Uchiyama~\cite{u} to $v^+$.
The proof of Lemma \ref{lemmev} uses the inequalities~(\ref{umu}) applied to
$v=u_{\mu}$. Since it is a bit technical, we postpone it in the next
subsection.
We first complete the proof of Theorem~\ref{th2}, that is the proof
of~\eq{phi+-} as $t\to+\infty$.\par
As already mentioned in the proof of Theorem~\ref{thsupport}, the function
$v=u_{\mu}$ is continuously decreasing with respect to $x$. In particular,
$v^+(0,\cdot)=v(0,\cdot)$ is continuously decreasing in~$\R$, with
$v^+(0,-\infty)=1$ and~$v^+(0,+\infty)=0$. From the maximum principle, for
every~$t\ge0$, the function $v^+(t,\cdot)$ is continuous and decreasing in $\R$
and converges to $0$ and $1$ at~$\pm\infty$. Therefore, for every $t\ge0$, there
is a unique real number~$Y^+(t)$ such that
\be\label{defY+}
v^+\big(t,Y^+(t)\big)=\frac12.
\ee
Owing to the definition of $v^+$ in~(\ref{defv+}), it follows then from
Lemma~\ref{lemmev} above and Theorems~8.1,~8.2 and~8.5 of~\cite{u} that
\be\label{convphi+}
\sup_{x\in[-Y^+(t),+\infty)}\big|v^+\big(t,Y^+(t)+x\big)-\varphi_{c_+}^+(x_0+x)
\big|\to0\ \hbox{ as }t\to+\infty,
\ee
where $x_0=(\varphi^+_{c_+})^{-1}(1/2)$ and $c_+=\tilde{\kappa}+\mu_+/
\tilde{\kappa}$ with
$0<\tilde{\kappa}=\min\big(\lambda^-_{\tilde{c}_+},\sqrt{\mu_+}\big)\le\sqrt{
\mu_+}$. Furthermore,~$Y^+(t)/t\to c_+$ as~$t\to+\infty$. Thus,
since~$0<v^+<1$ is decreasing with respect to~$x$
and~$\varphi^+_{c_+}(-\infty)=1$, the convergence~(\ref{convphi+})
holds uniformly in $x\in\R$, that is
\be\label{convphi+2}
\sup_{x\in\R}\big|v^+\big(t,Y^+(t)+x\big)-\varphi_{c_+}^+(x_0+x)\big|\to0\
\hbox{ as }t\to+\infty.
\ee\par
Let $(t_p)_{p\in\N}$ be any sequence of real numbers in $[0,+\infty)$ such that
$t_p\to+\infty$ as~$p\to+\infty$. For every~$p\in\N$ and
$(t,x)\in[t_p,+\infty)\times\R$, set
$v^+_p(t,x)=v^+\big(t_p+t,Y^+(t_p)+x\big)$. Up to extraction of a subsequence,
the functions $v_p^+$ converge in $C^{1,2}_{loc}(\R^2)$ as~$p\to+\infty$ to a
solution $0\le v^+_{\infty}\le1$ of~(\ref{eqf+}) in $\R^2$.

\begin{lemma}\label{vinfty+}
One has $v^+_{\infty}(t,x)=\varphi^+_{c_+}(x-c_+t+x_0)$ for all $(t,x)\in\R^2$.
\end{lemma}

\noindent{\bf{Proof.}} Since $v^+_p(0,0)=v^+\big(t_p,Y^+(t_p)\big)=1/2$ for all $p\in\N$, there holds $v^+_{\infty}(0,0)=1/2$ and~$0<v^+_{\infty}<1$ in~$\R^2$ from the strong maximum principle. For every $t\in\R$ and $p$ large enough so that $t_p+t\ge0$, one has
$$v^+\big(t_p\!+\!t,Y^+(t_p\!+\!t)\!+\!Y^+(t_p)\!-\!Y^+(t_p\!+\!t)\big)\!=\!v^+(t_p\!+\!t,Y^+(t_p)\big)\!=\!v^+_p(t,0)\to v^+_{\infty}(t,0)\in(0,1)$$
as $p\to+\infty$, while $v^+\big(t_p+t,Y^+(t_p+t)+Y^+(t_p)-Y^+(t_p+t)\big)-\varphi_{c_+}^+\big(x_0+Y^+(t_p)-Y^+(t_p+t)\big)\to0$ as~$p\to+\infty$, from~(\ref{convphi+2}). Therefore,
\be\label{Y+tp}
x_0+Y^+(t_p)-Y^+(t_p+t)\to(\varphi^+_{c_+})^{-1}(v_{\infty}^+(t,0))\ \hbox{ as }p\to+\infty.
\ee
Using again~(\ref{convphi+2}), with~(\ref{Y+tp}), one infers that, for every $(t,x)\in\R^2$ (and $t_p+t\ge0$)
$$v^+_p(t,x)\!=\!v^+\big(t_p\!+\!t,Y^+(t_p\!+\!t)\!+\!Y^+(t_p)\!-\!Y^+(t_p\!+\!t)\!+\!x\big)\to\varphi^+_{c_+}\big((\varphi^+_{c_+})^{-1}(v_{\infty}^+(t,0))\!+\!x\big)$$
as $p\to+\infty$, whence
\be\label{vinfty+2}
v^+_{\infty}(t,x)=\varphi^+_{c_+}\big((\varphi^+_{c_+})^{-1}(v_{\infty}^+(t,0))+x\big).
\ee
As a consequence, for every $s<t\in\R$ and $x\in\R$, it follows from the uniqueness of the bounded solutions of the Cauchy problem associated to~(\ref{eqf+}) that
$v^+_{\infty}(t,x)=\varphi^+_{c_+}\big((\varphi^+_{c_+})^{-1}(v_{\infty}^+(s,0))+x-c_+(t-s)\big)$,
whence~$(\varphi^+_{c_+})^{-1}(v_{\infty}^+(t,0))=(\varphi^+_{c_+})^{-1}(v_{\infty}^+(s,0))-c_+(t-s)$. Finally,
$$(\varphi^+_{c_+})^{-1}(v_{\infty}^+(t,0))=(\varphi^+_{c_+})^{-1}(v_{\infty}^+(0,0))-c_+t=(\varphi^+_{c_+})^{-1}\Big(\frac12\Big)-c_+t=x_0-c_+t\ \hbox{ for all }t\in\R,$$
and the desired conclusion of Lemma~\ref{vinfty+} follows from~(\ref{vinfty+2}).\hfill$\Box$\break

Now, call, for $p\in\N$ and $(t,x)\in[-t_p,+\infty)\times\R$,
$u_p(t,x)=u\big(t_p+t,Y^+(t_p)+x\big)$.
Each function~$u_p$ satisfies $(u_p)_t=(u_p)_{xx}+f(t_p+t,u_p)$ in $[-t_p,+\infty)\times\R$. It follows then from standard parabolic estimates and~(\ref{fpm}) that, up to extraction of a subsequence, the functions $u_p$ converge in $C^{1,2}_{loc}(\R^2)$ to a classical solution $0\le u_{\infty}\le1$ of~(\ref{eqf+}) in $\R^2$.

\begin{lemma}\label{uinfty}
There is a real number $x_{\infty}$ such that
$u_{\infty}(t,x)=\varphi^+_{c_+}(x-c_+t+x_{\infty})$ for all $(t,x)\in\R^2$.
\end{lemma}

\noindent{\bf{Proof.}} It follows from~(\ref{ubeta}) that $0<\beta\,v^+_p(t,x)\le u_p(t,x)\le\min\big(\beta^{-1}v^+_p(t,x),1\big)$ for all~$p\in\N$ and~$(t,x)\in[-t_p,+\infty)\times\R$. Therefore,
\be\label{uv+}
0<\beta\,v^+_{\infty}(t,x)\le u_{\infty}(t,x)\le\min\big(\beta^{-1}v^+_{\infty}(t,x),1\big)\ \hbox{ for all }(t,x)\in\R^2,
\ee
by passing to the limit as $p\to+\infty$. Since $\inf_{\R^2}v^+_{\infty}=0$ by Lemma~\ref{vinfty+}, the function $u_{\infty}$ cannot be identically equal to $1$. Finally,~$0<u_{\infty}<1$ in $\R^2$ from the strong maximum principle. Furthermore, Lemma~\ref{vinfty+} and~(\ref{uv+}) imply that $\inf_{x-c_+t\le0}u_{\infty}(t,x)>0$ and, as in the proof of Lemma~\ref{lemvfront}, one can then show that~$u_{\infty}(t,x)\to1$ as $x-c_+t\to-\infty$. On the other hand, it also follows from Lemma~\ref{vinfty+} and~(\ref{uv+}) together with $\varphi^+_{c_+}(\xi)\sim e^{-\tilde{\kappa}\xi}$ if $c_+>2\sqrt{\mu_+}$ (resp.~$\varphi^+_{c_+}(\xi)\sim\xi\,e^{-\tilde{\kappa}\xi}$ if~$c_+=2\sqrt{\mu_+}$) as $\xi\to+\infty$, that there is $A\ge0$ such that
$\varphi_{c_+}^+(x-c_+t+A)\le u_{\infty}(t,x)\le\varphi_{c_+}^+(x-c_+t-A)$ for all $(t,x)\in\R^2$ with $x-c_+t\ge0$.
As a consequence, as in Step~3 of the proof of Proposition~\ref{pro:exsupercrit}, by combining Proposition~4.3 in~\cite{NR1} and Theorem~3.5 of~\cite{bh1} (see also Lemma~8.2 of~\cite{hnrr2}, adapted here to the homogeneous case), the conclusion of Lemma~\ref{uinfty} follows, for some real number~$\xi_{\infty}\in[-A,A]$.\hfill$\Box$\break

We are finally able to complete the proof of the limit~\eq{phi+-} as $t\to+\infty$, that is the convergence of $u$ to $\varphi^+_{c_+}$ along its level sets as $t\to+\infty$. Remember that the function $\xi$ was defined by~(\ref{defxi2}) for $t>0$. In particular, $u\big(t_p,X(t_p)+\xi(t_p)\big)=u\big(t_p,X^+(t_p)\big)=\varphi_{c_+}^+(0)$ for every~$p\in\N$, while~$u\big(t_p,Y^+(t_p)\big)=u_p(0,0)\to u_{\infty}(0,0)=\varphi_{c_+}^+(x_{\infty})\in(0,1)$ as $p\to+\infty$, from Lemma~\ref{uinfty}. Therefore,~(\ref{gtf}) implies that the sequence $(X(t_p)+\xi(t_p)-Y^+(t_p))_{p\in\N}$ is bounded, whence
$$u_p\big(0,X(t_p)+\xi(t_p)-Y^+(t_p)\big)-u_{\infty}\big(0,X(t_p)+\xi(t_p)-Y^+(t_p)\big)\to0\ \hbox{ as }p\to+\infty.$$
But, for all $p\in\N$,
$$\left\{\baa{lcl}
u_p\big(0,X(t_p)+\xi(t_p)-Y^+(t_p)\big) & = & u\big(t_p,X(t_p)+\xi(t_p)\big)\ =\ \varphi^+_{c_+}(0),\vspace{3pt}\\
u_{\infty}\big(0,X(t_p)+\xi(t_p)-Y^+(t_p)\big) & = & \varphi_{c_+}^+\big(X(t_p)+\xi(t_p)-Y^+(t_p)+x_{\infty}\big).\eaa\right.$$
One infers that $X(t_p)+\xi(t_p)-Y^+(t_p)+x_{\infty}\to0$ as $p\to+\infty$. As a consequence,
$$u\big(t_p,X(t_p)+\xi(t_p)+x)=u_p\big(0,X(t_p)+\xi(t_p)-Y^+(t_p)+x\big)\mathop{\longrightarrow}_{p\to+\infty}u_{\infty}(0,x-x_{\infty})=\varphi_{c_+}^+(x)$$
from Lemma~\ref{uinfty}, and the convergence holds locally uniformly with respect to $x$. Furthermore, since~$u\big(t_p,X(t_p)+x\big)\to1$ (resp. $0$) as $x\to-\infty$ (resp. $x\to+\infty$) uniformly in $p\in\N$ by~(\ref{gtf}), since~$\xi:(0,+\infty)\to\R$ is bounded and since $\varphi_{c_+}^+(-\infty)=1$ and $\varphi_{c_+}^+(+\infty)=0$, one gets
that~$u\big(t_p,X(t_p)+\xi(t_p)+x\big)\to\varphi_{c_+}^+(x)$ as $p\to+\infty$ uniformly in
$x\in\R$. But the limit does not depend on the sequence $(t_p)_{p\in\N}$ converging to $+\infty$. From the compactness arguments used in the above proof, one concludes that $u\big(t,X(t)+\xi(t)+x\big)\to\varphi_{c_+}^+(x)$ as $t\to+\infty$ uniformly in $x\in\R$, and standard parabolic estimates also imply that the convergence holds in~$C^2(\R)$.\par
Finally, by defining $\xi:\R\to\R$ by~(\ref{defxi1}) in $(-\infty,0]$ and
by~(\ref{defxi2}) in $(0,+\infty)$, the function~$\xi$ is bounded and the
conclusion~\eq{phi+-}
holds with $\phi_{c_{\pm}}(x)=\varphi^{\pm}_{c_{\pm}}(x)$. The proof of
Theorem~\ref{th2} is thereby complete.\hfill$\Box$

\subsubsection{Proof of Lemma~\ref{lemmev}}

We recall from Substeps 5.1 and 5.2 above that $0<v=u_{\mu}<1$ obeys~(\ref{eqv}), where the measure~$\mu\in\mathcal{M}$ is supported in $[c_-,\tilde{c}_+]\subset(2\sqrt{\mu_-},+\infty)$, and $c_-$ and $\tilde{c}_+$ are the leftmost and rightmost points of the support of $\mu$. Here, $M=\mu\big([c_-,\tilde{c}_+]\big)>0$ and the inequalities~(\ref{umu}) amount to
\be\label{vtildeTbis}
M^{-1}\!\!\int_{[c_-,\tilde{c}_+]}\!\!\varphi^-_c\big(x\!-\!c\ln M\big)\,d\mu(c)\le v(0,x)\le M^{-1}\!\!\int_{[c_-,\tilde{c}_+]}\!\!e^{-\lambda^-_c(x-c\ln M)}\,d\mu(c)
\ee
for all $x\in\R$. If $c_-=\tilde{c}_+$, then $\mu=M\,\delta_{c_-}$ and $v(0,x)=\varphi^-_{c_-}\big(x-c_-\ln M\big)$ for all~$x\in\R$, hence the desired conclusion~(\ref{vtildeT}) is immediate since $\varphi^-_{c_-}(\xi)\sim e^{-\lambda_{c_-}^-\xi}=e^{-\lambda_{\tilde{c}_+}^-\xi}$ as~$\xi\to+\infty$.\par
Let us now consider in the sequel the case $c_-<\tilde{c}_+$. We first show that, for every $c'\in(c_-,\tilde{c}_+)$,
\be\label{vtilde4}
v(0,x)\sim M^{-1}\int_{[c',\tilde{c}_+]}e^{-\lambda^-_c(x-c\ln M)}\,d\mu(c)\
\hbox{ as }x\to+\infty.
\ee
To do so, let $c'$ be arbitrary in the open interval $(c_-,\tilde{c}_+)$ and let
$c''$ be such that $c_-<c'<c''<\tilde{c}_+$. It follows from~(\ref{vtildeTbis})
that
\be\label{vtildeT3}
\displaystyle\underbrace{M^{-1}\!\!\!\int_{[c',\tilde{c}_+]}\!\!\!\!\!\varphi^-_c\big(x\!-\!c\ln M\big)d\mu(c)}_{=:I_1(x)}\!\le\!v(0,x)\!\le\!\displaystyle\underbrace{M^{-1}\!\!\!\int_{[c_-,c')}\!\!\!\!\!e^{-\lambda^-_c(x-c\ln M)}d\mu(c)}_{=:I_2(x)}+\underbrace{M^{-1}\!\!\!\int_{[c',\tilde{c}_+]}\!\!\!\!\!e^{-\lambda^-_c(x-c\ln M)}d\mu(c)}_{=:I_3(x)}
\ee
for all $x\in\R$. We shall show that $v(0,x)\sim I_3(x)$ as $x\to+\infty$.
Since $\lambda^-_c\ge\lambda_{c'}^->\lambda^-_{c''}>0$ for all $c\in[c_-,c')$,
one gets from one hand that
$I_2(x)=O(e^{-\lambda^-_{c'}x})=o(e^{-\lambda^-_{c''}x})$ as $x\to+\infty$,
while, from the other, that
$$\displaystyle\liminf_{x\to+\infty}\big(e^{\lambda^-_{c''}x}
I_3(x)\big)\ge\displaystyle\liminf_{x\to+\infty}\left(e^{\lambda^-_{c''}x}M^{-1}
\!\!\int_{[c'',\tilde{c}_+]}\!\!\!\!e^{-\lambda^-_c(x-c\ln
M)}d\mu(c)\right)\!\ge\displaystyle M^{-1}\!\!
\int_{[c'',\tilde{c}_+]}\!\!\!\!e^{\lambda^-_cc\ln M}d\mu(c),$$
which is positive because $\mu\big([c'',\tilde{c}_+]\big)>0$ by definition of
$\tilde{c}_+$. As a consequence,~$I_2(x)=o(I_3(x))$ as $x\to+\infty$ and
\be\label{I23}
I_2(x)+I_3(x)\sim I_3(x)\ \hbox{ as }x\to+\infty.
\ee
As far as the left-hand side of~(\ref{vtildeT3}) is concerned, we claim that
\be\label{claimeta}
\forall\,\eta>0,\ \exists\,A_{\eta}\in\R,\ \forall\,x\ge A_{\eta},\ \forall\,c\in[c',\tilde{c}_+],\ \ \varphi_c^-(x)\ge(1-\eta)\,e^{-\lambda_c^-x}.
\ee
Let us assume temporarily this claim and finish the proof of
Lemma~\ref{lemmev}.
It follows from~\eqref{claimeta} that, for every $\eta>0$, there holds
$I_1(x)\ge(1-\eta)\,I_3(x)$ for $x$ large enough. Together with~(\ref{vtildeT3})
and~(\ref{I23}), one infers that $v(0,x)\sim I_3(x)$ as $x\to+\infty$, that
is~(\ref{vtilde4}).\par
In the final step, let $x_0\in\R$ be fixed. We want to show that $v(0,x+x_0)/v(0,x)\to e^{-\lambda^-_{\tilde{c}_+}x_0}$ as $x\to+\infty$. Let $\epsilon\in(0,1)$ be arbitrary. Since the map $c\mapsto\lambda^-_c$ is continuous on $[2\sqrt{\mu_-},+\infty)$, it follows that there is~$c'\in(c_-,\tilde{c}_+)$ such that
$(1-\epsilon)\,e^{-\lambda^-_{\tilde{c}_+}x_0}\le e^{-\lambda^-_cx_0}\le(1+\epsilon)\,e^{-\lambda^-_{\tilde{c}_+}x_0}$ for all $c\in[c',\tilde{c}_+]$.
From~(\ref{vtilde4}), there is~$A\in\R$ such that, for all $x\ge A$,
$$\displaystyle (1-\epsilon)\,M^{-1}\!\!\int_{[c',\tilde{c}_+]}\!\!e^{-\lambda^-_c(x-c\ln M)}d\mu(c)\le v(0,x)\le\displaystyle (1+\epsilon)\,M^{-1}\!\!\int_{[c',\tilde{c}_+]}\!\!e^{-\lambda^-_c(x-c\ln M)}d\mu(c)$$
and
$$\displaystyle (1-\epsilon)\,M^{-1}\!\!\int_{[c',\tilde{c}_+]}\!\!
e^{-\lambda^-_c(x+x_0-c\ln M)}d\mu(c)\le v(0,x+x_0)\le\displaystyle
(1+\epsilon)\,M^{-1}\!\!\int_{[c',\tilde{c}_+]}\!\!e^{-\lambda^-_c(x+x_0-c\ln
M)}d\mu(c).$$
Putting together the previous three displayed formulas leads to
$$\frac{(1-\epsilon)^2}{1+\epsilon}\,e^{-\lambda^-_{\tilde{c}_+}x_0}\,v(0,x)\le v(0,x+x_0)\le\frac{(1+\epsilon)^2}{1-\epsilon}\,e^{-\lambda^-_{\tilde{c}_+}x_0}\,v(0,x)\ \hbox{ for all }x\ge A.$$
Since $\epsilon\in(0,1)$ was arbitrary, the desired conclusion~(\ref{vtildeT}) follows and the proof of Lemma~\ref{lemmev} is complete.\hfill$\Box$\break

\noindent{\bf{Proof of~(\ref{claimeta}).}} More generally, we fix any two real numbers $a$ and $b$ such that~$2\sqrt{\mu_-}<a\le b$. We recall that, for every $c\in(2\sqrt{\mu_-},+\infty)$, $\varphi_c^-(x)\sim e^{-\lambda^-_cx}$ as $x\to+\infty$, where $\lambda^-_c=(c-\sqrt{c^2-4\mu_-})/2$ solves~$(\lambda^-_c)^2-c\lambda^-_c+\mu_-=0$, and we want to show that, for every~$\eta>0$, there exists a real number~$A_{\eta}$ such that
\be\label{Aeta}
\varphi^-_c(x)\ge(1-\eta)\,e^{-\lambda^-_cx}\hbox{ for all }c\in[a,b]\hbox{ and }x\ge A_{\eta}.
\ee
In other words, we want to show that the asymptotic exponential decay of $\varphi^-_c$ at $+\infty$ is uniform with respect to $c\in[a,b]$.\par
To do so, notice first that, for every $c\in[a,b]$,
$0<\lambda_b^-\le\lambda^-_c\le\lambda^-_a<\sqrt{\mu_-}<\tilde{\lambda}_c:=(c+\sqrt{c^2-4\mu_-})/2$.
Since, for every $c\in[a,b]$, $\lambda^-_c<\tilde{\lambda}_c$ are the two roots of the equation $X^2-cX+\mu_-=0$, there are $\alpha\in(0,1)$ and $\beta>0$ such that $\lambda_c^-<(1+\alpha)\,\lambda^-_c<\tilde{\lambda}_c$ for all $c\in[a,b]$ and
\be\label{alphabeta}
(1+\alpha)^2\,(\lambda^-_c)^2-c\,(1+\alpha)\,\lambda^-_c+\mu_-\le-\beta\ \hbox{ for all }c\in[a,b].
\ee
Now, since $f_-$ is of class $C^2([0,1])$ with $f_-(0)=0$ and $f_-'(0)=\mu_-$,
there is $s_0\in(0,1)$ such that~$f_-(s)\ge\mu_-s-s^{1+\alpha}$ for all
$s\in[0,s_0]$. Next, it is straightforward to check that there is a positive
real number $B$ such that
\be\label{B}
\beta\,B\ge1\ \hbox{ and }\ \underline{u}_c(x):=\max\big(e^{-\lambda^-_cx}-B\,e^{-(1+\alpha)\lambda^-_cx},0\big)\le s_0\hbox{ for all }c\in[a,b]\hbox{ and }x\in\R.
\ee
For all $c\in[a,b]$ and $x\in\R$ such that $\underline{u}_c(x)>0$, there holds
$0\!<\!\underline{u}_c(x)\!=\!e^{-\lambda^-_cx}\!-\!Be^{-(1\!+\!\alpha)\lambda^-_cx}\!\le\!\min(s_0,e^{-\lambda^-_cx})$,
whence
$$\baa{rcl}
\underline{u}_c''(x)+c\,\underline{u}_c'(x)+f_-(\underline{u}_c(x)) & \!\!=\!\! & -\mu_-e^{-\lambda^-_cx}-B\big((1+\alpha)^2(\lambda^-_c)^2-c(1+\alpha)\lambda^-_c\big)e^{-(1+\alpha)\lambda^-_cx}+f_-(\underline{u}_c(x))\vspace{3pt}\\
& \!\!\ge\!\! &
-\mu_-e^{-\lambda^-_cx}+B\big(\mu_-+\beta\big)e^{-(1+\alpha)\lambda^-_cx}
+f_-(\underline{u}_c(x))\\
& \!\!\ge\!\! &
\!-\mu_-\underline{u}_c(x)+\underline{u}_c^{1+\alpha}(x)+f_-(\underline{u}
_c(x))\vspace{3pt}\\
& \!\!\ge\!\! & 0\eaa$$
from~(\ref{alphabeta}) and~(\ref{B}). In other words, since $f_-(0)=0$, the functions $\underline{u}_c$ are subsolutions of the equations~$(\varphi^-_c)''+c(\varphi^-_c)'+f_-(\varphi_c)=0$ satisfied by the functions $\varphi^-_c$.\par
In this paragraph, $c$ is a fixed real number in $[a,b]$. We want to show that
$\underline{u}_c(x)\le\varphi^-_c(x)$ for all $x\in\R$. Notice that both
functions $\underline{u}_c$ and $\varphi^-_c$ have the same exponential decay,
namely~$e^{-\lambda^-_cx}$, as~$x\to+\infty$, but one cannot directly apply the
maximum principle as~$x\to+\infty$ since~$f'_-(0)=\mu_->0$.
However, we are going to use a sliding method. Remember
that~$\underline{u}_c(x)\le\min\big(e^{-\lambda^-_cx},s_0\big)$ for all
$x\in\R$, that $\varphi_c^-(x)\sim e^{-\lambda^-_cx}$ as $x\to+\infty$, and that
the positive continuous function~$\varphi^-_c$ converges to $1\,(>s_0)$ at
$-\infty$. Therefore, there is~$x_0>0$ such that
$\underline{u}_c(x)\le\varphi^-_c(x-x_0)$ for all $x\in\R$. Define
$$x_*=\min\big\{x'\ge0,\ \underline{u}_c(x)\le\varphi^-_c(x-x')\hbox{ for all }x\in\R\big\}.$$
The real number $x_*$ is well defined, with $0\le x_*\le x_0$, and $\underline{u}_c(x)\le\varphi^-_c(x-x_*)$ for all $x\in\R$. Assume, by contradiction, that $x_*>0$. Then there are some sequences $(x_p)_{p\in\N}$ in $(0,x_*)$ and~$(y_p)_{p\in\N}$ in $\R$ such that
$x_p\to x_*$ as $p\to+\infty$ and $\underline{u}_c(y_p)>\varphi^-_c(y_p-x_p)$ for all $p\in\N$.
Since $x_*>0$ and $\underline{u}_c(x)\le e^{-\lambda^-_cx}\sim\varphi^-_c(x)$
as
$x\to+\infty$, one infers that $\limsup_{p\to+\infty}y_p<+\infty$. Furthermore,
$\liminf_{p\to+\infty}y_p>-\infty$ since
$\underline{u}_c(-\infty)=0<1=\varphi^-_c(-\infty)$. Therefore, the sequence
$(y_p)_{p\in\N}$ is bounded and converges, up to extraction of a subsequence, to
a real number~$y_{\infty}$. It follows that
$\underline{u}_c(y_{\infty})\ge\varphi^-_c(y_{\infty}-x_*)$, while
$\underline{u}_c(x)\le\varphi^-_c(x-x_*)$ for all $x\in\R$ (hence,
$\underline{u}_c(y_{\infty})=\varphi^-_c(y_{\infty}-x_*)>0$). But
$\underline{u}_c$ is a subsolution of the equation satisfied by~$\varphi^-_c$.
It follows from the strong maximum principle that
$\underline{u}_c(x)=\varphi^-_c(x-x_*)$ for all $x$ belonging to any interval
$(z_1,z_2)$ containing $y_{\infty}$ and where $\underline{u}_c$ is positive. By
definition of~$\underline{u}_c$, there is~$x_-\in(-\infty,y_{\infty})$ such that
$\underline{u}_c(x)>0$ on $(x_-,y_{\infty}]$ with $\underline{u}_c(x_-)=0$. It
follows by continuity that~$\varphi^-_c(x_--x_*)=\underline{u}_c(x_-)=0$, which
is a contradiction since $\varphi^-_c$ is positive in~$\R$.
Consequently,~$x_*=0$ and~$\underline{u}_c(x)\le\varphi^-_c(x)$ for
all~$x\in\R$.\par
Finally, for all $c\in[a,b]$ and $x\in\R$, there holds $\varphi^-_c(x)\ge
e^{-\lambda^-_cx}(1-B\,e^{-\alpha\lambda^-_cx})$, where
$\alpha>0$ and $B$
defined in~(\ref{alphabeta}) and~(\ref{B}) are independent of $c\in[a,b]$.
Since~$\lambda^-_c\ge\lambda^-_b>0$ for all~$c\in[a,b]$,  the
conclusion~(\ref{Aeta}) follows immediately. The proof of the
claim~(\ref{claimeta}) is thereby complete.\hfill$\Box$


\subsection{Proof of Theorem~\ref{thm:decay}}\label{sec33}

Let $f$ and $u$ be as in Theorem~\ref{thm:decay}. In order to show that $u(t,x)$ has a exponential decay rate as~$x\to+\infty$, the method is, as in the proof of Theorem~\ref{th2}, to compare $u$ with a solution $v$ of~\eqref{eqv}. By using~\eqref{hyphn}, it will follow that~$v$ is of the type $v=u_{\mu}$ for some measure $\mu\in\mathcal{M}$. Since such solutions $u_{\mu}$ turn out to have an exponential decay rate $\lambda\ge0$ as $x\to+\infty$, so does~$u$. Lastly, we distinguish the two cases $\lambda>0$ and~$\lambda=0$ and we show in the former that $u$ is then a transition front connecting $0$ and $1$ and we identify its asymptotic past and future speeds by comparison arguments.\par
First of all, observe that Steps~1 and~2 of the proof of Theorem~\ref{th2} can be reproduced word by word, since they did not use the fact that $u$ was a transition front connecting $0$ and $1$. Therefore, by defining the sequence $(v_n)_{n\in\N}$ and the function $v$ as there, there is then a solution $0<v\le 1$ of~\eqref{eqv} such that~\eqref{comparuv} holds, that is,
\be\label{comparuvbis}
u(t,x)\,e^{-\mu_-\Theta(t)}\le v(t,x)\le u(t,x)\,e^{\mu_-\Theta(t)}\ \hbox{ for all }(t,x)\in\R^2,
\ee
where $\Theta$ (and $\zeta$) are as in~\eqref{defTheta} (and~\eqref{hypf-bis}).
Since $\Theta$ is bounded in, say, $(-\infty,0]$, it follows then from the
assumption~\eqref{hyphn} that $\max_{[-c|t|,c|t|]}v(t,\cdot)\to0$ as
$t\to-\infty$ with $c>2\sqrt{\mu_-}=2\sqrt{f'_-(0)}$ (whence, in particular,
$v(t,x)<1$ for all $(t,x)\in\R\times\R$, from the strong maximum principle).
Furthermore, Theorem~1.4 of~\cite{hn2} implies that the solution $v$
of~\eqref{eqv} is then of the type $v=u_{\mu}$ for some
measure~$\mu\in\mathcal{M}$ (associated with the function $f_-$) whose support
satisfies
\be\label{suppmu}
\hbox{supp}(\mu)\,\subset\,(-\infty,-c]\cup[c,+\infty)\cup\{\infty\}.
\ee
As a consequence, one infers from Theorem~1.9 of~\cite{hr} applied to the function $v^{\tau}$ defined in~$\R\times\R$ by~$v^{\tau}(t,x):=v(t+\tau,x)$ for an arbitrary $\tau\in\R$ that $\lambda_{\tau}:=-\lim_{x\to+\infty}(\ln v(\tau,x))/x$ exists in~$[0,\sqrt{\mu_-})$. Furthermore, as explained after the statement of Theorem~Ê1.9 in~\cite{hr}, the real numbers $\lambda_{\tau}$ do not depend on $\tau\in\R$. Finally, by using~\eqref{comparuvbis}, there is $\lambda\in[0,\sqrt{\mu_-})$ such that~$-(\ln u(t,x))/x\to\lambda$ as~$x\to+\infty$ for all~$t\in\R$. We will then consider separately the cases~$\lambda>0$ and $\lambda=0$.\par
{\it Case 1: $\lambda>0$.} In this case, Theorem~1.9 of~\cite{hr} implies that $v=u_{\mu}$ is a transition front connecting $0$ and $1$ for~\eqref{eqv}, with a function $X_v:\R\to\R$ satisfying~\eqref{gtf}. We get then from part~(i) of Theorem~\ref{thsupport} and from~\eqref{suppmu} that the support of $\mu$ is a compact subset of~$[c,+\infty)\,(\subset(2\sqrt{\mu_-},+\infty))$. Furthermore, by part~(ii) of Theorem~\ref{thsupport}, the transition front~$v$ has then an asymptotic past speed $c_-\in[c,+\infty)\subset(2\sqrt{\mu_-},+\infty)$ and an asymptotic future speed~$\tilde{c}_+\in[c_-,+\infty)$. On the other hand, by~\eqref{gtf} and~\eqref{infsup} applied with $X_v$ and $v$, the past speed $c_-$ satisfies
\be\label{c-v}
c_-=\sup\Big\{\gamma\ge0,\ \lim_{t\to-\infty}\max_{[-\gamma|t|,\gamma|t|]}v(t,\cdot)=0\Big\},
\ee
while Lemma~\ref{lemmev} yields
$\lim_{x\to+\infty}v(0,x+x_0)/v(0,x)\!=\!e^{-\lambda^-_{\tilde{c}_+}x_0}$ for all
$x_0\in\R$
with~$\lambda^-_{\tilde{c}_+}\!=\!(\tilde{c}_+-\sqrt{\tilde{c}_+^2-4\mu_-})/2$.
Since one already knows that $\ln v(0,x)\sim-\lambda x$ as $x\to+\infty$,
one gets $\lambda=\lambda^-_{\tilde{c}_+}$, and
$v(0,x+x_0)/v(0,x)\to e^{-\lambda x_0}$ as $x\to+\infty$ for all $x_0\in\R$.\par
Define now $X(t)=X_v(t)$ for all $t\le0$. One has $\lim_{t\to-\infty}X(t)/t=c_-\ge c>2\sqrt{\mu_-}$ and, by~\eqref{comparuvbis} and~\eqref{c-v}, there holds
$c_-=\sup\big\{\gamma\ge0,\ \lim_{t\to-\infty}\max_{[-\gamma|t|,\gamma|t|]}u(t,\cdot)=0\big\}$.
By using again~\eqref{comparuvbis} and by interchanging the roles of $u$ and $v$ in the proof of Lemma~\ref{lemvfront}, it also follows that
\be\label{utf1}\left\{\baa{ll}
u(t,X(t)+x)\to 1 & \hbox{as }x\to-\infty\vspace{3pt}\\
u(t,X(t)+x)\to 0 & \hbox{as }x\to+\infty\eaa\right.\hbox{ uniformly in }t\in(-\infty,0].
\ee\par
As in Substep~5.4 of the proof of Theorem~\ref{th2}, one can then define $u^+(t,x)$ and $v^+(t,x)$ for~$(t,x)\in[0,+\infty)\times\R$ as in~\eqref{defu+} and~\eqref{defv+} and notice that, as in~\eqref{ubeta}, there is $\beta\in(0,1]$ such that
\be\label{ubetabis}
0<\beta\,v^+(t,x)\le u(t,x)\le\min\big(\beta^{-1}v^+(t,x),1\big)\ \hbox{ for all }(t,x)\in[0,+\infty)\times\R.
\ee
The function $v^+(0,\cdot)=v(0,\cdot)=u_{\mu}(0,\cdot)$ is continuously
decreasing in $\R$ with $v^+(0,-\infty)=1$, $v^+(0,+\infty)=0$ and we know that
$\lim_{x\to+\infty}v^+(0,x+x_0)/v^+(0,x)=e^{-\lambda x_0}$ for all $x_0\in\R$.
Therefore, by
defining~$Y^+(t)$ for~$t\ge0$ as in~\eqref{defY+}, we get from~\cite{u} as in
Substep~5.4 of the proof of Theorem~\ref{th2} that $Y^+(t)/t\to c_+$ as~$t\to+\infty$
and that~\eqref{convphi+2} holds, that is,
\be\label{v+Y+bis}
\sup_{x\in\R}\big|v^+\big(t,Y^+(t)+x\big)-\varphi_{c_+}^+(\xi_0+x)\big|\to0\ \hbox{ as }t\to+\infty,
\ee
where $\xi_0=(\varphi^+_{c_+})^{-1}(1/2)$ and $\varphi^+_{c_+}(x-c_+t)$ is a standard traveling front connecting $0$ and $1$ for~\eqref{eqf+}, with speed
$$c_+=\min(\lambda,\sqrt{\mu_+})+\frac{\mu_+}{\min(\lambda,\sqrt{\mu_+})}.$$
\par
Set $X(t)=Y^+(t)$ for $t>0$. Let us finally prove that $u$ is a transition front connecting $0$ and~$1$ for~\eq{P=f} with the function $X:\R\to\R$ (remember that we already know that~\eqref{utf1} holds). Let $\epsilon>0$ be arbitrary. By~\eqref{ubetabis} and~\eqref{v+Y+bis}, there are $T_{\epsilon}>0$ and $M_\e>0$ such that
\be\label{uXte1}
0<u(t,X(t)+x)\le\epsilon\hbox{ for all }t\ge T_\e\hbox{ and }x\ge M_\e.
\ee
Using again~\eqref{ubetabis} and~\eqref{v+Y+bis} and arguing as in the proof of Lemma~\ref{lemvfront}, one gets the existence of $T'_\e>0$ and $M'_\e>0$ such that
\be\label{uXte2}
1-\e\le u(t,X(t)+x)<1\hbox{ for all }t\ge T'_\e\hbox{ and }x\le-M'_\e.
\ee
Since $u(0,-\infty)=1$ and $u(0,+\infty)=1$ as a particular consequence
of~\eqref{utf1}, standard parabolic estimates imply that $u(t,-\infty)=1$ and
$u(t,+\infty)=0$ locally uniformly in $t\in\R$. On the other
hand,~$Y^+(t)\,(=X(t)$ for $t>0$) is locally bounded in $[0,+\infty)$ since
$v^+(t,-\infty)=1$ and~$v^+(t,+\infty)=0$ locally uniformly in~$t\ge0$. These
properties together with~\eqref{uXte1} and~\eqref{uXte2} imply that
$u(t,X(t)+x)\to1$ as~$x\to-\infty$ and~$u(t,X(t)+x)\to0$ as $x\to+\infty$,
uniformly in $t\ge0$. Thanks to~\eqref{utf1}, one infers that $u$ is a
transition front connecting $0$ and $1$ for~\eq{P=f}, with the function
$X:\R\to\R$ satisfying~\eqref{gtf}. It also follows from the previous properties
of $X_v$ and $Y^+$ that $X(t)/t\to c_{\pm}$ as $t\to\pm\infty$, where $c_{\pm}$
are given as in~\eqref{asympspeeds}. Finally, by Theorem~\ref{th2},~\eq{phi+-}
holds with $\phi_{c_{\pm}}=\varphi^{\pm}_{c_{\pm}}$, for some bounded function
$\xi:\R\to\R$. The proof is therefore complete in the
case $\lambda>0$.\par
{\it Case 2: $\lambda=0$.}
We know in this case that  $(\ln u(\tau,x))/x\to0$ as
$x\to+\infty$ for every $\tau\in\R$.
Suppose by contradiction that $u$ is a transition front connecting $0$ and $1$
for~\eq{P=f}. Take $\tau$ large enough so that~$f(t,s)\ge f_+(s)/2$ for all
$(t,s)\in[\tau,+\infty)\times[0,1]$. For any
$\gamma>\sqrt{2\mu_+}=\sqrt{2f_+'(0)}$, consider the standard traveling front
$\t\vp_\gamma(x-\gamma t)$ connecting $0$ and $1$ for the equation
\eqref{eqg} with $g=f_+/2$. Since $\t\vp_\gamma$ decays exponentially at
$+\infty$, therefore faster than $u(\tau,\.)$, and $u(\tau,\.)$ is
positive, continuous and tends to $1$ at $-\infty$, we can find $\beta\in(0,1)$
such that $\beta\t\vp_\gamma(x-\gamma\tau)\leq u(\tau,x)$ for all $x\in\R$.
Furthermore, $\beta\t\vp_\gamma(x-\gamma t)$ is a subsolution to \eqref{eq:P=f}
for $t>\tau$. It then follows from the maximum principle that
$u(t,\gamma t)\ge\beta\t\vp(0)>0$ for
all~$t>\tau$, from which, owing to~(\ref{gtf}) or~(\ref{infsup}), we derive
$\liminf_{t\to+\infty}X(t)/t\ge\gamma$. Since this is true for all~$\gamma>\sqrt{2\mu_+}$,
we obtain a contradiction
with~(\ref{boundedspeed}). Therefore, $u$ is not a transition front connecting
$0$ and $1$ and the proof of Theorem~\ref{thm:decay} is thereby
complete.\hfill$\Box$


\end{document}